\documentclass[11pt,reqno,english,empty]{amsart}
\usepackage{array}
 \usepackage[hypertex,
 colorlinks=true,
 linkcolor=blue,
 citecolor=red,
 ]{hyperref}
\usepackage[T1]{fontenc}

\usepackage{amsmath,amsthm,amssymb}
\usepackage{mathrsfs}
\usepackage{amsmath,amsthm}
\usepackage{latexsym}

\usepackage{enumerate}
\usepackage{mathrsfs}
\usepackage{stmaryrd}
\usepackage{amsopn}
\usepackage{amsmath}
\usepackage{amsfonts}
\usepackage{amsbsy}
\usepackage{amscd,indentfirst,epsfig}
\usepackage{hyperref}

\usepackage{amsthm,psfrag}
\usepackage{dsfont}
\usepackage{colordvi}
\usepackage{pstricks}

\usepackage{pgf}
\usepackage{tikz}
\usepackage{graphicx}
\usepackage{hyperref,multimedia}

\usepackage{bbding}
\usepackage{pifont}
\usepackage{enumerate}
\usepackage{relsize}
\renewcommand{\theequation}{\arabic{section}.\arabic{equation}}

\newtheorem{theo}{Theorem}[section]
\newtheorem{lemme}[theo]{Lemma}
\newtheorem{lemmeA}{Lemma A.}
\newtheorem{nbA}{Remark A.}
\newtheorem{exaA}{Example A.}
\newtheorem{theoA}{Theorem A.}
\newtheorem{corA}{Corollary A.}
\newtheorem{propoA}{Proposition A.}

\newtheorem{theoB}{Theorem B.}
\newtheorem{lemmeB}{Lemma B.}
\newtheorem{nbB}{Remark B.}
\newtheorem{defiA}{Definition A.}

\newtheorem{propoB}{Proposition B.}
\newtheorem{propo}[theo]{Proposition}
\newtheorem{cor}[theo]{Corollary}
\newtheorem{hyp}[theo]{Assumptions}
\newtheorem{nb}[theo]{Remark}
\newtheorem{defi}[theo]{Definition}

\theoremstyle{definition}

\def \leq {\leqslant}
\def \geq {\geqslant}
\numberwithin{equation}{section}
\def\ind#1{\lower5pt\hbox{$\scriptstyle #1$}}
\def \le {\leqslant}
\def \ge {\geqslant}

\def \l {\ell}
\def \d {\, \mathrm{d} }

\def \fe {F_\la}

\def \gl {G_\la}
\def \la {\lambda}

\def \ds {\displaystyle}
\def \l {\ell}

\def \ds {\displaystyle}
\def\Q {\mathcal{Q}}

\def\R{{\mathbb R}}
\def \S {{\mathbb S}^2}
\def \E {\mathcal{E}}
\def \n {\widehat{n}}

\def \v {{v}}

\def \vb {\v_{\star}}

\def \IS {\int_{\S}}
\def \IR {\int_{\R^3}}
\def \IRR {\int_{\R^3 \times \R^3}}

\def \us {\widehat{u} \cdot \sigma}

\def \el {e_\la}
\def \m {\mathbf{m}}

\def \RF {\mathcal{RF}}
\def \F {\mathcal{F}}
\def \H {\mathcal{H}}

\begin{document}
\title[]
{\textbf{Uniqueness and regularity of steady states of the Boltzmann equation for viscoelastic hard-spheres driven by a thermal bath}}

\author{R. J. Alonso \& B. Lods}

\address{\textbf{Ricardo J. Alonso}, Dept. of Computational \& Applied Mathematics, Rice University
Houston, TX 77005-1892, USA.}
\email{Ricardo.J.Alonso@rice.edu}

\address{\textbf{Bertrand Lods},  Dipartimento di Statistica e Matematica Applicata \& Collegio Carlo Alberto, Universit\`{a} degli
Studi di Torino,  Corso Unione Sovietica, 218/bis, 10134 Torino, Italy.}\email{lods@econ.unito.it}

\thanks{The work of R. Alonso was partially supported by the Office of Naval Research, grant N00014-09-1-0290 and by the National Science Foundation Supplemental Funding DMS-0439872 to UCLA-IPAM}

\hyphenation{bounda-ry rea-so-na-ble be-ha-vior pro-per-ties
cha-rac-te-ris-tic  coer-ci-vity}

\maketitle

\begin{abstract}
We study the uniqueness and regularity of the steady states of the diffusively driven Boltzmann equation in the physically relevant case where the restitution coefficient depends on the impact velocity including, in particular, the case of viscoelastic hard-spheres. We adopt a strategy which is novel in several aspects, in particular, the study of regularity does not requires \textit{a priori} knowledge of the time-dependent problem.  Furthermore, the uniqueness result is obtained in the small thermalization regime by studying the so-called \textit{quasi-elastic limit} for the problem. An important new aspect lies in the fact that no entropy functional inequality is needed in the limiting process.
\end{abstract}

\medskip
\section{Introduction}
\label{intro}
\setcounter{equation}{0}

\subsection{General setting} We investigate in the present paper the properties of the steady states of the spatially
homogeneous diffusively driven inelastic Boltzmann equation for hard spheres interactions and non-constant restitution coefficient. More precisely, we consider inelastic hard-spheres particles described by their  distribution density
$F=F(v) \geq 0$, $v \in \R^3$ and we consider the case in which $F$ satisfies
the stationary equation
\begin{equation}\label{steady}
\Q_e(F,F)+\mu \Delta F=0
\end{equation}
for some positive thermalization (or diffusion) coefficient $\mu >0$.  Moreover, assume $F$ has a given mass $\varrho >0$  and vanishing momentum:
$$\IR F(v)\d v=\varrho, \qquad \IR F(v)v\d v=0.$$
The diffusion operator $\mu\Delta_v F(v)$ appearing in \eqref{steady} represents a constant heat bath which models particles uncorrelated random accelerations between collisions.  The quadratic collision operator $\Q_e(F,F)$ models the interactions of hard-spheres by inelastic  binary collisions where the inelasticity is characterized by the so-called normal restitution coefficient $e(\cdot)$ that we shall assume here, in contrast with previous contributions on the subject, to be \textit{non-constant}. This restitution coefficient  quantifies the loss of relative
normal velocity of a pair of colliding particles after the
collision with respect to the impact velocity. Namely, if $v$ and $\vb$  denote the velocities of two particles before collision, their respective velocities $v'$ and $\vb'$ after collision are such that
\begin{equation}\label{coef}
(u'\cdot \n)=-(u\cdot \n) \,e(|u \cdot \n|),
\end{equation}
where the restitution coefficient $e:=e(|u \cdot \n|)$ is such that
$0 \leq e \leq 1$.  The unitary vector $\n \in \mathbb{S}^2$  determines the impact direction, that is, $\n$ stands for the unit vector that points from the $v$-particle center to the $\vb$-particle center at the instant of impact.  Here above
$$u=v-\vb,\qquad u'=v'-\vb',$$
denote respectively the relative velocity before and after collision.  Assuming  the granular particles to be perfectly smooth hard-spheres of mass
$m=1$, the velocities after collision $v'$ and $\vb'$ are given, in virtue of \eqref{coef} and the conservation of momentum, by
\begin{equation}
\label{transfpre}
  v'=v-\frac{1+e}{2}\,(u\cdot \n)\n,
\qquad \vb'=\vb+\frac{1+e}{2}\,(u\cdot \n)\n.
\end{equation}  The main assumption on $e(\cdot)$ we shall need for our analysis is listed in the following, see \cite{AlonsoIumj}.
\begin{hyp}\label{HYP}
Throughout the paper, one assumes the following to hold:
\begin{enumerate}
\item The mapping  $r \in \mathbb{R}_+ \mapsto e(r) \in (0,1]$ is absolutely continuous and non-increasing.
\item The mapping $r\in\mathbb{R}^{+}\mapsto \vartheta_e(r):=r\;e(r )$ is strictly increasing.
\item There exist $\mathfrak{a} >0$ and $\gamma \geq 0$ such that
\begin{equation}\label{gamma}
e(r) \simeq 1 - \mathfrak{a}\,r^\gamma \quad \text{ as } \quad r \simeq 0.\end{equation}
\end{enumerate}
\end{hyp}
The assumption that $e(\cdot)$ is non-increasing can be relaxed and replaced by the more general Assumptions 3.1 in \cite{AloLo1} (notice that, if $e(\cdot)$ is non-increasing, it is proven in \cite[Appendix A]{AloLo1} that such Assumptions 3.1 are indeed satisfied). In several places in our analysis, we shall need slightly stronger assumptions on the restitution coefficient that will be properly stated when needed. When no supplementary assumption is specified means that the stated result is true under the sole Assumptions \ref{HYP}. Notice that all these assumptions will be met by the \textit{visco-elastic hard-spheres} model which is the most physically relevant model for applications \cite{BrPo}. For such a model, the properties of the restitution coefficient have been derived in \cite{BrPo,PoSc}; in particular, $e(r)$ can be defined explicitly by the following series
\begin{equation}\label{visc}e(r)=1+ \sum_{k=1}^\infty (-1)^k a_k r^{\frac{k}{5}}, \qquad r\geq 0\end{equation}
where $a_k >0$ for any $k \in  \mathbb{N}$ are parameters depending on the material viscosity. In such a case, Assumptions \ref{HYP} are met with $\gamma=\frac1 5$ and $\mathfrak{a}=a_1$.
In the sequel, it shall be more convenient to deal with a second, and equivalent, parametrization of the post-collisional velocities.  Fix $v$ and $\vb$ with $v\neq \vb$ and let $\widehat{u}={u}/{|u|}$. Performing in \eqref{transfpre} the change of unknown
$\sigma=\widehat{u}-2 \,(\widehat{u}\cdot \n)\n \in\mathbb{S}^2$ provides an alternative parametrization of the unit sphere $\mathbb{S}^2$.  In this case, the impact velocity reads $|u\cdot\n|=|u| \sqrt{\tfrac{1-\widehat{u} \cdot \sigma}{2}}$ and the post-collisional velocities $v'$ and $\vb'$ are then given by
\begin{equation}\label{postsig} v'=v-\beta\frac{u-|u|\sigma}{2}, \qquad \vb'=\vb+ \beta\frac{u-|u|\sigma}{2}\end{equation}
where $\beta=\frac{1+e}{2}=\beta\left(|u| \sqrt{\tfrac{1-\widehat{u} \cdot \sigma}{2}}\right)  \in \left(\tfrac{1}{2},1\right].$
This representation allows us to give a precise definition of the Boltzmann collision operator in \textit{weak form} by
\begin{multline}\label{Ie3BHS}
\int_{\mathbb{R}^{3}}\Q_{e}(f,f)(v)\psi(v)\d v=\\
\frac{1}{2}\int_{\mathbb{R}^{3} \times \mathbb{R}^3 \times \mathbb{S}^{2}}f(v)f(\vb)\bigg(\psi( {v'})+\psi( {\vb'})-\psi(v)-\psi(\vb)\bigg)\d\sigma\d\vb\d v
\end{multline}
for any test function $\psi=\psi(v)$.  Here, the post-collisional velocities $v'$ and $\vb'$ are defined by \eqref{postsig}.  Notice that
$$|v'|^2+|\vb'|^2-|v|^2-|\vb|^2=-|u|^2\dfrac{1-\us}{4}\left(1- e\left(|u|\sqrt{\frac{1-\us}{2}}\right)^2\right),$$
thus, it follows that (see \cite{AloLo1} for details)
\begin{equation}\label{PhiPsi}\IR \Q_{e}(f,f)(v)|v|^2\d v=-\IRR f( v)f( \vb) \mathbf{\Psi}_{e}(|u|^2)\d v\d\vb \leq 0\end{equation}
where  the \textit{energy dissipation potential} $\mathbf{\Psi}_{e}$ is given by
\begin{equation}\label{Psie}
\mathbf{\Psi}_{e}(r):=\frac{r^{3/2}}{2}\int_0^{1} \left(1-e(\sqrt{r}z)^2\right)z^3\,\d z , \qquad \qquad \forall r >0.
\end{equation}
Notice that, under Assumptions \ref{HYP}, the mapping $\mathbf{\Psi}_e(\cdot)$ is \textit{convex and non-decreasing} (see again \cite{AloLo1}). The functional
\begin{equation}\label{dissIeF}\mathcal{I}_e(f):=\IRR f( v)f( \vb) \mathbf{\Psi}_{e}(|u|^2)\d v\d\vb\end{equation}
can be seen as an energy dissipation functional for the operator $\Q_e$. In particular, multiplying \eqref{steady} by $|v|^2$, one sees that
$$\mathcal{I}_e(F)=6\mu\varrho$$
for any solution $F$ to \eqref{steady} with mass $\varrho.$

Stationary solutions for equation \eqref{steady} in the case of \textit{constant} restitution coefficient have been studied from the mathematical viewpoint by different authors.  Existence of such   solutions was shown in \cite{GaPaVi}.  The study of moments and tails was described in \cite{BoGaPa}. Uniqueness and stability of these steady states (in the elastic limit)  was presented in \cite{MiMo3}.  Different kinds of forcing terms have also been considered in the literature. In particular, for the inelastic Boltzmann equation in \textit{self-similar variables} (corresponding to an anti-drift forcing term), stationary solutions correspond to the so-called \textit{homogeneous cooling state} and uniqueness, study of the elastic limit and convergence to self-similarity  were presented in \cite{MiMo2}.  Uniqueness of steady states for the Boltzmann equation under the thermalization induced by a host medium with a fixed Maxwellian
  distribution was recently presented in \cite{canizo}.  We mention that the case of dissipative Maxwell molecules has been studied as well in \cite{Bobylev-Carrillo-Gamba} and \cite{ccc}.

Regarding the existence of stationary  states, it is not very difficult to extend the results given in \cite{GaPaVi} to non-constant restitution coefficient $e(\cdot)$ and obtain the existence of a steady solutions for the diffusively driven Boltzmann equation \eqref{steady}.

\begin{theo}\label{theo:exists} Assume that the restitution coefficient $e(\cdot)$ satisfies Assumption \ref{HYP}. Then, for any $\mu >0$ and any $\varrho >0$, there exists a nonnegative $F=F(v) \in L^1_2(\R^3) \cap L^2(\R^3)$ such that
$$\Q_e(F,F)+\mu \Delta F=0$$
with $\ds\IR F(v)\d v=\varrho$ and $\ds \IR vF(v)\d v=0$.
\end{theo}
The proof of this theorem follows the path given for constant inelasticity parameter in \cite{GaPaVi} and can be deduced from the properties of the solution to the Cauchy problem associated to \eqref{steady}. We refer to Appendix B for the main steps of the proof of the Theorem \ref{theo:exists}.

\subsection{Scaling argument and formal limit $\la \to 0$}

Let us discuss the main concern of the present work, namely, proving the uniqueness of solutions to \eqref{steady}
$$\Q_e(F,F)+\mu \Delta F=0$$
in the weak thermalization regime, i.e. when the diffusion parameter $\mu$ is sufficiently small.  In order to understand this regime and the strategy fix $\mu >0$ and denote by $F$ a solution to \eqref{steady} with given mass $\varrho$ and vanishing momentum.  Introduce the following rescaled solution
\begin{equation}
\label{GL}\gl (v)=\la^3 F(\la v), \qquad \la >0
\end{equation}
and the rescaled restitution coefficient
$$\el(r)=e(\la r) \qquad \forall r >0.$$
Since
$$\la^2 \Q_e(F,F)(\la v)=\Q_{\el}(\gl,\gl)(v)\qquad \text{and}\qquad \la^5 \Delta_v F(\la v)=\Delta_v \gl(v)$$  for any $v \in \R^3$, one gets that $\gl$ is a solution to the rescaled stationary problem
\begin{equation}
\label{rescaledmu}\Q_{\el}(\gl,\gl)=-\dfrac{\mu}{\la^3}\Delta_v \gl.
\end{equation}
In other words, for any $\la >0$ the rescaled distribution $\gl$ is a solution to the steady diffusively driven Boltzmann equation  with thermalization coefficient ${\mu}/\la^{3}$ and restitution coefficient $\el$ (notice that $\el$ still satisfies Assumptions \ref{HYP}). For any $\la >0$, the solution to \eqref{steady} is unique if and only if the solution to \eqref{rescaledmu} is unique. Such a scaling is particularly interesting because, in addition to preserve mass and momentum, when $\la \to 0$ the rescaled restitution coefficient $\el(r)$ converges pointwise to the elastic restitution coefficient $\lim_{\la \to 0}\el(r)=1$ for any $r\geq 0.$  Consequently, one formally expects that
$$\Q_{\el}(f,f) \simeq \Q_1(f,f) \qquad \text{ as } \quad \la \to 0$$
where $\Q_1(f,f)$ denotes the classical Boltzmann operator for elastic interactions. This means that the dissipation of energy is expected to vanish as $\la \to 0$.  Formally, one sees that if $\mu >0$ is kept fixed the right side of \eqref{rescaledmu} will be infinite in the limit $\la \to 0$.  In other words, the thermalization $\mu$ has to be reduced to compensate the loss of dissipation, i.e. one has to choose a diffusion coefficient $\mu=\mu_\la$ depending on $\la$ such that $\lim_{\la \to 0}\mu_\la=0.$ 
Intuitively, it make sense to look for a parameter that keeps the solution's energy $$\E_\lambda=\dfrac{1}{\varrho} \IR \gl(v)|v|^2\d v$$
of order one in the limit $\la \to 0.$  Let us investigate the correct scaling by multiplying \eqref{rescaledmu} by $|v|^2$ and integrating over $\R^3$ to get
$$6\mu_\la \varrho=\mathcal{I}_{\el}(\gl)$$
where $\mathcal{I}_{\el}$ is the energy dissipation functional associated to the rescaled restitution coefficient $\el$ given by
\begin{equation}\label{scaling}
\mathcal{I}_{\el}(f)= \IRR f(v)f(\vb)\mathbf{\Psi}_e(\la^2|u|^2)\d v\d\vb.\end{equation}
Since $\mathbf{\Psi}_e$ is convex, a simple use of Jensen's inequality yields $6\mu_\la \varrho \geq \varrho^2 \mathbf{\Psi}_{e}\left(\lambda^2\,\E_\lambda \right).$ Moreover,
\begin{equation}
\label{psie0}
\mathbf{\Psi}_e(r)\simeq \dfrac{\mathfrak{a}r^{\frac{3+\gamma}{2}}}{4+\gamma}  \quad \text{as} \quad r \simeq 0,
\end{equation}
from which it follows that $\lambda^2 \E_\lambda=\mathrm{O} (\mu_\lambda^{\frac{2}{3+\gamma}} )$   as $\lambda \simeq 0.$
Consequently, to keep the kinetic energy $\E_\la$ of unit order we must have
$$\mu:=\mu_\lambda=\lambda^{3+\gamma}.$$
For such a scaling, equation \eqref{rescaledmu} becomes
\begin{equation}
\label{rescaled}
\Q_{\el}(\gl,\gl)+\la^\gamma\Delta_v \gl=0.
\end{equation}
Note that with our choice of $\mu_\la$ the limit $G_0$ as $\la \to 0$ of $\gl$, if it exists, has to satisfy
$$\Q_1(G_0,G_0)=0.$$
In other words, $G_0$ is a suitable Maxwellian with same mass and momentum that $\gl$. Moreover, using the dissipation functional
\begin{equation}
\label{scaleDE}
6\varrho=\dfrac{1}{\lambda^\gamma}\IRR \gl(v)\gl(\vb)\mathbf{\Psi}_{e}(\la^2|v-\vb|^2)\d v\d\vb
\end{equation}
one expects that the limit $G_0$ satisfies
\begin{equation}\label{rhozeta0} 6\varrho=\IRR G_0(v)G_0(\vb)\zeta_0(|v-\vb|^2)\d v\d\vb\end{equation}
with $\zeta_0(r^2)=\lim_{\la \to 0}\frac{1}{\la^\gamma}\mathbf{\Psi}_{e}(\la^2 r^2)$ (several properties of such energy dissipation functionals are investaged in Appendix A).  With this observation, it is not difficult to prove that the unique possible limit as $\la \to 0$ of $\gl$ is the Maxwellian distribution
\begin{equation}\label{maxwellian}
\mathcal{M}(v)=\dfrac{\varrho}{(2\pi\Theta)^{3/2}}\exp\left(-\dfrac{|v|^2}{2\Theta}\right)
\end{equation}
for some explicit temperature $\Theta$ determined by the above identity \eqref{rhozeta0}.

The limiting Maxwellian $\mathcal{M}(v)$ is called the \textit{quasi-elastic limit} for this problem.  With this knowledge we will prove uniqueness of solutions for the problem \eqref{rescaled} when the rescaling parameter $\lambda$ is small (lying in an explicit interval).  In this sense our uniqueness result will be valid in the \textit{weak thermalization regime.}

\subsection{Main results and strategy} \label{sec:strategy} Let us state precisely the main problems we wish to address in this document:
\begin{enumerate}
\item Prove that \textit{any solution} $\gl$ to \eqref{rescaled} satisfies
$$\lim_{\la \to 0}\gl=\mathcal{M}$$
in some suitable sense.  Specifically, find a suitable Banach space $X$ such that $\gl \in X$ for any $\la>0$ and $\lim_{\la \to 0}\|\gl-\mathcal{M}\|_X=0$.
\item Prove that solution $\gl$ to \eqref{rescaled} is unique, at least in the weak elastic regime.  That is, determine $\la^\dag \in (0,1)$ such that $\mathfrak{S}_\la$ reduces to a singleton as soon as $\la\in(0,\la^\dag]$.  We use the symbol $\mathfrak{S}_\la$ to denote the set of solution $\gl$ to \eqref{rescaled} with given mass and vanishing momentum.
\item Provide quantitative answers to the two previous questions. More precisely, find the rate of convergence of $\gl$ towards $\mathcal{M}$ as well as some estimate for the parameter $\la^\dag$.
\end{enumerate}

The first question is answered with the following theorem (see Theorem \ref{convergenceMax} for a detailed statement) that can be interpreted as a  \textit{quasi-elastic limit} result.
\begin{theo}\label{intro:conv} If $e(\cdot)$ is of class $\mathcal{C}^m$  with $m \geq 3$ (with some additional regularity properties), one has
$$\lim_{\la \to 0}\|\gl-\mathcal{M}\|_{\mathbf{H}^\ell_k}=0 \qquad \forall k \geq 0, \qquad \forall \ell \in [0,m-2].$$ The limit $\mathcal{M}$ is the Maxwellian given by \eqref{maxwellian} with an explicit temperature $\Theta$ given by \eqref{temptheta}. The convergence also holds in exponential weighted $L^1$-spaces.
\end{theo}
The proof of the above result is based upon a compactness argument and requires a careful investigation of the regularity properties of the solution to \eqref{rescaled}.  Our approach for the study of regularity of solutions to \eqref{rescaled} differs from the related contributions on the matter \cite{MiMo3,MiMo2} where the regularity of steady solutions is deduced from the properties of the time-dependent problem (namely on the propagation of regularity combined with the damping in time of the singularities for solution to the time-dependent problem, see \cite{MoVi}). In contrast with these results, our methodology is direct and \textit{relies only} on the steady equation \eqref{rescaled}.   It is not difficult to prove by a bootstrap argument that any solution $\gl$ to \eqref{rescaled} is smooth, however, it is more delicate to obtain regularity estimates which are \textit{uniform} with respect to the parameter $\la >0$ since the diffusive heating in \eqref{rescaled} is vanishing in the limit $\la \to 0$.  On the basis of \textit{new regularity estimates} of the collision operator (see Theorem \ref{regularite}), we can prove the following proposition.
\begin{propo}\label{introregu}   Under the regularity assumptions on $e(\cdot)$ of Theorem \ref{intro:conv}, one has
$$\sup_{\la \in (0,1]} \|\gl\|_{\mathbf{H}^\ell_k} < \infty \qquad \forall k \geq 0, \qquad \ell \in (0,m-1].$$
\end{propo}
Theorem \ref{intro:conv} serves as a fundamental brick to prove the main result of the paper.
\begin{theo}\label{introunique} Under suitable regularity assumptions on $e(\cdot)$ there exists $\la^\dag \in (0,1]$ such that the set $\mathfrak{S}_\la$ of solutions to \eqref{rescaled} with given mass $\varrho$ and vanishing momentum reduces to a singleton for any $\la \in [0,\la^\dag]$.
\end{theo}

Theorem \ref{introunique} can be interpreted as an uniqueness result in the \textit{quasi-elastic} regime where $\la$ is small.  This theorem, however, can also be interpreted as a \textit{weak thermalization} uniqueness result since in this regime the diffusion parameter $\mu$ is small as well.

\begin{theo}\label{introdiff} Under suitable regularity assumptions on $e(\cdot)$, there exists $\mu^\dag >0$ such that, for any $\mu \in (0,\mu^\dag]$ the steady problem
$$\Q_e(F,F)+\mu\Delta F=0$$
admits an unique solution $F$ for a given mass and vanishing momentum.
\end{theo}

The proof of Theorem \ref{introunique} follows the strategy of \cite{MiMo3} (see also \cite{MiMo2} and \cite{canizo}).  Essentially, it is based on the knowledge of the quasi-elastic limit problem and on quantitative estimates of the difference between solutions to the
original problem and the equilibrium state as $\la \to 0$.  More precisely, let us consider two steady solutions $\gl, \fe \in \mathfrak{S}_\la$. Set then $H_\la=\fe-\gl$ and define  the linearized  elastic Boltzmann operator around the limiting Maxwellian $\mathcal{M}$
\begin{equation}
  \label{linearized1}
  \mathscr{L}_1 (h)=\Q_1(\mathcal{M},h)+\Q_1(h,\mathcal{M}).
\end{equation}
Observing that $ \mathscr{L}_1$ is a symmetric operator one recognizes
\begin{multline*}
\mathscr{L}_1(H_\la)=\bigg(\Q_1(H_\la,\mathcal{M}) -\Q_{\el}(H_\la,\mathcal{M})\bigg)
+\bigg(\Q_1(\mathcal{M}, H_\la) -\Q_{\el}(\mathcal{M},H_\la)\bigg)\\
+\bigg(\Q_{\el}(\mathcal{M}-\fe,H_\la) +\Q_{\el}(H_\la,\mathcal{M}-\gl)\bigg)-\la^\gamma\Delta H_\la
\end{multline*}
where we used that $\Q_{\el}(\fe,\fe)-\Q_{\el}(\gl,\gl)=\la^\gamma \Delta H_\la.$  Assume that there exist two Banach spaces $\mathcal{X}$ and $\mathcal{Y}$ independent of $\la$ such that
\begin{equation}\label{continueQ1}
\left\|\Q_{\el}(f,g)\right\|_{\mathcal{X}}+ \left\|\Q_{\el}(g,f)\right\|_{\mathcal{X}} \leq C_1  \|f\|_{\mathcal{Y}}\|g\|_{\mathcal{Y}},
\end{equation}
and
\begin{equation}
  \label{XY}
  \left\|\Q_1(f,\mathcal{M})-\Q_{\el}(f,\mathcal{M})\right\|_{\mathcal{X}} + \left\|\Q_1(\mathcal{M},f)-\Q_{\el}(\mathcal{M},f)\right\|_{\mathcal{X}}  \leq C_2\la^p \|f\|_{\mathcal{Y}},
\end{equation}
and also,
\begin{equation}\label{LaplX}
\|\Delta H_\la\|_{\mathcal{X}} \leq C_3 \|H_\la\|_{\mathcal{Y}} \qquad \forall \la > 0
\end{equation}
for some constants $C_1, C_2, C_3>0$, $p \in (0,\gamma)$ independent of $\la$.  Notice that \eqref{LaplX} is too restrictive for a general function $f \in \mathcal{Y}$, it is only assumed for any difference $H_\la$. Then,
$$\|\mathscr{L}_1(H_\la)\|_{\mathcal{X}}\leq C_1\|H_\la\|_{\mathcal{Y}}\big(\|\mathcal{M}-\gl\|_{\mathcal{Y}} +\|\mathcal{M}-\fe\|_{\mathcal{Y}}\big)
+(C_2+C_3)\la^p \|H_\la\|_{\mathcal{Y}}.$$
If $\mathcal{M}$ is the universal limit of the family $\mathfrak{S}_{\lambda}$, if one is able to prove that
\begin{equation}\label{conveX}
\lim_{\la \to 0} \big(\|\mathcal{M}-\gl\|_{\mathcal{Y}} +\|\mathcal{M}-\fe\|_{\mathcal{Y}}\big)=0
\end{equation}
then, for any   $\varepsilon >0$ there exists $\la_0 >0$ such that
$$\|\mathscr{L}_1(H_\la)\|_{\mathcal{X}} \leq \varepsilon \|H_\la\|_{\mathcal{Y}}.$$
The strategy ends by proving that there exists a subspace $\widehat{\mathcal{Y}} \subset \mathcal{Y} $ containing the net $\{H_{\lambda}\}_{\lambda\in(0,1]}$ such that
\begin{equation}\label{XY2}
\|\mathscr{L}_1 (h)\|_\mathcal{X} \geq c_0 \|h\|_\mathcal{Y} \qquad \forall h \in \widehat{\mathcal{Y}}.
\end{equation}
Thus, for any  $\varepsilon >0$ there exists $\la_0 >0$ such that
$$c_0\| H_\la\|_{\mathcal{Y}}\leq \varepsilon \|H_\la\|_{\mathcal{Y}} \qquad \forall \la \in (0,\la_0).$$
This proves that $H_\la=0$ for any $\la \in (0,\la_0).$ Notice that $\la_0$ would become explicit if we are able to make explicit the rate of convergence of \eqref{conveX}.

To summarize, the proof reduces to find Banach spaces $\mathcal{X}$ and $\mathcal{Y}$ for which the above equations \eqref{continueQ1}--\eqref{XY2} hold.  We warn the reader here that \eqref{XY2} will have to be slightly modified because \textit{a priori} the energy of the difference $H_{\lambda}$ is not necessarily zero.  This technical detail is overcame by introducing a suitable lifting operator of $\mathscr{L}_1$, see \cite{MiMo3} for the original implementation of this idea.
 We can already anticipate that the strategy will be applied to the following weighted $L^1$-spaces
$$\mathcal{X}=L^1(m_a)=L^1(\R^3\,;\,m_a(v)\d v) \quad \text{ and } \quad \mathcal{Y}=L^1_1(m_a)=L^1(\R^3\,;\,m_a(v)\sqrt{1+|v|^2}\d v)$$
where the exponential weight function $m_a$ is given by
$$m_a(v)=\exp\left(a|v|\right)\,\quad  a \geq 0.$$
The most technical parts of the proof will be the $\lambda$-uniform regularity of $G_{\lambda}$ and the continuity estimate \eqref{XY} with respect to the restitution coefficient.  These aspects have been proved for \textit{constant restitution coefficient} in \cite{MiMo3}, however, their extension to the case of a variable restitution coefficient will be delicate and require a series of new technical results.\\

Finally, explicit estimates on the rate of convergence of the rescaled solution $G_{\lambda}$ towards the elastic limit $\mathcal{M}$ can be found \textit{a posteriori} by seeking a nonlinear inequality satisfied by
$\|\gl-\mathcal{M}\|_{\mathcal{Y}}$.  More precisely, we shall prove that there exist some explicit constants $C_1, C_2 >0$ such that
$$\|\gl-\mathcal{M}\|_{\mathcal{Y}}\leq C_1\la^p+C_2\|\gl-\mathcal{M}\|^2_{\mathcal{Y}} \qquad \forall \la \in (0,1].$$
Combining this estimate with the convergence of $\gl$ towards $\mathcal{M}$ will lead to the existence of some $\la^\dag \in (0,1]$ such that
$$\|\gl-\mathcal{M}\|_{\mathcal{Y}} \leq C_3 \la^p \qquad \forall \la \in (0,\la^\dag],$$
for a suitable exponent $p>0$ and explicit constant $C_{3}$.

We stress here that in contrast with the reference \cite{MiMo3}, the present manuscript provides an approach which does not rely on entropy estimates.  Consequently, it does not require neither exponential pointwise lower bounds nor strong regularity properties in the steady state.  In particular, it is well-suited for problems in which no regularity of the steady solution is available, see \cite{cold} for an example of this situation.

\subsection{Organization of the paper} The plan of the paper is the following.  In Section 2, we establish \textit{new} regularity estimates of the collision operator $\Q_e^+$ generalizing known results for the elastic case \cite{lions,Wennberg,MoVi} and for the inelastic case with constant restitution coefficient \cite{MiMo,MMR}.  Section 3 is devoted to study regularity properties of the steady solution $\gl \in \mathfrak{S}_\la$.  In particular, we study moments \cite{BoGaPa,MMR,AloLo1}, general weighted Sobolev regularity given in Proposition \ref{introregu} and a technical result on the difference of two solutions, see Proposition \ref{prop:Hkdiff}.  Moreover, we also address in this section the fundamental problem of the continuity properties of $\Q_{\el}$ with respect to the inelasticity parameter, proving the convergence of $\Q_{\el}(f,g)$ towards $\Q_1(f,g)$ as $\la \to 0$ in different norms.  Several of these results are non-trivial extensions of those in \cite{MiMo2} given for constant restitution coefficient.  Others, like Proposition \ref{propo:sobdiff} are new.   Section 4 contains the main results of the paper, namely the elastic limit result Theorem \ref{intro:conv}, the uniqueness result Theorem \ref{introunique} and their quantitative versions. In Appendix A, we present several technical results used throughout the paper and Appendix B contains a proof of Theorem \ref{theo:exists} on existence of steady profiles adapted from \cite{GaPaVi}.

\subsection{Notation} Let us introduce the notations we shall use
in the sequel. Throughout the paper we shall use the notation
$\langle \cdot\rangle = \sqrt{1+|\cdot|^2}$. We denote, for any
$p\in[1,+\infty)$, $\eta \in \R$ and weight function $\varpi\::\:\R^3 \to \R^+$, the weighted Lebesgue space
\begin{equation*}
     L^p_\eta (\varpi)= \left\{f: \R^3 \to \R \hbox{ measurable} \, ; \; \;
     \| f \|_{L^p_\eta(\varpi)} := \left(\int_{\R^3} | f (v) |^p \, \langle v \rangle^{p\eta}\varpi(v)\d\v\right)^{1/p}
     < + \infty \right\}.
\end{equation*}
Similarly, we define the weighted Sobolev space $\mathbf{W}^{k,p}_\eta (\varpi)$, with $k \in \mathbb{N}$, using the norm
 \begin{equation*}\label{sobnorm}
\| f \|_{\mathbf{W}^{k,p}_\eta (\varpi)}  =   \left( \sum_{|s| \le k} \|\partial^s f\|_{L^p_\eta(\varpi)} ^p \right)^{1/p}.
\end{equation*}
The symbol $\partial^s$ denotes the partial derivative associated with
the multi-index $s \in \mathbb{N}^3$: $\partial^s=\partial_{v_1}^{s_1}\partial_{v_2}^{s_2}\partial_{v_3}^{s_3}$. The order of the multi-index being defined as $|s|=s_1+s_2+s_3.$ In the particular case $p=2$ we denote $\mathbf{H}^k _\eta(\varpi)=\mathbf{W}^{k,2} _\eta(\varpi)$ and whenever $\varpi(v)\equiv 1$, we shall simply use $\mathbf{H}^k_\eta$.  This definition can be extended to $\mathbf{H}^s _\eta$ for any $s \ge 0$ by using Fourier transform.  

\section{Regularity properties of the collision operator}

The smoothing properties of the gain operator $\Q_e^+(f,f)$ have been investigated in our previous contribution \cite{AloLo1}. However, for the results we have in mind, we shall need  the regularity of the \textit{bilinear} operator $\Q_e^+(f,g)$ rather than the one of the \textit{quadratic} one. To do so, we shall use a slightly different approach than the one we used in \cite{AloLo1}. In particular, our main purpose here is to extend \cite[Theorem 4.1]{AloLo1} to smooth kernel that are \textit{not compactly supported}: in such a case, the price to pay for the control  of large velocities consists in additional moments estimates. Precisely, using the notations of Appendix A, let $B(u,\sigma)$ be a collision kernel of the form
$$
B(u,\sigma)=\Phi(|u|)b(\widehat{u} \cdot \sigma)
$$
where $\Phi(\cdot) \geq 0$ and $b(\cdot) \geq 0$ satisfies (A.3) and (A.4) of the Appendix. Then, one can define the following operator
$\Gamma_B$ by
\begin{equation}\label{expgammaB}
\Gamma_B(\varphi)(x)=\int_{\omega^\perp} \mathcal{B}(z+\alpha_e(r)\omega,\alpha_e(r))\varphi(\alpha_e(r)\omega+z)\d\pi(z),
 \qquad x=r\omega, \:r\geq 0,\:\omega \in \mathbb{S}^2,
 \end{equation}
where $\d\pi $ is the Lebesgue measure over the hyperplane $\omega^\perp$ perpendicular to $\omega$ and $\alpha_e(\cdot)$ is the inverse of the mapping $s \mapsto s\beta_e(s)$ while the kernel $\mathcal{B}(\cdot,\cdot)$ is given by
\begin{equation}\label{calB}
\mathcal{B}(z,\varrho)= \frac{8\Phi(|z|)}{|z|(\varrho \beta_e(\varrho))^2}b\left(1-2\frac{\varrho^2}{|z|^2}\right)\dfrac{\varrho}{1+\vartheta_e'(\varrho)}, \qquad \varrho \geq 0, \quad z \in \R^3
\end{equation}
with $\vartheta_e(\cdot)$ defined in Assumption \ref{HYP} (2) and $\vartheta'_e(\cdot)$ denoting its derivative. The operator $\Gamma_B$ can be seen as an inelastic version of the so-called cold thermostat operator investigated in \cite{cold} (and originally derived in the seminal paper \cite{lions}) and plays a crucial role in the smoothing properties of the gain operator $\Q_e^+$ because of the representation formula
\begin{equation}\label{qfgamma}
\Q_{B,e}^+(f,g)(v)=\IR f(z)\left[\left(t_z  \circ \Gamma_B \circ t_z\right)g\right] (v)\d z \qquad \forall f,g \end{equation}
where $[t_v \psi](x)=\psi(v-x)$ for any $v,x \in \R^3$ and test-function $\psi$ (see \cite{AloLo1} for the derivation of \eqref{qfgamma}). Before investigating the regularity of the full gain operator, we shall first deal with that of the cold thermostat.

\subsection{Regularity properties for cut-off collision kernels}

For this section we assume that the kernel $B(u,\sigma)$ satisfies:
\begin{equation}\label{smoothphi}
\Phi(\cdot) \in \mathcal{C}^\infty(0,\infty), \quad b(\cdot) \in \mathcal{C}_0^\infty(-1,1) \quad \text{ and } \quad \Phi(s)=\left\{
\begin{array}{ccl}
0 & \mbox{for} & s<\epsilon \\
s  & \mbox{for} & s>2\epsilon,
\end{array}
\right.
\end{equation}
for some $\epsilon>0$.   We introduce the following definition:
 \begin{defi}
\label{HYP2}
We shall say that a restitution coefficient $e(\cdot)$ satisfying Assumptions \ref{HYP} is belonging to the class $\mathbb{E}_m$ for some integer $m \geq 1$ if  $e(\cdot) \in \mathcal{C}^m(0,\infty)$  and
\begin{equation}\label{assm}
\sup_{r \geq 0} re^{(k)}(r) < \infty \qquad \forall k=1,\ldots,m
\end{equation}
where $e^{(k)}(\cdot)$ denotes the $k$-th order derivative of $e(\cdot)$.
 \end{defi}
\begin{nb} For the physically relevant case of visco-elastic hard-spheres, the restitution coefficient $e(\cdot)$ is given by \eqref{visc} but admits also the following implicit representation (see \cite{BrPo}):
$$e(r)+\mathfrak{a}r^{\tfrac{1}{5}}\,e^{\tfrac{3}{5}}(r)=1 \qquad \forall r >0$$
for some $\mathfrak{a} >0$. Then, it is possible to deduce from such representation that $e(\cdot)$ belongs to the class $\mathbb{E}_m$ for \emph{any} integer $m \geq 1$.
\end{nb}
Under these assumptions we have the following generalization of \cite[Lemma 4.6]{AloLo1}.
\begin{lemme}\label{Sob}
Assume that $e(\cdot)$ belongs to the class $\mathbb{E}_m$ for some integer  $m \geq 2$ and that the collision kernel $B(u,\sigma)$ satisfies assumption \eqref{smoothphi}. Then, for any $0 \leq s \leq m-2$, there exists $C=C(s,\epsilon,e)$ such that
\begin{equation}\label{gammaBHS}
\left\|\Gamma_B(f)\right\|_{\mathbf{H}^{s+1}_\eta} \leq C\,\left\|f\right\|_{\mathbf{H}^s_{\eta+\mu(s)}},\qquad \forall \eta \geq 0
\end{equation}
with $\mu(s)=s+4$ and where the constant $C(s,\epsilon,e)$ depends only on $s$, on the collision kernel $B$ and the restitution coefficient $e(\cdot)$.
\end{lemme}
\begin{proof} There is no loss of generality in assuming that $s$ is an integer. The proof is divided into five steps.

$\bullet$ \textit{First step: change of variables} Recalling \cite[Lemma 4.6]{AloLo1}, we first define
\begin{equation}\label{lien}
\widetilde{\Gamma_B}(f)(r\omega)=\Gamma_B(f)(\alpha_e^{-1}(r),\omega)=\Gamma_B(f)(r\beta_e(r),\omega)\end{equation}
so that
$$\widetilde{\Gamma_B}(f)(r\omega)=\int_{\omega^\perp} \mathcal{B}(z+r\omega,r)\varphi(r\omega+z)\d\pi(z).$$
We begin proving the result for  $\widetilde{\Gamma_B}$ instead of $\Gamma_B$, that is
\begin{equation}\label{gammaBHStilde}
\left\|\widetilde{\Gamma_B}(f)\right\|_{\mathbf{H}^{s+1}_\eta} \leq \widetilde{C}(s,B,e)\,\left\|f\right\|_{\mathbf{H}^s_{\eta+\mu(s)}},\qquad \forall \eta \geq 0
\end{equation}
The proof of this estimate follows the approach given in \cite[Theorem 3.1]{MoVi} where a similar estimate has been obtained, for $\mu(s)=0$, under the additional assumption that $\Phi(\cdot)$ has support in $[\epsilon,M]$ with $M<\infty$. Our proof will consists essentially in proving that the weighted estimate (i.e. with $\mu=\mu(s) >0$) allows to take into account large velocities.

$\bullet$ \textit{Second step: Estimates on the radial derivative of $\widetilde{\Gamma_B}.$} We introduce the radial Fourier
transform $\RF$ and the Fourier transform $\F$ in $\R^3$ with
the formulas
$$\RF \left[f\right] (\varrho w) =(2\pi)^{-1/2} \int_{\R}
\exp(i\varrho r) f(r w) \d r\,,\qquad \F\left[ f\right](\xi) ={(2\pi)^{-3/2}} \int_{\R^3}
\exp(i v\cdot \xi) f(v) \d v$$
and, for any measurable mapping $g$, we define the $\mathbf{H}^{s+1}_\eta(\mathbb{S}^2 \times \R)$ norm of $g$ as
$$\|g\|^2 _{\mathbf{H}^{s+1} _\eta(\mathbb{S}^{2} \times \R)}:=\IS \d w \int_{\R} \langle \varrho\rangle^{2(s+1)}|\RF \left[g\right](\varrho w)|^2\d\varrho.$$
Then we compute,
\begin{multline*}
\RF \left[\langle r \rangle^\eta \,\widetilde{\Gamma_B}f\right] (\varrho w)=(2\pi)^{-\frac{1}{2}}\IR \exp(i \varrho u \cdot w) \langle u \cdot w\rangle^\eta \mathcal{B}(u,|u\cdot w|)f(u)\d u\\
=2\pi \F\left[f(\cdot)\mathbf{G}_w(\cdot) \right](\varrho w)
\end{multline*}
where $\mathbf{G}_w(u)=\mathcal{B}(u,|u\cdot w|)\langle u \cdot w \rangle^\eta$ for any $u \in \R^3$. Then, setting $\xi=\varrho w$, since $\d\varrho\d w=|\xi|^{-2}\d\xi$ we get
$$\left\|\widetilde{\Gamma_B}f\right\|^2 _{\mathbf{H}^{s+1} _\eta(\mathbb{S}^{2} \times \R)}
=2\pi \int_{\R^3}
 \langle \xi \rangle^{2s+2} |\xi|^{-2}\left| \F
 \left[ f(\cdot)\mathbf{G}_{\frac{\xi}{|\xi|}}(\cdot)
 \right] (\xi)  \right|^2 \d\xi.$$
Now, with $\mu=\mu(s)=s+4$, introducing $g(v)=f(v)\langle v \rangle^\mu$  and $\mathbf{G}^\mu_w(z)=\langle z\rangle^{-\mu}\mathbf{G}_w(z)$
we can write the above as
$$\left\|\widetilde{\Gamma_B}f\right\|^2 _{\mathbf{H}^{s+1} _\eta(\mathbb{S}^{2} \times \R)}=2\pi \int_{\R^3}
 \langle \xi \rangle^{2s+2} |\xi|^{-2}\left| \F
 \left[ g(\cdot)
 \mathbf{G}^\mu_{\frac{\xi}{|\xi|}}(\cdot)
 \right] (\xi)   \right|^2 \d\xi.$$
Using \cite[Lemma A.5]{MoVi} we can replace \cite[Eq. (3.7)]{MoVi} by
 \begin{equation}\begin{split}\label{eq:estim:step2}
 \left\|\widetilde{\Gamma_B} f\right\| _{\mathbf{H}^{s+1} _\eta
 (\mathbb{S}^{2} \times \R)} &\leq
 C_s \, \left\|g\right\|_{\mathbf{H}^s _{\eta}}
 \sup_{w \in \mathbb{S}^{2}}
 \left\|\mathbf{G}^\mu_w(\cdot)\right\|_{\mathbf{H}^{s+2} (\R^3)}\\
 &=
C_s \, \left\|f\right\|_{\mathbf{H}^s _{\eta+\mu}}
 \sup_{w \in \mathbb{S}^{2}}
 \left\|\mathcal{B} (z, | z \cdot w |)\frac{\langle  z \cdot w  \rangle^\eta}
 {\langle z \rangle^{\eta+\mu}}\right\|_{\mathbf{H}^{s+2} (\R^3_z)}.
\end{split} \end{equation}
\medskip

$\bullet$ \textit{Third step: Control of the angular derivatives.}
In order to adapt the analysis of \cite{MoVi}, it suffices to check that there are two positive constants $a,b >0$ such that
$$\mathrm{Supp}(\mathcal{B}) \subset [a,\infty) \times [b,\infty).$$
We already saw this is the case with our assumption on $\Phi(\cdot)$ and $b(\cdot)$.  We can straightforwardly apply the reasoning of the \textit{op. cit.} to get that the $j$-th angular derivative of $\widetilde{\Gamma_B} f$ can be estimated by the radial derivative of $\widetilde{\Gamma_{B_j}} f$ where the new kernel $\mathcal{B}_j$ is given by $\mathcal{B}_j(z,\varrho)=\mathcal{B}(z,\varrho) z_j$. This finally leads to \eqref{gammaBHStilde} with
\begin{equation}\label{imc}
\widetilde{C}(s,B,e)=C_s(a,b)\underset{ |\nu|\leq s+1}{\sup_{\nu \in \mathbb{N}^3}} \sup_{w \in \mathbb{S}^{2}}
 \left\|\mathcal{B} (z, | z \cdot w |)z^\nu\frac{\langle  z \cdot w  \rangle^\eta}
 {\langle z \rangle^{\eta+\mu}}\right\|_{\mathbf{H}^{s+2} (\R^3_z)} .
\end{equation}

$\bullet$ \textit{Fourth step:}  Let us check that the above quantity is indeed finite, i.e.
$$\mathcal{C}_s(\mathcal{B},\nu):=\sup_{w \in \mathbb{S}^{2}}
 \left\|\mathcal{B} (z, | z \cdot w |)z^\nu\frac{\langle  z \cdot w  \rangle^\eta}
 {\langle z \rangle^{\eta+\mu}}\right\|_{\mathbf{H}^{s+2} (\R^3_z)}
< \infty$$
for any multi-index $\nu$ with $|\nu| \leq s+1$.
Observe that, because of our cut-off assumptions \eqref{smoothphi} together with the fact that $b(1-x)=0$ for small values of $x$, the kernel $\mathcal{B}(z,\varrho)$ vanishes for small values of $|z|$ and $\varrho$.  Thus, for a given $\nu$ with $|\nu|\leq s+1$, it suffices to investigate the regularity properties of the above mapping
$$F_w\::\:z \longmapsto \mathcal{B} (z, | z \cdot w |)z^\nu\frac{\langle  z \cdot w  \rangle^\eta}
 {\langle z \rangle^{\eta+\mu}}$$
for large value of $z$ (uniformly with respect to $w$). From the definition of $\mathcal{C}_s(\mathcal{B},\nu)$, one needs  to compute $s+2$ derivatives of $F_w$ which explains the restriction $s \leq m-2$. Recall that, for $|z| >\epsilon,$ the expressions of $\mathcal{B}(\cdot,\cdot)$ and $\Phi(\cdot)$ yield  the following
$$F_w(z)=\dfrac{8b(1-2|\widehat{z}\cdot w|^2)}{\mathcal{H}(|z\cdot w|)}  \dfrac{z^\nu\langle z \cdot w\rangle^\eta}{\langle z\rangle^{\eta+\mu}}$$
where $\widehat{z}=z/|z|$,
$$\mathcal{H}(r)=r\beta_e^2(r)(1+\vartheta_e'(r)), \qquad \qquad r \geq 0.$$
We recall that $b(1-2x^2)=0$ for $|x| \leq \delta$ for some $\delta >0$. In particular, for $|z|> \epsilon$, $F_w(z)\neq 0 \implies r:=|z\cdot w|> \delta\epsilon$. Since $\beta_e(\cdot) \in (1/2,1]$ and $\vartheta_e' \geq 0$, it is easy to check that
$$|F_w(z)| \leq \dfrac{32|z|^{|\nu|}}{\delta \epsilon\langle z\rangle^\mu}b(1-2|\widehat{z}\cdot w|^2)\leq \dfrac{32\|b\|_\infty}{\delta\epsilon\,\langle z\rangle^{\,\mu-|\nu|}}\qquad \forall w\in\mathbb{S}^2,\:|z|>\epsilon.$$
Since $\mu-s=4$, this proves that $\sup_{w\in \mathbb{S}^2}\|F_w(z)\|_{L^2(\R^3_z)} < \infty.$ One proceeds in the same way with the $z$-derivatives of $F_w(z)$. It is clear that any $z$-derivative of the rational expression
$$R_w(z):=\dfrac{z^\nu\langle z \cdot w\rangle^\eta}{\langle z\rangle^{\eta+\mu}}$$
has a faster decay (for $|z| \to \infty$) than $R_w(z)$. Therefore, the crucial point is the control of the derivatives of $\frac{1}{ \mathcal{H}(r)}$. It turns out that
$$\dfrac{\mathcal{H}'(r)}{\mathcal{H}(r)}=\dfrac{1}{r}+\dfrac{2\beta_e'(r)}{\beta_e(r)}+\dfrac{\vartheta_e''(r)}{1+\vartheta'_e(r)}
=\dfrac{1}{r}+\dfrac{2e'(r)}{1+e(r)}+\dfrac{2e'(r)+re''(r)}{1+\vartheta_e'(r)}.$$
Now, our assumption \eqref{assm} on the restitution coefficient $e(\cdot)$ implies easily that ${\mathcal{H}'}/{\mathcal{H}} \in L^\infty([\delta\epsilon,\infty))$ and, as a direct consequence, $$\dfrac{\d}{\d r}\dfrac{1}{\mathcal{H}(r)} \in L^\infty([\delta\epsilon,\infty)).$$ Similar calculations show that, for any $k=1,\ldots,m$,
$\frac{\d^k}{\d r^k}\frac{1}{\mathcal{H}(r)} \in L^\infty([\delta\epsilon,\infty)).$
Tedious but simple calculations show  then that  any $z$-derivative of $F_w(z)$ can be controlled by $1/|z|^{\mu-|\nu|+1}$ for large $|z|$. This is enough to prove that $\mathcal{C}_s(\mathcal{B},\nu) < \infty.$

$\bullet$ \textit{Final step: turning back to the original variables.} Following \cite{AloLo1}, it remains now to deduce estimates on $\Gamma_B f$ from $\widetilde{\Gamma_B} f$ which are linked by formula \eqref{lien}. Using polar coordinates
\begin{equation*}
\| \Gamma_B(f)\|_{\mathbf{H}^s_\eta}^2=\sum_{|j| \leq s} \int_0^\infty F_j(\varrho)\varrho^2 \langle \varrho\rangle^{2\eta} \d\varrho \IS |\partial_v^j \widetilde{\Gamma_B}(f)(\varrho,\omega)|^2\d\omega
\end{equation*}
where, see \cite{AloLo1} for details, one can check that for any $|j| \leq k$ the function $F_j(\varrho)$ can be written as
\begin{equation}\label{polynome}
F_j(\varrho)=P_j(\vartheta_e^{(1)}(\varrho),\ldots,\vartheta_e^{(j)}(\varrho)) (1+\vartheta_e^{(1)}(\varrho))^{-n_j}.
\end{equation}
Here $P_j(y_1,\ldots,y_j)$ is a suitable polynomial, $n_j \in \mathbb{N}$ and $\vartheta_e^{(p)}$ denotes the $p$-th derivative of $\vartheta_e(\cdot)$. Because of our assumption on $e(\cdot)$ (more precisely, because $\limsup_{\varrho \to \infty}\vartheta_e^{(i)}(\varrho) < \infty$), we see that $\sup_{\varrho \in (0,\infty)} F_j(\varrho)=C_j < \infty $ for any $|j| \leq k$.  Thus
\begin{equation}\label{tilde}
\| \Gamma_B(f)\|_{\mathbf{H}^s_\eta} \leq C_\eta \|\widetilde{\Gamma_B}(f)\|_{\mathbf{H}^s_{\eta}}
\end{equation}
where $C_\eta$ is an explicit constant involving the $L^\infty$ norm of the first $s$-th order derivatives of $\alpha_e(\cdot)$.
\end{proof}

\begin{propo}\label{theo:smooth}
Let $B(u,\sigma)=\Phi(|u|)b(\widehat{u} \cdot \sigma)$ be a collision kernel satisfying \eqref{smoothphi} and $e(\cdot)$ be in the class $\mathbb{E}_m$ $(m \geq 2)$.  Then, for any $0 \leq s \leq m-2$,
$$
\left\|\Q_{B,e}^+(f,g)\right\|_{\mathbf{H}^{s+1}_\eta} \leq C(s,B,e)\,\|g\|_{\mathbf{H}^s_{\eta+\mu(s)}}\,\|f\|_{L^1_{2\eta+\mu(s)}}
$$
with constant $C(s,B,e)$ and $\mu(s)=s+4$.
\end{propo}
\begin{proof} One uses the representation formula \eqref{qfgamma} together with Minkowski's inequality to get that
$$\|\Q_{B,e}^+(f,g)\|_{\mathbf{H}^{s+1}_\eta} \leq \IR |f(z)|\,\|t_z \circ \Gamma \circ t_z g\|_{\mathbf{H}^{s+1}_\eta}\d z.$$
Now, since $\|t_z \psi\|_{\mathbf{H}^N_k} \leq \langle z \rangle^k \|\psi\|_{\mathbf{H}^N_k}$ for any $\psi \in \mathbf{H}^N_k$ for all $N \in \mathbb{N}$ and any $k \geq 0$, one easily deduces from Lemma \ref{Sob} the conclusion.
\end{proof}

\subsection{{Regularity properties for hard-spheres collision kernel}} We now use the previous result for smooth collision kernels to estimate the regularity properties of $\Q^+_e(f,g)$ for true hard-spheres interactions. We shall combine Theorem A. \ref{alogam} of the Appendix together with the estimates of the previous section to get the following:
\begin{theo}\label{regularite} Assume that $e(\cdot)$ belongs to the class $\mathbb{E}_m$ with $m \geq 2$. Then, for any $\varepsilon >0$ and any  $\eta \geq 0$, there exists $C_e=C(e,\varepsilon,\eta)$ such that
\begin{multline}\label{regestim}
\|\Q_e^+(f,g)\|_{\mathbf{H}^{s+1}_\eta} \leq C_e\,\|g\|_{\mathbf{H}^s_{2\eta+\mu(s)}}\|f\|_{L^1_{2\eta+\mu(s)}} + \varepsilon\| f\|_{\mathbf{H}^s_{\eta+3}}\,\|g\|_{\mathbf{H}^s_{\eta+1}}\\
+\varepsilon\left( \| g\|_{L^1_{\eta+1}}\,\|\partial^\ell f\|_{L^2_{\eta+1}}+\| f\|_{L^1_{\eta+1}}\,\| \partial^\ell g\|_{L^{2}_{\eta+1}}\right) \qquad \forall |\l|=s+1 \leq m-1.
\end{multline}
\end{theo}
\begin{proof} Notice that, for hard-spheres interactions, one has $B(u,\sigma)=\Phi(|u|)b(\widehat{u} \cdot \sigma)$ with $\Phi(|u|)=|u| \in L^\infty_{-1}$ and $b(s)=1/4\pi$ for any $s \in (-1,1)$. In particular, for any $\eta \geq 0$, both the constant $C_{2,1,\eta,1}(b)$ and $C_{2,2,\eta,1}(b)$ appearing in \eqref{Crpakb} are \textit{finite}. Let us now fix $\eta \geq 0$ and $\varepsilon >0$ and split the kernel into four pieces
\begin{multline}\label{splitting}
B(|u|,\hat{u}\cdot\sigma)
=\Phi_S(|u|)b_{S}(\hat{u}\cdot\sigma)+\Phi_S(|u|)b_{R}(\hat{u}\cdot\sigma)\\
+\Phi_R(|u|)b_{S}(\hat{u}\cdot\sigma)+\Phi_R(|u|)b_{R}(\hat{u}\cdot\sigma)
\end{multline}
with the following properties:
\begin{itemize}
\item [\it(i)] $b_{S}$ and $\Phi_{S}$ are smooth satisfying the assumptions of the previous section.
\item [\it(ii)] $b_{R}(s):=\frac{1}{4\pi}-b_{S}(s)$ is the angular remainder satisfying
$$C_{2,1,\eta,1}(b_R)\leq  \varepsilon \qquad \text{ and  } \qquad C_{2,2,\eta,1}(b_R)\leq  \varepsilon .$$
\item [\it(iii)] $\Phi_{R}(|u|)=|u|-\Phi_{S}(|u|)$ is the magnitude remainder satisfying $$\left\|\Phi_{R}\right\|_{L^{\infty}}\leq \dfrac{\varepsilon}{\left(C_{2,1,\eta,1}(b_S)+C_{2,2,\eta,1}(b_S)\right)}.$$
\end{itemize}
%
%
Notice that, in contrast to previous approaches, the last point is made possible because $\Phi_S(|u|)=|u|$ for large $|u|$ which makes $\Phi_R$ compactly supported. Thus, on the basis of relation \eqref{splitting}, one splits $\Q^+_e$ into the following four parts using obvious notations,
$$
\Q^+_e=\Q^+_{{SS}}+\Q^+_{SR}  + \Q^+_{RS}  + \Q^+_{RR}.
$$
We shall then deal separately with each of these parts. First, we know that
$$\|\Q^+_{SS}(f,g)\|_{\mathbf{H}^{s+1}_\eta}\leq C_{m,n,e}\|g\|_{\mathbf{H}^s_{\eta+\mu}}\|f\|_{L^1_{2\eta+\mu}} \qquad \forall m,n.$$
Second, let us estimate $\Q^+_{SR}$. Since
$$\partial^\ell\, \Q^+_{SR}(f,g)(v)=\sum_{\nu=0}^\ell \left(\begin{array}{c}
\l \\ \nu
\end{array}\right)
\ \Q^+_{SR}(\partial^\nu f,\partial^{\l-\nu}g)
$$
for any multi-index $\l$ with $|\l|\leq s+1$, one gets
$$\|\Q^+_{SR}(f,g)\|_{\mathbf{H}^{s+1}_\eta}^2 \leq C_s \sum_{|\l|\leq s+1}\sum_{\nu=0}^\l\left(\begin{array}{c}
\l \\ \nu
\end{array}\right)\|\Q^+_{{SR}}(\partial^\nu f,\partial^{\l-\nu}g)\|_{L^2_\eta}^2.$$
We treat differently the cases $|\l|=s+1$ and $|\l| < s+1$.  According to Theorem A. \ref{alogam} if $|\l|\leq s$ one has for any $|\nu|\leq |\l|$
\begin{equation*}\begin{split}
\|\Q^+_{{SR}}(\partial^\nu f,\partial^{\l-\nu}g)\|_{L^2_\eta}&\leq  C_{2,1,\eta,1}(b_R)\|\Phi_S\|_{L^\infty_{-1}}\|\partial^\nu f\|_{L^1_{\eta+1}}\,\|\partial^{\l-\nu}g\|_{L^2_{\eta+1}}\\
&\leq   \varepsilon \,\|\partial^\nu f\|_{L^1_{\eta+1}}\,\|\partial^{\l-\nu}g\|_{L^2_{\eta+1}}
\end{split}\end{equation*}
where we used the assumption (ii) with the fact that $\|\Phi_S\|_{L^\infty_{-1}} \leq 1$.  Recall the general estimate
\begin{equation}
\label{taug21}\|g\|_{L^1_k} \leq \tau_\theta\|g\|_{L^2_{k+3/2+\theta}}\qquad \forall k \geq 0, \qquad \forall \theta >0
\end{equation}
where the universal constant $\tau_\theta$ is given  by $\tau_\theta=\|\langle \cdot \rangle^{-\tfrac{3}{2}-\theta}\|_{L^2} < \infty$.  Taking for simplicity $\theta=1/2$ and since $|\l|\leq s$,
$$\sum_{|\l| < s+1}\sum_{\nu=0}^\ell \left(\begin{array}{c}
\l \\ \nu
\end{array}\right)\|\Q^+_{{SR}}(\partial^\nu f,\partial^{\l-\nu}g)\|_{L^2_\eta}\leq A_s \varepsilon \,\| f\|_{\mathbf{H}^s_{\eta+3}}\,\|g\|_{\mathbf{H}^s_{\eta+1}}$$
for some constant $A_s >0$ depending only on $s$.  In the case $|\ell|=s+1$, argue in the same way to obtain
$$\|\Q^+_{{SR}}(\partial^\nu f,\partial^{\l-\nu}g)\|_{L^2_\eta}\leq \varepsilon\| f\|_{\mathbf{H}^s_{\eta+3}}\,\|g\|_{\mathbf{H}^s_{\eta+1}}$$
for any $0 < |\nu| < |\l|.$ If $\nu=0$ one still has
$$\|\Q^+_{{SR}}(f,\partial^{\l}g)\|_{L^2_\eta}\leq  C_{2,1,\eta,1}(b_R)\| f\|_{L^1_{\eta+1}}\,\|\partial^\l g\|_{L^2_{\eta+1}}$$
additionally, for $\nu=\ell$ we use Theorem A. \ref{alogam} with $(p,q)=(2,1)$ to get
$$\|\Q^+_{{SR}}(\partial^{\l}f,g)\|_{L^2_\eta}\leq   C_{2,2,\eta,1}(b_R)\| g\|_{L^1_{\eta+1}}\,\|\partial^\l f\|_{L^2_{\eta+1}}.$$
Therefore,
\begin{multline*}
\|\Q^+_{SR}(f,g)\|_{\mathbf{H}^{s+1}_\eta} \leq\\
 A_s\,\varepsilon\left(\| f\|_{\mathbf{H}^s_{\eta+3}}\,\|g\|_{\mathbf{H}^s_{\eta+1}} +  \| g\|_{L^1_{\eta+1}}\,\|\partial^\ell f\|_{L^2_{\eta+1}}+\| f\|_{L^1_{\eta+1}}\,\| \partial^\ell g\|_{L^{2}_{\eta+1}}\right)
\ \ \  \forall |\l|=s+1.\end{multline*}
Third, argue in the same way using the smallness assumption (ii) to prove that
\begin{multline*}
\|\Q^+_{RR}(f,g)\|_{\mathbf{H}^{s+1}_\eta} \leq\\
 A_s\,\varepsilon\left(\| f\|_{\mathbf{H}^s_{\eta+3}}\,\|g\|_{\mathbf{H}^s_{\eta+1}} +  \| g\|_{L^1_{\eta+1}}\,\|\partial^\ell f\|_{L^2_{\eta+1}}+\| f\|_{L^1_{\eta+1}}\,\| \partial^\ell g\|_{L^{2}_{\eta+1}}\right)\ \ \ \forall |\l|=s+1.
 \end{multline*}
Finally, the estimate for $\Q^+_{RS}$ follows from the fact that $\|\Phi_R\|_{L^\infty}$ is small,
\begin{multline*}
\|\Q^+_{RS}(f,g)\|_{\mathbf{H}^{s+1}_\eta} \leq A_s\,\varepsilon\left(\| f\|_{\mathbf{H}^s_{\eta+2}}\,\|g\|_{\mathbf{H}^s_{\eta}} +  \| g\|_{L^1_{\eta}}\,\|\partial^\ell f\|_{L^2_{\eta}}+\| f\|_{L^1_{\eta}}\,\| \partial^\ell g\|_{L^{2}_{\eta}}\right)\\ \qquad \forall |\l|=s+1.
\end{multline*}
Combining all these estimates and replacing $A_s \varepsilon$ to $\varepsilon$  we get \eqref{regestim}.
\end{proof}

\begin{nb}\label{scaleCe} Recall that, by virtue of our scaling argument, we will have to apply the above regularity result for the scaled restitution coefficient $\el$. Arguing as in  \cite[Corollary 4.14]{AloLo1} we can prove without major difficulty that $\sup_{\la \in (0,1]} C_{\el} < \infty$ where $C_{\el}$ is the constant appearing in \eqref{regestim} for the scaled restitution coefficient $\el$.
\end{nb}

\section{Properties of the steady state}\label{sec:stea}

The purpose of this Section is to establish all the general \textit{a posteriori} properties of the family $(\gl)_{\la}$ of solutions to \eqref{rescaled} that will be necessary to establish the uniqueness result. Of course, this analysis will require fine properties of the collision operator $\Q_{\el}$ associated to the rescaled restitution coefficient $\el$, in particular, its behavior as $\la \to 0.$ In all this section, $\gl$ denotes any solution to \eqref{rescaled} with $\la \in [0,1]$. There is no loss in generality in assuming from now on that
$$\varrho=\IR \gl(v)\d v=1 \qquad \forall \la \in (0,1].$$
We shall define, for any $\la \in (0,1]$ the solution set:
\begin{multline}\label{frakSl}
\mathfrak{S}_\la=\bigg\{G_\la \in L^1_2\,; \gl  \text{ solution to \eqref{rescaled} with } \\
\IR \gl(v) \d v =1 \text{ and } \IR v \gl(v)\d v=0 \bigg\}.
\end{multline}
Recall that our choice of scaling  implies that for any $\gl \in  \mathfrak{S}_\la$, the energy identity is given by, see \eqref{scaleDE}
\begin{equation}\label{scaleDE1}
6=\dfrac{1}{\lambda^{3+\gamma}}\IRR \gl(v)\gl(\vb)\mathbf{\Psi}_e (\la^2|v-\vb|^2)\d v\d\vb \qquad \forall \la \in (0,1] \end{equation}
where $\mathbf{\Psi}_e$ has been defined in \eqref{Psie}. We deduce from \eqref{psie0} that, for \textit{fixed} $r >0$,
$$\dfrac{1}{\lambda^{3+\gamma}}\mathbf{\Psi}_{e}(\la^2 r^2) \simeq \frac{\mathfrak{a}}{4+\gamma} r^{3+\gamma} \qquad \text{as} \qquad \la \simeq 0.$$
Intuitively, one gets then that, for $\lambda \simeq 0$,
$$6\simeq \frac{\mathfrak{a}}{4+\gamma} \IRR \gl(v)\gl(\vb)|v-\vb|^{3+\gamma}\d v\d\vb.$$
 The use of Jensen's inequality proves that the moment of order $3+\gamma$ of $\gl$ remains bounded uniformly with respect to $\lambda$
$$\m_{\frac{3+\gamma}{2}}(\la) =\mathrm{O}(1)$$
where the moments are defined as
\begin{equation}\label{mp}
\m_{p}(\la)=\IR \gl(v)|v|^{2p}\d v \qquad p \geq1.
\end{equation}
Existence of higher moments for $G_{\lambda}$ is the objective of the following section, see also Lemma A. \ref{lem:dissTem} in the Appendix, which properly justify above computations.

\subsection{Moment estimates}

Recall that our choice of scaling is such that
$$\sup_{0\leq \la \leq 1}\m_1(\la)=\sup_{0 \leq \la \leq 1}\E_\la=\E_\textrm{max}< \infty.$$
By a simple induction argument, this actually implies that all the moments of $\gl$ are uniformly bounded.
\begin{propo}\label{mo}
For any $p \geq 0$, there exists $\mathbf{C}_p >0$ such that
$$\sup_{0\leq \la \leq 1}\m_p(\la) \leq \mathbf{C}_p.$$
\end{propo}
\begin{proof} Let $p \geq 1$ be fixed. Multiplying Eq. \eqref{rescaled} by $\psi(v)=|v|^{2p}$ and integrating over $\R^3$ we get
$$-\lambda^\gamma \IR \gl(v)\Delta |v|^{2p}\d v =\IR \Q_{\el}(\gl,\gl)(v)\,|v|^{2p}\d v.$$
Since $\Delta |v|^{2p}=2p(2p+1)|v|^{2p-2}$, using Lemma B. \ref{momentsp}, there are two positive constants $k_p, A_p >0$ independent of $\la$ such that
$$-2p(2p+1)\la^\gamma \m_{p-1}(\la) \leq - k_p \varrho \m_{p+\tfrac{1}{2}}(\la)+A_p\,\m_{\tfrac{1}{2}}(\la)\,\m_{p}(\la) \qquad \forall \la \geq 0.$$
Since $\m_{\tfrac{1}{2}}(\la) \leq \sqrt{\E_\la}\leq \sqrt{\E_\textrm{max}}$ for any $\la \in (0,1)$, we see that there are two positive constants $C_{1,p},\,C_{2,p} >0$ independent of $\la$ such that
$$\m_{p+\tfrac{1}{2}}(\la) \leq  C_{1,p}\m_p(\la)+C_{2,p}\m_{p-1}(\la)  \qquad \forall \la \in (0,1],\:\forall p \geq 1.$$
Both $\sup_{\la \in (0,1)}\m_{1}(\la)=\E_\mathrm{max}$ and $\sup_{\la \in (0,1]}\m_0(\la)=1$ are finite, thus, a simple induction yields the conclusion for any $p \in \mathbb{N}$. The result extends then to any parameter $p \geq 0$ by interpolation.
\end{proof}
\begin{propo}\label{prop:inf}
There exist  $\E_\mathrm{min} >0$ and $c_0 >0$ such that $\inf_{\la \in [0,1]}\E_\la=\E_\mathrm{min} >0$
and
$$\IR \gl(\vb)|v-\vb|\d\vb \geq c_0\langle v \rangle \qquad \forall v \in \R^3, \:\qquad \forall \lambda \in [0,1].$$
\end{propo}
\begin{proof} Note that $\mathbf{\Psi}_e(x) \simeq C x^{3/2}$  as $x \to \infty$ for some positive $C >0$.  Additionally, using \eqref{psie0}, there exists a positive constant $K >0$ such that $\mathbf{\Psi}_e(r^2) \leq K r^{3+\gamma}$ for any $r >0$.  According to \eqref{scaleDE1}, it follows
$$6\leq K \IRR \gl(v)\gl(\vb)|v-\vb|^{3+\gamma}\d v \d\vb.$$
Therefore,
\begin{equation}\label{borneinf}
\inf_{\la >0}\m_{\tfrac{3+\gamma}{2}}(\la)=c >0.
\end{equation}
Knowing \eqref{borneinf}, it is a standard procedure to deduce the result from Proposition \ref{mo}.
\end{proof}
\begin{propo}\label{propo:tails} There exist positive constants $A >0$ and $M >0$ such that any solution $\gl$ to \eqref{rescaled}, with $\la \in (0,1]$, satisfies
\begin{equation}\label{exponen}\int_{\R^3} \gl(v)\exp\left(A \,|v|^{\frac{3}{2}}\right)\d v \leq M.\end{equation}
\end{propo}
\begin{proof} The proof follows the lines of the analogous result \cite[Theorem 1]{BoGaPa} for constant restitution coefficient. It consists in proving that there exist $K >0$ such that
\begin{equation}\label{mpRp}
\sup_{\la \in (0,1]}\m_p(\la) \leq \Gamma\left(\frac{4}{3}p + \frac{1}{2}\right)\,K^p \qquad \forall p \geq 1
\end{equation}
where $\Gamma(\cdot)$ is the gamma function while $\m_p(\la)$ is defined in \eqref{mp}. In order to prove \eqref{mpRp} note that,
$$-2p(2p+1)\la^\gamma \m_{p-1}(\la)=\int_{\R^3}\Q_{\el}(\gl,\gl)(v)|v|^{2p}\d v \qquad \forall p \geq 1,\:\la \in (0,1].$$
One can estimate the right side thanks to \cite[Proposition 2.7]{AloLo1},
$$\int_{\R^3}\Q_{\el}(\gl,\gl)(v)|v|^{2p}\d v \leq -(1-\kappa_p)\m_{p+\frac{1}{2}}(\la)+\kappa_p \mathcal{S}_p(\la)$$
where
$$\mathcal{S}_p(\la)=\sum^{[\frac{p+1}{2}]}_{k=1}\left(
\begin{array}{c}
p\\k
\end{array}
\right)\left(\m_{k+1/2}(\la)\;\m_{p-k}(\la)+\m_{k}(\la)\;\m_{p-k+1/2}(\la)\right).$$
Here $[\frac{p+1}{2}]$ denotes the integer part of $\frac{p+1}{2}$ and $\kappa_p \in (0,1)$ is independent of $\la$ and satisfies  $\kappa_p =\mathrm{O}(1/p)$ as $p \to \infty.$ Then, one sees that \cite[Equations (4.6) and (4.11)]{BoGaPa} hold with $\mu=\la^\gamma \in (0,1]$. At this point, we can resume exactly the proof of \cite{BoGaPa} noticing that all the estimates there are uniform with respect to the coefficient $\mu$ appearing in front of the thermal bath. In other words, we obtain \eqref{mpRp} with a positive constant $K >0$ which is independent of $\la$.  This is enough to get \eqref{exponen}.
\end{proof}
One actually can make more precise the above estimates by evaluating the difference of two solutions to \eqref{rescaled}.  A simple adaptation of \cite[Proposition 2.7, Step 1]{MiMo3} gives the following estimate.
\begin{propo}\label{prop:momdiff}
For any $s \in [0,\frac{3}{2}]$ there exist some positive constants $r_s >0$ and $M_s >0$ such that
\begin{equation}\label{exponenDiff}
\int_{\R^3}\left| \gl(v)-F_\la(v)\right|\exp\left(r_s \,|v|^{s}\right)\d v \leq M_s\|\gl-F_\la\|_{L^1_1} \quad \forall \la \in (0,1]
\end{equation}
for any $F_\la, \gl \in \mathfrak{S}_\la$.
\end{propo}

\subsection{Sobolev estimates} We prove now that the family $(\gl)_\la$ is uniformly bounded in any Sobolev norms $\mathbf{H}^\l$. We begin showing uniform $L^2_k$-estimates of $\gl$ for sufficiently small $\la$.
\begin{propo}\label{propo:L2} For any $k \geq 0,$  one has
$\mathbf{A}_k:=\sup_{\lambda\in (0,1]}\|\gl\|_{L^2_{k}}< \infty.$
 \end{propo}
\begin{proof}
First, observe that for any test function $\psi(v)$ integration by parts yields
$$-\IR \Delta \gl(v) \gl(v)\psi(v)\d v=\IR |\nabla \gl(v)|^2 \psi(v)\d v -\frac{1}{2}\IR \gl(v)^2 \Delta \psi(v)\d v.$$
Fix $k \geq 0$  and multiply equation \eqref{rescaled} by $\gl(v)  \langle v \rangle^{2k}$.  Apply above identity to $\psi(v)=\langle v \rangle^{2k}$ and use the inequality $\Delta \psi(v) \leq 2k(2k+1)\langle v\rangle^{2(k-1)}$ to obtain
\begin{equation}\label{L2k}
\lambda^{\gamma}\|\nabla \gl\|^2_{L^2_k}\leq \IR \Q_{\el}(\gl,\gl)(v)\gl(v)\langle v \rangle^{2k}\d v +(2k^2+k)\lambda^{\gamma}\|\gl\|^2_{L^2_{k-1}}.
\end{equation}
Applying \cite[Corollary 4.14]{AloLo1} with $p=2$ and $\eta=k$, we see that there exist $\theta \in (0,1),$ $z >0$ and $C_{\el} >0$ depending only on $\el$ such that for any $\delta >0$,
\begin{equation*}
\IR \Q_{\el}^+(\gl,\gl)(v)\gl(v)\langle v \rangle^{2k}\d v\leq C_{\el}\delta^{-z}\|\gl\|^{1+2\theta}_{L^1_k}\|\gl\|^{2(1-\theta)}_{L^2_k}+\delta\|\gl\|_{L^1_{2+k}}\|\gl\|^{2}_{L^2_{k+1/2}}.
\end{equation*}
Same reasoning as in \cite[Corollary 4.15]{AloLo1} shows that \footnote{With the notation of \cite[Corollary 4.15]{AloLo1}, one can prove that for any compact interval $I \subset (0,\infty)$ it follows that $\max_{k=0,1}\|D^k G_{\el}(\cdot)\|_{L^\infty(I)}=\mathrm{O}(1)$ as $\la \simeq 0$ where $G_{\el}(r)=\dfrac{r}{(1+\vartheta_{\el}'(r))\beta_{\el}(r)}$. In particular, $\lim_{\la \to 0} C_{\el}=C_0 >0$.}
$$\sup_{\la \in (0,1)}C_{\el} < \infty.$$
Therefore, there exist $\theta \in (0,1)$ such that for any $\la \in (0,1)$ and $\delta >0$ one can find some $K_\delta >0$ independent of $\la$ for which it holds
\begin{multline*}
\IR \Q_{\el}^+(\gl,\gl)(v)\gl(v)\langle v \rangle^{2k}\d v\\
\leq K_\delta\|\gl\|^{1+2\theta}_{L^1_k}\|\gl\|^{2(1-\theta)}_{L^2_k}+\delta\|\gl\|_{L^1_{2+k}}\|\gl\|^{2}_{L^2_{k+1/2}} \ \ \ \forall \la \in (0,1).
\end{multline*}
Second, estimate the loss term thanks to Proposition \ref{prop:inf}. Indeed,
\begin{equation*}\begin{split}
\IR \Q_{\el}^-(\gl,\gl)(v)\gl(v)\langle v \rangle^{2k}\d v&=\IRR |v-\vb|\gl^2(v)\gl(\vb)\langle v \rangle^{2k}\d v\d\vb\\
&\hspace{1cm} \geq c_0\IR \gl^2(v)\langle v \rangle^{2k+1}\d v=c_0\|\gl\|_{L^2_{k+1/2}}^2.\end{split}
\end{equation*}
Thus, plugging the previous two estimates into \eqref{L2k}
\begin{multline*}
\lambda^{\gamma}\|\nabla \gl\|^{2}_{L^2_k}\leq K_\delta\|\gl\|^{1+2\theta}_{L^1_k}\|\gl\|^{2(1-\theta)}_{L^2_k}+\left(\delta\|\gl\|_{L^1_{2+k}}-c_0\right)\|\gl\|^{2}_{L^2_{k+1/2}} \\ +(2k^2+k)\lambda^{\gamma}\|\gl\|^2_{L^2_{k-1}} \qquad \forall \la \in (0,1].
\end{multline*}
Using the notation of Proposition \ref{mo} and choosing $\delta={c_0}/{2\mathbf{C}_{2+k}}$, one sees that there exists $C_k=K_\delta \mathbf{C}_k^{1+2\theta} >0$ such that
$$\lambda^{\gamma}\|\nabla \gl\|^{2}_{L^2_k}\leq C_k \|\gl\|^{2(1-\theta)}_{L^2_k}-\frac{c_0}{2}\|\gl\|^{2}_{L^2_{k+1/2}} +(2k^2+k)\lambda^{\gamma}\|\gl\|^2_{L^2_{k-1}} \qquad \forall \la \in (0,1].$$
In particular,
$$ \frac{c_0}{2}\|\gl\|^{2}_{L^2_{k+1/2}}\leq C_k \|\gl\|^{2(1-\theta)}_{L^2_k} +(2k^2+k)\lambda^{\gamma}\|\gl\|^2_{L^2_{k-1}}.$$
The case $k=0$ follows directly from this estimate, i.e.
\begin{equation}\label{L20}
\sup_{\la \in (0,1]} \|\gl\|_{L^2} < \infty.
\end{equation}
Assume now $k \geq 1$. For any $R >0$, it can be checked that
$$\|\gl\|_{L^2_{k-1}}^2 \leq R^{2k-2}\|\gl\|_{L^2}^2+R^{-3}\|\gl\|_{L^2_{k+1/2}}^2.$$
Hence, choosing $R=\left(\frac{4}{c_0}\la^{\gamma}(2k^2+k)\right)^{1/3}$ we get
$$ \frac{c_0}{4}\|\gl\|^{2}_{L^2_{k+1/2}}\leq C_k \|\gl\|^{2(1-\theta)}_{L^2_k}+B_k(\la)\|\gl\|^2_{L^2}$$
with $B_k(\la)= (2k^2+k)\lambda^{\gamma}R^{2k-2}$.  In particular, using \eqref{L20} there exists some positive constant $A_k:=B_k(1)\sup_{\la \in (0,1)}\|\gl\|^2_{L^2} >0$, independent of $\la$, such that
$$ \frac{c_0}{4}\|\gl\|^{2}_{L^2_{k+1/2}}\leq A_k+ C_k \|\gl\|^{2(1-\theta)}_{L^2_k} \qquad \forall \la \in (0,1].$$
This yields the result.
\end{proof}
Since $\gl \in L^2_{1} \cap L^1_{1}$, Theorem A. \ref{alogam} shows that $\Q_{\el}(\gl,\gl) \in L^2$.  The equation
\begin{equation}\label{steady2}
-\la^\gamma \Delta \gl=\Q_{\el}(\gl,\gl)
\end{equation}
implies that $\Delta \gl \in L^2$.  Thus, a bootstrap argument shows the smoothness of $\gl$.  This reasoning will not help to find $\la$-uniform Sobolev estimates since the diffusive heating in \eqref{steady2} will vanish in the formal limit $\la \to 0$.

\begin{theo}\label{theo:HKL} Assume that $e(\cdot)$ belongs to the class $\mathbb{E}_m$ for some integer $m \geq 2$. Then, for any $k \geq 0$ and integer $\ell  \in [0,m-1]$
$$\sup_{\la \in (0,1]} \|\gl\|_{\mathbf{H}^\ell_k} < \infty.$$
In particular, if $e(\cdot)$ belongs to the class $\mathbb{E}_m$ with $m \geq 3$ one has $\sup_{\la \in (0,1]}\|\gl\|_{L^\infty_k} < \infty$ for any $k \geq 0$.
\end{theo}

\begin{proof}
Use induction over $|\ell| \in \mathbb{N}$. Proposition \ref{propo:L2} shows that the result is true if $|\ell|=0.$ Let then $|\ell|:=s+1>0$ be fixed and assume that  for any $k\geq 0$ there exists  $C_k >0$ such that
\begin{equation}\label{Hsk}
\max_{|\nu|\leq s}\sup_{\lambda\in (0,1]}\|\partial^{\nu}\gl\|_{L^2_k}\leq C_k.
\end{equation}
Observe that differentiating $\ell$-times Equation \eqref{scaleDE} yields
\begin{equation*}
-\lambda^{\gamma}\Delta\partial^{\l}\gl=\partial^{\l}\Q_{\el}(\gl,\gl).
\end{equation*}
Multiplying this equation by $\partial^\l \gl(v) \langle v\rangle^{2k} $ and integrating over $\R^3$  we get, as in Proposition \ref{L2k}
\begin{equation}\label{be}
\lambda^{\gamma}\|\nabla \partial^{\l }\gl\|^2_{L^2_k} \leq  \int_{\mathbb{R}^{3}}
 \partial^\ell \Q_{\el}(\gl,\gl)(v)\partial^{\l}\gl(v)\langle v\rangle^{2k}\d v + (2k^2+k)\lambda^{\gamma}\|\partial^{\l}\gl\|^{2}_{L^2_k}.
\end{equation}
Fix $k \geq \frac{1}{2}$. One has
\begin{equation*}\begin{split}
\IR \partial^\ell \Q^+_{\el}(\gl,\gl)(v)\partial^{\l}\gl(v)\langle v\rangle^{2k}\d v &\leq \|\partial^\ell \Q^+_{\el}(\gl,\gl)\|_{L^2_{k-\frac{1}{2}}}\|\partial^\l \gl\|_{L^2_{k+\frac{1}{2}}}\\
&\leq \|\Q^+_{\el}(\gl,\gl)\|_{\mathbf{H}^{s+1}_{k-\frac{1}{2}}}\|\partial^\l \gl\|_{L^2_{k+\frac{1}{2}}}\end{split}\end{equation*}
since $|\ell|=s+1$.  One estimates the Sobolev norm of $\Q^+_{\el}(\gl,\gl)$ thanks to Theorem \ref{regularite} applied to $\eta=k-\frac{1}{2}$. Precisely, for any $\varepsilon >0$,
\begin{multline*}
\|\Q^+_{\el}(\gl,\gl)\|_{\mathbf{H}^{s+1}_{k-\frac{1}{2}}} \leq C(\varepsilon)\,\|\gl\|_{\mathbf{H}^s_{2k+s+3}}\|\gl\|_{L^1_{2k+s+3}} + \varepsilon\| \gl\|_{\mathbf{H}^s_{k+\frac{5}{2}}}\,\|\gl\|_{\mathbf{H}^s_{k+\frac{1}{2}}}\\
+\varepsilon\left( \|\gl\|_{L^1_{k+\frac{1}{2}}}\,\|\partial^\ell \gl\|_{L^2_{k+\frac{1}{2}}}+\| \gl\|_{L^1_{k+\frac{1}{2}}}\,\| \partial^\ell \gl\|_{L^{2}_{k+\frac{1}{2}}}\right).
\end{multline*}
Using the uniform bounds in $\mathbf{H}^s_k$ given by \eqref{Hsk} together with Proposition \ref{propo:L2} and the uniform $L^1_k$ bounds, one notes that there exist $\alpha_k,\beta_k  >0$ such that
$$\|\Q^+_{\el}(\gl,\gl)\|_{\mathbf{H}^{s+1}_{k-\frac{1}{2}}} \leq \alpha_k +  \varepsilon\,\beta_k\,  \|\partial^\ell \gl\|_{L^{2}_{k+\frac{1}{2}}} \qquad \qquad \forall \la \in (0,1].$$
Therefore,
\begin{equation}\label{Q+el}
\IR \partial^\ell \Q^+_{\el}(\gl,\gl)(v)\partial^{\l}\gl(v)\langle v\rangle^{2k}\d v \leq  \alpha_k\|\partial^\ell \gl\|_{L^{2}_{k+\frac{1}{2}}} + \varepsilon\,\beta_k\, \|\partial^\ell \gl\|_{L^{2}_{k+\frac{1}{2}}}^2.
\end{equation}
Regarding the loss part of the collision operator, first note that
$$\partial^{\l}\Q^{-}_{\el}(\gl,\gl)=\sum^{\l}_{\nu=0}
\left(\begin{array}{c}
\l \\ \nu
\end{array}\right)
\Q_{\el}^{-}(\partial^\nu \gl,\partial^{\l-\nu}\gl).$$
For any $|\nu|\neq |\ell|$, integration by parts yields
\begin{equation*}\begin{split}
\left|\Q^{-}_{\el}(\partial^\nu \gl,\partial^{\l-\nu}\gl)(v)\right|&=\left|\partial^\nu \gl(v)\right|\,\left|\IR \partial^{\l-\nu} \gl(\vb)|v-\vb|\d\vb\right|\\
&\leq \left|\partial^\nu \gl(v)\right|\,\|\partial^{\l-\nu-1}\gl\|_{L^1}
\end{split}\end{equation*}
where, $\ell-1=(\ell_1-1,\ell_2,\ell_3)$ for any multi-index $\ell=(\ell_1,\ell_2,\ell_3)$. Using again the control of $L^1$ norms by weighted $L^2$-norms, see inequality \eqref{taug21}, we get
$$\left|\Q^{-}_{\el}(\partial^\nu \gl,\partial^{\l-\nu}\gl)(v)\right|\leq \tau\,|\partial^\nu \gl(v)|\,\|\partial^{\l-\nu-1}\gl\|_{L^2_2}$$
for some universal constant $\tau >0$ independent of $\la$.  From the induction hypothesis \eqref{Hsk}, this last quantity is uniformly bounded and using Cauchy-Schwarz inequality we obtain
\begin{multline*}
\sum_{|\nu| < |\ell|}
\left(\begin{array}{c}
\l \\ \nu
\end{array}\right)\IR  \Q^{-}_{\el}(\partial^\nu \gl,\partial^{\l-\nu}\gl)(v) \partial^{\l}\gl(v)\langle v\rangle ^{2k}\d v
\\
\leq C_2\sum_{|\nu|< |\ell|}
\left(\begin{array}{c}
\l \\ \nu
\end{array}\right)\|\partial^\nu \gl\|_{L^2_k}\,\|\partial^{\l}\gl(v)\|_{L^2_k}
\leq C_{k,\l}\|\partial^{\l}\gl\|_{L^2_k} \qquad\forall \la \in (0,1]
\end{multline*}
for some positive constant $C_{k,\l}$ independent of $\la$.  Second, whenever $\nu=\ell$ we have according to Proposition \ref{prop:inf} the lower bound
\begin{equation*}
\int_{\mathbb{R}^{3}}\Q^-(\partial^{\l}\gl,\gl)(v)\,\partial^{\l}\gl(v)\langle v\rangle ^{2k}\d v\geq c_0\|\partial^{\l}\gl\|^{2}_{L^2_{k+\frac{1}{2}}}.
\end{equation*}
Thus, summarizing,  inequality \eqref{be} reads
\begin{multline*}
\lambda^{\gamma}\|\nabla \partial^{\l }\gl\|^2_{L^2_k} \leq C_{k,\l}\|\partial^\ell \gl\|_{L^2_k} + \alpha_k\|\partial^\ell \gl\|_{L^{2}_{k+\frac{1}{2}}} + \varepsilon\,\beta_k\, \|\partial^\ell \gl\|_{L^{2}_{k+\frac{1}{2}}}^2\\
+C_k\|\partial^{\l}\gl\|_{L^2_k}-c_0\|\partial^{\l}\gl\|^{2}_{L^2_{k+\frac{1}{2}}} + (2k^2+k)\lambda^{\gamma}\|\partial^{\l}\gl\|^{2}_{L^2_k} \qquad \forall \la \in (0,1].
\end{multline*}
Choose $\varepsilon >0$ such that $\varepsilon\,\beta_k=\frac{c_0}{2}$.  We note that, after neglecting the gradient term in the above left side and bounding all $L^2_k$ norms by $L^2_{k+\frac{1}{2}}$ norms, there exists some positive constant $A_k >0$ such that
\begin{equation*}
\frac{c_0}{2}\|\partial^{\l}\gl\|^{2}_{L^2_{k+\frac{1}{2}}} \leq A_k\|\partial^\ell \gl\|_{L^{2}_{k+\frac{1}{2}}}+ (2k^2+k)\lambda^{\gamma}\|\partial^{\l}\gl\|^{2}_{L^2_k} \qquad \forall \la \in (0,1].
\end{equation*}
Finally, following the proof of Proposition \ref{propo:L2}, we get that $\sup_{\la \in (0,1]}\|\partial^\ell \gl\|_{L^2_{k+\frac{1}{2}}} < \infty$ for any $k \geq\frac{1}{2}$.
\end{proof}
\begin{nb} In the constant restitution case, uniform regularity estimates where obtained using the propagation of regularity and damping with time of singularities for solution to the  time-dependent problem.   More precisely, using the fact that the solution to
$$\partial_t f(t,v)=\Q_\alpha(f,f)(t,v)+(1-\alpha)\Delta_v f(t,v)$$
can be written as $f(t,v)=f_{S}(t,v)+f_{R}(t,v)$ where $f_{S}$ is smooth and the reminder $f_{R}$ is small in some appropriate norm, see  \cite{MoVi}.  Our approach applies to such case yielding a much more direct proof of these estimates.
\end{nb}
The proof of Theorem \ref{theo:HKL} can be easily modified to get an estimate of the difference of solutions to \eqref{rescaled}.
\begin{propo}\label{prop:Hkdiff}
Assume that $e(\cdot)$ belongs to $\mathbb{E}_m$ for some $m \geq 3$. For any integer $\ell \in [0,m-1]$,  there exist some positive constant  $C_\ell >0$  such that
\begin{equation}\label{hkk}
\|F_\la-G_\la\|_{\mathbf{H}^\ell} \leq C_\ell \|F_\la-G_\la\|_{L^1_1} \qquad \forall \la \in (0,1]
\end{equation}
for any $\gl,\fe \in \mathfrak{S}_\la$. As a consequence, there exists a positive constant $C_a >0$ such that
\begin{equation}\label{Deltama}
\| \fe-\gl \|_{\mathbf{W}^{\ell,1}(m_a)} \leq C_a \|\fe-\gl\|_{L^1(m_a)} \qquad \forall a\in[0,\tfrac{3}{2}],\ \ \la \in (0,1]
\end{equation}
where the weight $m_a:=m_a(v)=\exp(a|v|)$.
 \end{propo}
\begin{proof} We follow the argument of \cite[Proposition 2.7]{MiMo3} that uses induction and only give the details for the initial step $\ell=0$. Set $H_\la=F_\la-G_\la$, we aim therefore to control $\|H_\la\|_{L^2}$ by $\|H_\la\|_{L^1_1}$. Notice that $H_\la$ satisfies
$$\Q_{\el}(H_\la,F_\la)+\Q_{\el}(G_\la,H_\la)=-\la^\gamma \Delta H_\la, \qquad \forall \la \in [0,1].$$
Multiplying this identity by $H_\la$ and integrating over $\R^3$ yields
\begin{multline*}
\la^\gamma \|\nabla H_\la\|_{L^2}^2 + \IR \Q^-(H_\la,F_\la)H_\la \d v \\
= \IR \left(\Q_{\el}^+(H_\la,F_\la)+\Q_{\el}^+(G_\la,H_\la)\right)H_\la(v)\d v -\IR \Q^-(G_\la,H_\la)H_\la\d v.\end{multline*}
From Proposition \ref{prop:inf} one has
$$ \IR \Q^-(H_\la,F_\la)H_\la \d v \geq c_0  \|H_\la\|^2_{L^2_{\frac{1}{2}}}.$$
In addition,
$$\IR \Q^-(G_\la,H_\la)H_\la\d v \leq \IR  G_\la(v)|H_\la(v)|\d v \IR |H_\la(\vb)|\,|v-\vb|\d \vb \leq C \|H_\la\|^2_{L^1_1}$$
where the constant $C$ depends only on the $L^\infty$ norm of $\gl$ which is uniformly bounded.  In order to control the gain operator, split the angular kernel $b(s)=\frac{1}{4\pi}$ into $b(s)=b_{1}(s)+b_{2}(s)$ with $b_{1}(s):=\frac{1}{ 4\pi} 1_{(-1+\delta,1-\delta)}(s)$ for some $\delta >0$ to be determined latter on. Using Young's inequality, see Theorem A. \ref{alo}
\begin{multline*}
\IR \left(\Q_{\el}^+(H_\la,F_\la)+\Q_{\el}^+(G_\la,H_\la)\right)H_\la(v)\d v \\
\leq \|\Q_{\el,b_{1}}^+(H_\la,F_\la)+\Q_{\el,b_1}^+(G_\la,H_\la)\|_{L^\infty} \|H_\la \|_{L^1}\\
+ \|\Q_{\el,b_2}^+(H_\la,F_\la)+\Q_{\el,b_2}^+(G_\la,H_\la)\|_{L^2_{-\frac12}} \|H_\la \|_{L^2_{\frac12}}\\
\leq C(b_{1})\left(\|F_\la\|_{ L^{\infty}_{1} }+\|G_\la\|_{ L^{\infty}_{1} } \right) \|H_\la \|^2_{L^1_1}+C(b_2)\left(\|F_\la\|_{L^{1}_{1}}+\|G_\la\|_{L^{1}_{1}}\right) \|H_\la \|^2_{L^2_{\frac12}}.
\end{multline*}
Fix $\varepsilon>0$, from the explicit expression of both $C(b_1)$ and $C(b_2)$ provided by Theorem A. \ref{alo} one notes that it is possible to choose $\delta>0$ such that $C(b_{2})\leq\varepsilon$ (recall that $b$ is bounded). Summarizing, for any $\varepsilon>0$
$$\la^\gamma \|\nabla H_\la\|_{L^2}^2 + c_0 \|H_\la\|^2_{L^2_{\frac12}} \leq \varepsilon  \|H_\la\|_{L^2_{\frac12}}^2 + C(\varepsilon)\|H_\la\|^{2}_{L^1_1}$$
where $C(\varepsilon)$ is a positive constant independent of $\la$. Choosing $\varepsilon =\frac{c_0}{2}$ we deduce that
$$\|H_\la\|^2_{L^2_{\frac12}} \leq \frac{2C(\varepsilon)}{c_0}\|H_\la\|_{L^1_1}^2 \qquad \forall \la \in (0,1]$$
which gives the result for $\ell=0$. To extend these estimates to higher order derivatives, one proceeds by induction using Theorem \ref{theo:HKL} yielding \eqref{hkk}. To deduce now estimate \eqref{Deltama}, recall the interpolation inequality given in \cite[Appendix B]{MiMo3} (note a misprint in the \textit{op. cit.} where the exponents $1/8$ have been replaced by $1/4$):  For any  $a \geq 0$ and $\ell \geq 0$, there exist $C(a,\ell) >0$  such that
$$\| h \|_{\mathbf{W}^{\ell,1}(m_a)} \le C \, \| h \|_{\mathbf{H}^{\ell_0}}^{1/8} \, \,\| h \|_{L^1( m_{b})}^{1/8} \,  \| h \|_{L^1(m_a)}^{3/4} \qquad \forall \: h \in \mathbf{H}^{\ell_0} \cap L^1(m_b)$$
where $\ell_0 := 8\ell+ \frac{35}{2}$ and $b=12a$.  According to Proposition \ref{prop:momdiff} there exists some $c> 0$ such that $\|H_\la\|_{L^1(m_b)} \leq c\|H_\la\|_{L^1_1}$ for any  $\la   \in  (0,1]$.  Moreover, \eqref{hkk} implies that the $\mathbf{H}^{\ell_0}$-norm of $H_\la$ can be controlled from above by $\|H_\la\|_{L^1_1}$. Combining these estimates we get
$$\|H_\la\|_{\mathbf{W}^{\ell,1}(m_a)} \leq C(a,\ell)\|H_\la\|_{L^1_1}^{1/4} \|H_\la\|_{L^1(m_a)}^{3/4} \quad \forall \la \in (0,1]$$
which yields the desired conclusion.
\end{proof}
\subsection{Continuity properties of $\Q_{\el}^+$ as $\la \to 0$}\label{sec:continuity}
We investigate in this section the continuity of the gain part $\Q_{e}^+(f,g)$ with respect to the restitution coefficient.  We shall prove that, for sufficiently smooth functions $f$ and $g$, the collision operator $\Q_{\el}^+(f,g)$ converges strongly towards $\Q_1^+(f,g)$ as $\lambda \to 0$ in a suitable norm to be specified.
\begin{propo} \label{exponetialest}
For any $k\geq0$, there exist some explicit constants $C(\gamma,k)$ and $\tilde{C}(\gamma,k)$ such that \begin{align}\label{H-1elast}
\|\left(\Q^+_{\el}(f,g)-\Q^+_1(f,g)\right) \langle v \rangle^{k} \|_{\mathbf{H}^{-1}}
&\leq C(\gamma,k,a)\lambda^{\gamma}\|f \|_{L^1_{k+\gamma+2}}\,\|g\; \|_{L^2_{k+\gamma+2}}\\
\|\left(\Q^+_{\el}(f,g)-\Q^+_1(f,g)\right)  \langle v \rangle^{k} \|_{\mathbf{H}^{-1}}
&\leq \tilde{C}(\gamma,k )\lambda^{\gamma}\|f \|_{L^2_{k+\gamma+2}}\,\|g \|_{L^1_{k+\gamma+2}}
\end{align}
\end{propo}
\begin{proof} Fix $\la >0$, a test function $\phi \in \mathbf{H}^1$ and define $\psi(v)=\langle v \rangle^{k}  \phi(v)$.  Use the weak form of $\Q^+_{\el}(f,g)-\Q_1^+(f,g)$ to get
\begin{multline*}
\IR \left(\Q_{\el}^+(f,g)-\Q_1^+(f,g)\right)(v)\psi(v)\d v\\
=\dfrac{1}{2\pi}\int_{\R^3 \times \R^3 \times \S} |u \cdot \n|f(v)g(\vb)\bigg(\psi(v^{(\la)})+\psi(\vb^{(\la)})-\psi(v')-\psi(\vb')\bigg)\d v \d\vb \d\n
\end{multline*}
where $(v^{(\la)},\vb^{(\la)})$ denotes the post-collisional velocities associated to the restitution coefficient $\el$ while $(v',\vb')$ denotes the post-collisional velocities for elastic interactions, that is,
\begin{align*}
v'=v-(u\cdot \n)\n\ , &\quad \vb'=\vb+\,(u\cdot \n)\n\\
v^{(\la)}=v-\beta_\la\,(u\cdot \n)\n\ , &\quad \vb^{(\la)}=\vb+\beta_\la\,(u\cdot \n)\n
\end{align*}
with $\beta_{\la}=\beta_{\la}(|u\cdot\n|)=\frac{1+\el(|u \cdot \n|)}{2}$.  Set
$$I_\la=\int_{\R^3 \times \R^3 \times \S} |u \cdot \n|f(v)g(\vb)\big(\psi(v^{(\la)})-\psi(v')\big)\d v \d\vb \d\n$$
and define
$$\zeta=\zeta(u,\n,\la)=v^{(\la)}-v'=\frac{1-\el}{2}\,(u\cdot \n)\n.$$
According to Assumption \ref{HYP},  $$\ell_\gamma(e)=\sup_{r >0} \dfrac{1-e(r)}{r^\gamma} < \infty,$$
thus, for any $u,\n,\la$,
$$|\zeta| \leq \dfrac{\ell_\gamma(e)}{2} \la^\gamma |u \cdot \n|^{\gamma+1}.$$
Moreover,  for any fixed $v,\vb,\n$,
$$\psi(v^{(\la)})-\psi(v')=\int_0^1 \nabla \psi (v'+s\,\zeta) \cdot \zeta \d s.$$
These two observations lead to
$$I_\la \leq  \dfrac{\ell_e(\gamma)}{2} \la^\gamma  \int_{\R^3 \times \R^3 \times \S}\d v \d\vb \d\n \int_0^1   |u \cdot \n|^{\gamma+2} f(v)g(\vb)\left|\nabla \psi (v'+s\,\zeta)\right|\d s.$$
At this point it is important to recognize that for any fixed $s\in (0,1]$ the integral
$$\int_{\R^3 \times \R^3 \times \S} |u \cdot \n|^{\gamma+2} f(v)g(\vb)\left|\nabla \psi (v'+s\,\zeta)\right|\d v \d\vb \d\n $$
is just the weak form of the gain part of some peculiar Boltzmann-like operator.  Indeed, set $\varphi(v)=|\nabla \psi(v)|$ and $V'_s=v'+s\zeta$ (notice that $V'_s$ depends on $u,\n,\la$ and $s$) and observe that
$$V'_s=v- \widetilde{\beta}_s\,(u\cdot \n)\n$$
for some parameter
$$\widetilde{\beta}_s=\widetilde{\beta}_s(|u\cdot \n|)=(1-s)+s \beta_{\la}(|u\cdot \n|) \in (1/2,1].$$
Therefore, $V'_s$ is in fact a new post-collisional velocity associated to the above $\widetilde{\beta}_s$.  We compute for any $s \in (0,1]$,
\begin{multline*}
\int_{\R^3 \times \R^3 \times \S}|u \cdot \n|^{\gamma+2} f(v)g(\vb)\left|\nabla \psi (v'+s\,\zeta)\right|\d v \d\vb \d\n\\
=\int_{\R^3 \times \R^3 \times \S}|u \cdot \n|^{\gamma+2} f(v)g(\vb)\varphi(V'_s)\d v \d\vb \d\n
=\IR \Q_{B_0,\widetilde{e}_s}^+(f,g)(v)\varphi(v)\d v
\end{multline*}
where the collision kernel $B_0$ is given by $B_0(u,\n)=|u\cdot\n|^{\gamma+2}$ and the restitution coefficient $\widetilde{e}_s$ is such that $\widetilde{\beta}_s=\frac{1+\widetilde{e}_s}{2}$.  Since
\begin{equation*}
\begin{split}
\varphi(v) &\leq \langle v \rangle^{k} |\nabla\phi(v)| +k \langle v \rangle^{k-1}  |\phi(v)| \\
&\leq \max(1,k) \left(|\nabla\phi(v)|+|\phi(v)|\right)\langle v \rangle^{k},
\end{split}
\end{equation*}
one has for any $s \in (0,1)$
\begin{multline*}
\int_{\R^3 \times \R^3 \times \S}|u \cdot \n|^{\gamma+2} f(v)g(\vb)\left|\nabla \psi (v'+s\,\zeta)\right|\d v \d\vb \d\n\\
\leq \max(1, k) \IR \Q_{B_0,\widetilde{e}_s}^+(f,g)(v)\left(|\nabla\phi(v)|+|\phi(v)|\right)\langle v \rangle^{k} \d v.
\end{multline*}
As a consequence, thanks to Cauchy-Schwarz inequality
\begin{equation}\label{Ilam}
I_\la \leq \sqrt{2}\max(1, k)\ell_\gamma(e)\la^\gamma \|\phi\,\|_{\mathbf{H}^1}\int_0^1 \|\Q_{B_0,\widetilde{e}_s}^+(f,g)\|_{L^2_k} \d s.
\end{equation}
It remains to estimate the norm $\|\Q_{B_0,\widetilde{e}_s}^+(f,g) \|_{L^2_k}$ for any $s \in (0,1)$.  This is simply done using Theorem A. \ref{alogam}
\begin{align*}
\|\Q_{B_0,\widetilde{e}_s}^+(f,g)\|_{L^2_k}&\leq  C(\widetilde{e}_{s})\|f \|_{L^1_{k+\gamma+2}}\,\|g \|_{L^2_{k+\gamma+2}}.
\end{align*}
In Theorem A. \ref{alogam} is shown that $C(\widetilde{e}_{s})$ only depends on the value at zero of the restitution coefficient.  Since $\widetilde{e}_s(0)=1$ for any $s$ one gets that $C(\widetilde{e}_{s})$ is independent of the variable $s$.  Thus, estimate \eqref{H-1elast} follow from \eqref{Ilam}.  Exchanging the role of $f$ and $g$ in Theorem A. \ref{alogam} gives the second estimate.
\end{proof}
We use the equivalence of norms (that follows using Fourier transform)
\begin{equation}\label{H-1L2}
\|\nabla \varphi\|_{\mathbf{H}^{-1}}^2+\|\varphi\|_{\mathbf{H}^{-1}}^2=\|\varphi\|_{L^2}^2
\end{equation}
valid for any $\varphi \in L^2$ to make Proposition \ref{exponetialest} stronger.
\begin{propo}\label{propo:sobdiff}
For any $\ell \in \mathbb{N}$ and $k\geq0$ there exists $C(\gamma,k,\ell)$ such that
\begin{multline*}
\|\Q^+_{\el}(f,g)-\Q^+_1(f,g)\|_{\mathbf{H}^{\ell}_{k}}\\
 \leq C(\gamma,k,\ell)\; \la^\gamma\; \left(\|f\|_{\mathbf{W}^{1,\ell}_{k+\gamma+2}}\,
 \|g\|_{\mathbf{H}^{\ell+1}_{k+\gamma+2}}+\|f\|_{\mathbf{H}^{\ell+1}_{k+\gamma+2}}\,\|g\|_{\mathbf{W}^{1,\ell}_{k+\gamma+2}}\right)
\end{multline*}
holds for any $\la \in [0,1].$
\end{propo}
\begin{proof} Since
\begin{equation}\label{gradiff}
\nabla \Q_{\el}^+(f,g)=\Q_{\el}^+(\nabla f,g)+ \Q_{\el}^+(f,\nabla g)
\end{equation}
and the same is true for $\Q_1^+(f,g)$, it suffices to apply Proposition \ref{exponetialest} and identity \eqref{H-1L2} conveniently to each term to get the conclusion for $\ell=1$.  Use induction to obtain the result for higher derivatives $\ell>1$.
\end{proof}
It is actually possible to extend these estimates to the smaller space $L^1(m_a)$ with exponential weights
\begin{equation}\label{malpha}
m_{a}(v):=\exp\left(a |v|\right), \qquad v \in \R^3,\qquad a \geq 0.
\end{equation}
We work with exponent $1$ for simplicity even if, as suggested also by Proposition \ref{propo:tails}, it is likely that our results are still valid for general weights of the form $\exp(a|v|^p)$ with $0 \leq p < \frac{3}{2}$. The advantage of the following result with respect to the previous one is that it involves the derivative of  \textit{only one} of the functions $f$ or $g$. Precisely, one has the following that extends \cite[Proposition 3.2]{MiMo2}.
\begin{theo}\label{prop:conti} There exists an \emph{explicit} constant $\la_0 \in (0,1)$ such that for any $a\geq0$ there exists $C(\gamma,a) >0$ for which there holds
\begin{equation}\label{L1diff}
\| \Q^+_{\el}(f,g)-\Q^+_1(f,g) \|_{L^1( m_{a})}
\leq C(\gamma,a) \la^{\frac{\gamma}{8+3\gamma}}\|f\|_{L^{1}_{1}(m_a)}\,\|g\|_{\mathbf{W}^{1,1}_1(m_a)} \qquad \forall \la \in (0,\la_0)
\end{equation}
and
\begin{equation*}
\| \Q^+_{\el}(f,g)-\Q^+_1(f,g) \|_{L^1( m_{a})}
\leq C(\gamma,a) \la^{\frac{\gamma}{8+3\gamma}}\|g\|_{L^{1}_{1}(m_a)}\,\|f\|_{\mathbf{W}^{1,1}_1(m_a)} \qquad \forall \la \in (0,\la_0).
\end{equation*}
\end{theo}
\begin{proof} The proof follows the argument of the analogue \cite[Proposition 3.2]{MiMo2} where the crucial estimate is provided by Proposition A.\ref{propB:cont} (see Appendix A). Precisely, as in the \textit{op. cit.},  for any given $v,\vb \in  \R^3$, $w = v+v_* \not = 0$ and $\sigma \in \S$, we define  the angle $\chi \in [0,\frac{\pi}{2}],$ by $\cos \,\chi := |\sigma \cdot \hat w|$. Let $\delta \in (0,1)$ and $R >1$ be fixed and let
 $b_\delta \in \mathbf{W}^{1,\infty}(-1,1)$ such that $b_\delta(s)=b_\delta(-s)$ for any $s \in (0,1)$ and
$$b_\delta (s) = \begin{cases}1 \quad &\text{ if } \quad s \in (-1+2\delta,1-2\delta)\\
 0 \quad &\text{ if } \quad  s \notin (-1+\delta,1-\delta)\end{cases}$$
with moreover
\begin{equation*}
0 \leq b_\delta(s) \leq 1 \qquad \text{ and } \qquad |b_\delta'(s)| \leq \frac{3}{\delta}\qquad \forall s \in (-1,1).\end{equation*}
Let us define also $\Theta_R(r) = \Theta(r/R)$ with $\Theta(x) = 1$ on $[0,1]$, $\Theta(x) = 1-x$ for $x \in [1,2]$ and $\Theta(x) = 0$ on $[2,\infty)$. We define the sets $A(\delta)  := \{ \sigma \in \S; \,\, \sin^2 \chi \ge \delta \}$,  $B(\delta) := \{ \sigma \in \S\,; \,\, \us \notin (-1+2\delta,1-2\delta) \text{ or } \sin^2 \chi \le \delta \}$. With these notations, for any restitution coefficient $e(\cdot)$, we split $\Q^+_e$ into
$$\Q^+_e = \Q^{+}_{B_0,e}  + \Q^{+}_{B_1,e}  + \Q^{+}_{B_2,e}$$
where the collision kernels $B_i(u,\us)$, $i=0,1,2$, are defined by
$$B_2(u,\us)=b_\delta(\us) \,\Theta_R(u)\frac{|u|}{4\pi}, \qquad B_1(u,\us) := \frac{|u|}{4\pi} \, {\bf 1}_{A(\delta)} \, (1-\Theta_R(|u|))$$ and
$$B_0(u,\us)=\frac{|u|}{4\pi} \, (1 - b_\delta(\us)) \, \Theta_R(|u|) + \frac{|u|}{4\pi} \, (1-\Theta_R(|u|)) \, {\bf 1}_{A^c(\delta)}.$$ We shall of course apply this splitting to the restitution coefficients $\el$ and the elastic one $e\equiv 1$ which corresponds to $e_0$. The proof is divided into three steps.\\

\noindent $\bullet$  \emph{Step 1. Estimate for $\Q_{B_0,\el}^+$:} We can prove exactly as in \cite[Proposition 3.2]{MiMo2} (precisely, using Theorem A. \ref{alogam}) that for any $\la  \in [0,1)$ and any $\delta \in (0,1)$ there holds
\begin{equation}\label{est1}\big\|\Q^{+}_{B_0,\el}(f,g) \big\|_{L^1(m_a)} \leq
2   \, \delta \,  \| f\|_{L^1_1(m_a)} \|g \|_{L^1_1(m_a) }.\end{equation}

\noindent $\bullet$ \emph{Step 2. Estimate for $\Q_{B_1,\el}^+$:} We can check without difficulty that \cite[Lemma 3.3]{MiMo2} still holds true for non-constant restitution coefficient $\el$ with exponent $k$ in \cite[Eq. (3.8)]{MiMo2} given by $k=(1-\delta/160)^{\frac{1}{2}}$ independent of $\la$. In particular, reproducing the proof of the \emph{op. cit.} we get that there exists a constant $C >0$ such that:
\begin{equation}\label{est2}\big\|\Q^{+}_{B_1,\el}(f,g) \big\|_{L^1(m_a)} \leq
\dfrac{C}{\delta^2 R}\,  \| f\|_{L^1_1(m_a)} \|g \|_{L^1_1(m_a) } \qquad \forall \la \in [0,1], \delta \in (0,1),\;R >1.
\end{equation}

\noindent $\bullet$ \emph{Step 3. Estimate for the difference $\Q^+_{B_2,\el}-\Q^+_{B_2,1}$:} The crucial point is now to estimate
$\|\Q^+_{B_2,\el}(f,g)-\Q^+_{B_2,1}(f,g)\|_{L^1(m_a)}$ and, as already mentioned, we shall resort to Proposition A. \ref{propB:cont} given in Appendix A. Precisely, let $\phi \in L^\infty$ and $\psi(v)=m_a(v)\phi(v)$. Notice that the collision kernel $B_2(u,\us)$ satisfies the assumption of Proposition A. \ref{propB:cont} since $\mathrm{Supp}b_\delta \subset (-1+\delta,1-\delta)$. Applying this Proposition to the restitution coefficient $\el$ (with fixed $\la  \in (0,1]$) one sees that there exists $C_{\el} >0$ such that
\begin{multline*}
\left|\IR \left[\Q^+_{B_2,\el}(f,g)-\Q^+_{B_2,1}(f,g)\right]\psi(v)\,\d v\right| \leq C_{\el}   \IR \Q^+_{\mathcal{B}_\gamma,1}(f , g )\,|\psi(v)|\d v \\
 + 2^{\gamma+6}\ell_\gamma(\el)\int_0^1\d s\int_{\R^3} \Q^+_{\mathcal{\overline{B}_\gamma},\tilde{e}_{s}^\la }\left(f, h\right)|\psi(v)|\d v \end{multline*}
where $h(v)=g(v)+|\nabla g(v)|$ while the  kernels $\mathcal{B}_\gamma$ and $\mathcal{\overline{B}_\gamma}$ are given by
$$\mathcal{B}_\gamma(u,\us)={B}_2(u,\us)|u|^\gamma, \quad \mathcal{\overline{B}_\gamma}(u,\us)=\max(B_2(u,\us),|\nabla_u B_2(u,\us)|)|u|^{\gamma+2}$$
and, for any   $s \in [0,1]$, $\tilde{e}_{s}^{\la}(\cdot)$ is a given restitution coefficient with in particular $\tilde{e}_s^\la(0)=1$ for any $s,\la$. One estimates these two integrals using Theorem A. \ref{alogam}. Precisely, by Holder's inequality
$$\IR \Q^+_{\mathcal{B}_\gamma,1}(f , g )\,|\psi(v)|\d v \leq \|\Q^+_{\mathcal{B}_\gamma,1}(f,g)m_a\|_{L^1}\,\|\phi\|_{L^\infty}$$
while, for any $s \in (0,1)$
$$\int_{\R^3} \Q^+_{\mathcal{\overline{B}_\gamma},\tilde{e}_{s}^\la }\left(f, h\right)(v)|\psi(v)|\d v \leq \|\Q^+_{\mathcal{\overline{B}_\gamma},\tilde{e}_{s}^\la }\left(f, h\right)\|_{L^1}\,\|\phi\|_{L^\infty}.$$
Now, one notices that
\begin{equation*}
 m_{a}(v'_s)\leq  m_{a}(v)\,m_{a}(v_{\star}) \qquad \text{ and } \qquad m_a(v'_1) \leq  m_{a}(v)\,m_{a}(v_{\star})
\end{equation*} where $v'_s$ and $v'_1$ denote the post-collision velocities associated to the restitution coefficient $\tilde{e}_s^\la$ and $e\equiv 1$ respectively, so that
$\|\Q^+_{\mathcal{B}_\gamma,1}(f,g)m_a\|_{L^1} \leq \|\Q^+_{\mathcal{B}_\gamma,1}(m_a\,f,m_a\,g)\|_{L^1}$ and $$\|\Q^+_{\mathcal{\overline{B}_\gamma},\tilde{e}_{s}^\la }\left(f, h\right)\|_{L^1} \leq \|\Q^+_{\mathcal{\overline{B}_\gamma},\tilde{e}_{s}^\la }\left(m_a\,f, m_a\,h\right)\|_{L^1}.$$
Since $\Theta_R(|u|)=0$ whenever $|u| > 2R$, one has $\Theta_R(|u|)|u|^{\gamma} \leq (2R)^{\gamma }$ for any $u \in \R^3$ and there exists an universal constant $c_1 >0$ such that
\begin{equation}\label{elastique}\|\Q^+_{\mathcal{B}_\gamma,1}(m_a\,f,m_a\,g)\|_{L^1} \leq c_1\,R^{\gamma} \|m_a f\|_{L^1_1}\,\|m_a g\|_{L^1_1}.\end{equation}
To estimate $\|\Q^+_{\mathcal{\overline{B}_\gamma},\tilde{e}_{s}^\la }\left(m_a\,f, m_a\,h\right)\|_{L^1}$, one only notices that the kernel $\mathcal{\overline{B}_\gamma}(u,\us)$ can be estimated by
$$\mathcal{\overline{B}_\gamma}(u,\us) \leq \bigg[\left(\frac{1}{R}+\frac{1}{\delta}\right)|u|+1\bigg]\,b_\delta(\us)\Xi_R(|u|)\,|u|^{\gamma+2}$$
for some positive mapping $\Xi_R(\cdot)$ such that $\Xi_R(|u|)=0$ whenever $|u| > 2R;$ the factor $R^{-1}$ coming from the derivative of $\Theta_R$ while the term ${\delta}^{-1}$ comes from that of $b_\delta$. Then, using as above Theorem A. \ref{alo} and   because $\tilde{e}_s^\la(0)=1$ is independent of $s \in (0,1)$, one gets the existence of an universal constant $c_2 >0$ such that
$$\|\Q^+_{\mathcal{\overline{B}_\gamma},\tilde{e}_{s}^\la }\left(m_a\,f, m_a\,h\right)\|_{L^1} \leq c_2 \left(\frac{1}{R}+\frac{1}{\delta}\right)R^{\gamma+2}\|m_a f\|_{L^1_1}\,\|m_a h\|_{L^1_1}$$
or, equivalently,
\begin{equation}\label{overlineB}
\|\Q^+_{\mathcal{\overline{B}_\gamma},\tilde{e}_{s}^\la }\left(m_a\,f, m_a\,h\right)\|_{L^1} \leq c_2 \left(\frac{1}{R}+\frac{1}{\delta}\right)R^{\gamma+2} \,\|m_a\,f\|_{L^1_1}\,\|m_a g\|_{\mathbf{W}^{1,1}_1}.
\end{equation}
Finally, using the fact that $\ell_\gamma(\el) \leq \la^\gamma\,\ell_\gamma(e)$ while, as noticed in Remark A. \ref{nbB:Cl}, $C_{\el} \leq c_3\la^\gamma$ for any $\la \in (0,\la_0]$ for some constructive $\la_0 >0$ and some positive constant $c_3 >0$, we finally obtain, combining \eqref{elastique} and \eqref{overlineB} that
\begin{multline}\label{est3}
\|\Q_{B_2,\el}^+(f,g)-\Q^+_{B_2,1} (f,g)\|_{L^1(m_a)}\\ \leq C\la^\gamma \left(\frac{1}{R}+\frac{1}{\delta}\right)R^{\gamma+2}\,\|f\|_{L^1_1(m_a)}\,\|g\|_{\mathbf{W}^{1,1}_1(m_a)} \quad \forall \la \in (0,\la_0]
\end{multline}
for some positive constant $C >0$. Collecting estimates \eqref{est1}--\eqref{est2}--\eqref{est3}, we finally get that there is some positive $C >0$ such that
\begin{multline} \label{est4}
\|\Q_{\el}^+(f,g)-\Q^+(f,g)\|_{L^1(m_a)}\\ \leq C\left(\delta+\delta^{-2}R^{-1}+\frac{R^{\gamma+2}\la^\gamma}{\delta}+R^{\gamma +1} \la^\gamma\right)\|f\|_{L^1_1(m_a)}\,\|g\|_{\mathbf{W}^{1,1}_1(m_a)}\\
\forall \la \in (0,\la_0),\;\: \delta >0,\:\:R >1.
\end{multline}
Then, choosing $\delta$ and $R >1$ such that
$$\delta=\delta^{-2}R^{-1}=\frac{R^{\gamma+2}\la^\gamma}{\delta}=\la^p$$
for some $p >0$, one sees that necessarily $p=\frac{\gamma}{8+3\gamma}$ and $R^\gamma \la^\gamma=\la^{5p}$. This gives the conclusion. One proves the second estimate exactly in the same way.
 \end{proof}
Notice that, increasing the polynomial weights in the various norms of $f$ and $g$, we can get an optimal control rate $\la^\gamma$.
\begin{cor}\label{preciserate} There exists some \emph{explicit} $\la_0 \in (0,1)$ such that for any $a\geq0$ there exists some explicit constant $C(\gamma,a) >0$ for which there holds
\begin{equation}\label{L1diffK}
\| \Q^+_{\el}(f,g)-\Q^+_1(f,g) \|_{L^1( m_{a})}
\leq C(\gamma,a) \la^\gamma\|f\|_{L^{1}_{k}(m_a)}\,\|g\|_{\mathbf{W}^{1,1}_k(m_a)} \qquad \forall \la \in (0,\la_0)
\end{equation}
and
\begin{equation*}
\| \Q^+_{\el}(f,g)-\Q^+_1(f,g) \|_{L^1( m_{a})}
\leq C(\gamma,a) \la^{\gamma}\|g\|_{L^{1}_{k}(m_a)}\,\|f\|_{\mathbf{W}^{1,1}_k(m_a)} \qquad \forall \la \in (0,\la_0)
\end{equation*}
where $k=\gamma+\frac{10}{3}.$
\end{cor}
\begin{proof} The proof follows the lines given for Theorem \ref{prop:conti}.  Let us explain the small changes. Bounding directly $\Theta_R$ by $1$ allows to replace Estimate \eqref{elastique} by
$$\|\Q^+_{\mathcal{B}_\gamma,1}(m_a\,f,m_a\,g)\|_{L^1} \leq c_1\,\|m_a f\|_{L^1_{\gamma+1}}\,\|m_a g\|_{L^1_{\gamma+1}}.$$
In the same way, for some given $1 < k < \gamma+3$ to be determined later, one can replace Estimate \eqref{overlineB} by
$$\|\Q^+_{\mathcal{\overline{B}_\gamma},\tilde{e}_{s}^\la }\left(m_a\,f, m_a\,h\right)\|_{L^1} \leq c_2 \left(\frac{1}{R}+\frac{1}{\delta}\right)R^{\gamma+3-k} \,\|m_a\,f\|_{L^1_k}\,\|m_a g\|_{\mathbf{W}^{1,1}_{k}}.$$
With such a choice, \eqref{est3} becomes
\begin{multline*}
\|\Q_{B_2,\el}^+(f,g)-\Q^+_{B_2,1} (f,g)\|_{L^1(m_a)} \leq C\la^\gamma \left(\frac{1}{R}+\frac{1}{\delta}\right)R^{\gamma+3-k}\,\|f\|_{L^1_k(m_a)}\,\|g\|_{\mathbf{W}^{1,1}_k(m_a)}\\
+C\la^\gamma \,\|f\|_{L^1_{\gamma+1}(m_a)}\,\|g\|_{L^1_{\gamma+1}(m_a)} \forall \la \in (0,\la_0]
\end{multline*}
and, collecting all the estimates as above we get
\begin{multline*}
\|\Q_{\el}^+(f,g)-\Q^+(f,g)\|_{L^1(m_a)} \leq\\ C\left(\delta+\delta^{-2}R^{-1}+\frac{R^{\gamma+3-k}\la^\gamma}{\delta}+R^{\gamma +2-k} \la^\gamma+\la^\gamma\right)\|f\|_{L^1_s(m_a)}\,\|g\|_{\mathbf{W}^{1,1}_s(m_a)}\\
\forall \la \in (0,\la_0),\;\: \delta >0,\:\:R >1
\end{multline*}
where $s=\max(k,\gamma+1).$ One looks now for $k \in (1,\gamma+2)$ for which it is possible to choose $\delta$, $R >1$ such that
$$\delta=\delta^{-2}R^{-1}=\frac{R^{\gamma+3-k}\la^\gamma}{\delta}=\la^\gamma$$
and we get that, necessarily, $k=\gamma+\frac{10}{3}$. In this case, $R^{\gamma+2-k}\la^\gamma=\la^{5\gamma}$ and we obtain finally \eqref{L1diffK}.
\end{proof}

\section{Uniqueness}

We are now in position to prove the uniqueness of the solution to \eqref{rescaled} for sufficiently small $\la$, that is, there exists $\la^\dagger >0$ such that for any $\la \in (0,\la^\dagger)$ the stationary problem \eqref{rescaled} admits a unique solution $\gl$ with unit mass and  vanishing momentum. The strategy of proof has been sketched in the Introduction and we shall refer to Section \ref{sec:strategy} for the main steps of the proof. In particular, a crucial point consists in proving and quantifying the convergence of $(\gl)_{\la}$ towards an universal limit $\mathcal{M}$. This is the object of the following paragraph.

\subsection{The limit $\la \to 0$: non quantitative version} Using the continuity properties, specifically Theorem \ref{prop:conti}, and a compactness argument, we establish a first convergence result, non quantitative in the sense that no rate of convergence is provided.
\begin{theo}\label{convergenceMax}
Assume that $e(\cdot)$ belongs to the class $\mathbb{E}_m$ with $m > \frac{7}{2}$. For any $k \geq 0$ and $\ell\in[0,m-1]$, one has
$$\lim_{\la \to 0}\|\gl-\mathcal{M}\|_{\mathbf{H}^\ell_k}=0$$
where $\mathcal{M}$ is the Maxwellian
$$\mathcal{M}(v)=(2\pi\Theta)^{-\frac{3}{2}}\exp\left(-\frac{|v|^2}{2\Theta}\right).$$
The Maxwellian's temperature $\Theta$ is given by
\begin{equation}\label{temptheta}\Theta= \left(\frac{6(4+\gamma)}{ \mathfrak{a}\,2^{\frac{\gamma}{2}}m_{3+\gamma}}\right)^{\frac{2}{3+\gamma}}\end{equation}
where $m_{3+\gamma}$ is the $(3+\gamma)$-th moment of a normalized Gaussian
$$m_{3+\gamma}=\pi^{-\frac{3}{2}}\int_{\R^3} \exp\left(-\frac{|v|^2}{2}\right)|v|^{3+\gamma}\d v=2^{\frac{3+\gamma}{2}}\frac{\Gamma\left(3+\frac{\gamma}{2}\right)}{\Gamma\left(\frac{3}{2}\right)}.$$
\end{theo}
\begin{proof} The proof is divided in several steps and essentially based upon a compactness argument through Theorem \ref{theo:HKL}.

$\bullet$ \textit{First step: compactness argument.} Let us choose $m-1 \geq \ell > \frac{5}{2} $ and $k_0 \geq 1$ in the Theorem \ref{theo:HKL}. It clearly exists a sequence $(\la_n)_n$ with $\la_n \to 0$ and $G_0 \in \mathbf{H}^\ell_k$ such that $\left(G_{\la_n}\right)_n$ converges weakly, in $\mathbf{H}^\ell_{k_0}$, to $G_0$ (notice that, \textit{a priori}, the limit function $G_0$ depends on the choice of $\ell$ and $k_0$). Using the decay of $\gl$ guaranteed by the \textit{polynomially weighted} Sobolev estimates, we can prove thanks to a simple  localization argument (and using compact embedding for Sobolev spaces)  that  the convergence is actually \textit{strong} in $\mathbf{H}^1_k$ for any $0\leq k < k_0$
\begin{equation}\label{strongconv}\lim_{n \to \infty}\|G_{\la_n}-G_0\|_{\mathbf{H}^1_k}=0.\end{equation}
Indeed, since $\sup_{\la \in (0,1)}\|\gl-G_0\|_{\mathbf{H}^\ell_{k_0}} <\infty$, for any fixed $0 \leq k < k_0$ and any $\varepsilon >0$, there is $R >0$ large enough such that
\begin{equation}\label{BRC}\sup_{\la \in (0,1)}\|G_{\la}-G_0\|_{\mathbf{H}^1_k(B^c_R)}\leq \varepsilon\end{equation}
where $B_R=\{v \in \R^3\,,\,|v|\leq R\}$ and $B^c_R$ its complementary. Let $\tilde{G}_{\la_n}$ and $\tilde{G}_0$ denote the restrictions of $G_{\la_n}$ and $G_0$ to the ball $B_R$. Since $\ell >\frac{5}{2}$, according to Rellich-Kondrachov compactness theorem \cite[Theorem 6.2, p.144]{adams}, the embedding $\mathbf{H}^\ell(B_R) \hookrightarrow \mathbf{H}^1(B_R)$ is compact so that there is a subsequence of $(\tilde{G}_{\la_n})_n$ that converges strongly to $\tilde{G}_0$ in $\mathbf{H}^1(B_R)$. Since $\tilde{G}_0$ is the unique limit of all subsequences, it is actually the whole sequence $(\tilde{G}_{\la_n})_n$ that converges to $\tilde{G}_0$ in $\mathbf{H}^1(B_R)$. Combining this with \eqref{BRC} yields \eqref{strongconv}.

$\bullet$ \textit{Second step: identification of the limit $G_0$.} Let us prove now that the above limit $G_0$ is actually a Maxwellian distribution  with temperature $\Theta.$ To do so, one uses \eqref{rescaled} to get
$$\|\Q_{\el}(\gl,\gl)\|_{L^2}=\lambda^\gamma \|\Delta_v \gl\|_{L^2} \qquad \forall \la >0$$
and, since $\sup_{\la \in (0,1]}\|\Delta_v \gl\|_{L^2}\leq \sup_{\la \in (0,1]}\|\gl\|_{\mathbf{H}^2}=:C_0 < \infty$ according to Theorem \ref{theo:HKL}, we get
\begin{equation}\label{qelgl}\|\Q_{\el}(\gl,\gl)\|_{L^2} \leq C_0\la^\gamma\qquad \forall \la \in (0,1).\end{equation}
Now, from the identity $\Q_1^-(\gl,\gl)=\Q_{\el}^-(\gl,\gl)$, one has $\|\Q_{\el}(\gl,\gl)-\Q_1(\gl,\gl)\|_{L^2}=\|\Q_{\el}^+(\gl,\gl)-\Q^+_1(\gl,\gl)\|_{L^2}$ so that,
$$\|\Q_1(\gl,\gl)\|_{L^2} \leq \|\Q_{\el}^+(\gl,\gl)-\Q^+_1(\gl,\gl)\|_{L^2}+\|\Q_{\el}(\gl,\gl)\|_{L^2}.$$
Combining the above estimate \eqref{qelgl} with Proposition \ref{propo:sobdiff}  we get
$$\|\Q_1(\gl,\gl)\|_{L^2} \leq  C_0\la^\gamma+ C_1\la^\gamma \|\gl\|_{\mathbf{H}^1_2}^2$$
for some positive constant $C_1 >0$ independent of $\la$. Using again Theorem \ref{theo:HKL}, we get that there exists some explicit constant $C_2 >0$ such that
$$\|\Q_1(\gl,\gl)\|_{L^2} \leq C_2 \la^\gamma \qquad \forall \la \in (0,1].$$
In particular the sequence $(G_{\la_n})_n$ constructed in the first step satisfies
$$\lim_{n \to \infty}\|\Q_1(G_{\la_n},G_{\la_n})\|_{L^2}=0.$$
Since $G_{\la_n} \to G_0$ strongly in $L^2_1$, we get easily that
$$\Q_1(G_0,G_0)=0$$
i.e. $G_0$ is a Maxwellian distribution. By conservation of mass and momentum, we get that $G_0$ has unit mass and zero momentum and it remains only to determine its temperature $\Theta$. To do so, we shall use equation \eqref{scaleDE} and Lemma A. \ref{lem:dissTem}. With the notations of Lemma A. \ref{lem:dissTem}, equation \eqref{scaleDE} writes $\mathcal{I}_\la(\gl)=6$ for any $\la \in (0,1]$.  Applying Lemma A. \ref{lem:dissTem} with $f_1=g_1=G_{\la_n}$, $f_2=g_2=G_0$ (with for simplicity $\delta=1$) and estimating the weighted $L^1$-norms by $L^2$-norms (using equation \eqref{taug21} with $\theta=\frac{1}{2}$ for instance) we get that, for any $\varepsilon >0$ there is $n_0 \geq 1$ such that
$$|\mathcal{I}_0(G_0)-6|\leq C_1 \|G_{\la_n}-G_0\|_{L^2_{5+\gamma}}+C_2\,\varepsilon \qquad \forall n \geq n_0$$
for some positive constants $C_1,C_2 >0$ independent of $n$ where we used that $k_0 \geq 6+\gamma$ and the uniform estimates on $\|\gl\|_{L^2_{5+\gamma}}$. Letting $n$ go to infinity, we get that $\mathcal{I}_0(G_0)=6$.  Therefore,
$$6= C_\gamma \IRR G_0(v)G_0(\vb)|v-\vb|^{3+\gamma}\d v\d\vb=C_\gamma \Theta^{\frac{3+\gamma}{2}}\IRR M(v)M(\vb)|v-\vb|^{3+\gamma}\d v\d\vb$$
where $M(\cdot)$ is the normalized Maxwellian $M(v)=\pi^{-\frac{3}{2}} \exp\left(-\frac{|v|^2}{2}\right)$.  Some algebra yields
$6=C_\gamma  \Theta^{\frac{3+\gamma}{2}}2^{\frac{\gamma}{2}}m_{3+\gamma}$ from which we deduce \eqref{temptheta}, and thus, $G_0=\mathcal{M}$.

$\bullet$ \textit{Final step: convergence of the whole net $(\gl)_{\la}$.} We conclude the proof by showing that
\begin{equation*}
\lim_{\la \to 0}\|\gl-\mathcal{M}\|_{\mathbf{H}^1}=0.
\end{equation*}
Argue by contradiction assuming this does not hold. Then, there exists $\epsilon_{0}>0$ and a sequence $(\lambda_{n})_{n}$ converging to zero such that
\begin{equation*}
\| G_{\lambda_{n}} - \mathcal{M}\|_{\mathbf{H}^1}\geq\epsilon_0 \qquad \forall n \in \mathbb{N}.
\end{equation*}
We just proved above that $(G_{\lambda_{n}})_{n}$ admits a subsequence $(G_{\lambda_{n_j}})_{j}$ converging strongly in $\mathbf{H}^1$ to $\mathcal{M}$.  Therefore,
\begin{equation*}
\epsilon_0\leq \| {G_{\la}}_{n_j} - \mathcal{M}\|_{\mathbf{H}^1}\underset{j \to \infty}{\longrightarrow} 0,
\end{equation*}
which is a contradiction.  This proves that the full net $(\gl)_{\lambda}$ converges to $\mathcal{M}$ strongly in any $\mathbf{H}^1$. We proceed along the same path (using a version of Rellich-Kondrachov Theorem for higher-order Sobolev spaces) to prove that the convergence actually holds in any weighted Sobolev space $\mathbf{H}^\ell_k$, $k \geq 0$ and $\ell\in[0, m-1]$.
\end{proof}
This convergence in Sobolev spaces can be extended easily to weighted $L^1$-spaces with exponential weights. Recall that, for any $a \geq 0$, we denote
$$m_a(v)=\exp(a|v|),\qquad \qquad v \in \R^3.$$
\begin{cor}\label{cor:limit}  Assume that $e(\cdot)$ belongs to the class $\mathbb{E}_m$ with $m > \frac{7}{2}$. For any $a \geq 0$ and any $k \geq 0$ it holds
$$\lim_{\la \to 0} \left\|\gl-\mathcal{M}\right\|_{L^1_k(m_a)}=0.$$
\end{cor}
\begin{proof}  Taking $\ell > \frac{3}{2}$ in the above Theorem, observe that by classical Sobolev embedding
$$\lim_{\la \to 0} \left\|\gl-\mathcal{M}\right\|_{L^\infty} =0.$$
The proof follows then using interpolation.  First, observe that the convergence holds in exponential weighted $L^2$-spaces
$$\int_{\R^3}\left|\gl(v)-\mathcal{M}(v)\right|^2 m_b(v)\d v \leq  C_b\|\gl-\mathcal{M}\|_{L^\infty} \longrightarrow 0 \qquad \text{ as } \quad \la \to 0$$
where $C_b:=\sup_{\la \in (0,1]} \left\|\gl-\mathcal{M}\right\|_{L^1(m_b)}$ is finite for any $b \geq 0$ thanks to Proposition \ref{propo:tails}. Then, using Holder's inequality, for any $b,a >0$
$$\|\gl-\mathcal{M}\|_{L^1_k(m_a)} \leq \left(\int_{\R^3}\left|\gl(v)-\mathcal{M}(v)\right|^2 m_b(v)\d v\right)^{\frac{1}{2}} \left(\IR m_b(v)^{-1}m_a^2(v)\langle v \rangle^{2k}\d v\right)^{\frac{1}{2}}.$$
The last integral in the right side is finite provided $b > 2a$, therefore, the $L^1$ convergence follows from the $L^2$ convergence just proved.
\end{proof}

\subsection{Uniqueness result}  On the basis of the above convergence result, we are in position to apply our general strategy as explained in Section \ref{sec:strategy}. Recall that, for any given $\la  \in (0,1]$ and $\gl, \fe \in \mathfrak{S}_\la$, we set
$$H_\la=\fe-\gl.$$
We have to determine Banach spaces $\mathcal{X}$ and $\mathcal{Y}$ for which the estimates \eqref{continueQ1} -- \eqref{XY2} hold true. The analysis of the previous section suggests the choice
$$\mathcal{X}=L^1(m_a), \qquad \qquad \mathcal{Y}=L^1_1(m_a)$$
for some exponential weight $m_a(v)=\exp(a|v|)$, $a \geq 0$. Indeed, Proposition \ref{prop:conti} and Corollary \ref{cor:limit} already ensure that \eqref{XY} and  \eqref{conveX} are fulfilled.
The fact that \eqref{continueQ1} stands is a classical property of Boltzmann operator with hard-spheres interaction (see \cite[Theorem 12]{AloGam}). Since
$$\|\Delta H_\la\|_{\mathcal{X}} \leq \|H_\la\|_{\mathbf{W}^{2,1}(m_a)}$$
one sees that estimate \eqref{LaplX} holds because of Proposition \ref{prop:Hkdiff} (precisely, inequality \eqref{Deltama}). Now, property \eqref{XY2} is a consequence of the following spectral property of $\mathscr{L}_1$, first established in \cite{Mo}.
\begin{propo}\label{spect}
  The spectrum of the linearized operator $\mathscr{L}_1$ in $\mathcal{X}$ (with domain $\mathscr{D}(\mathscr{L}_1)=\mathcal{Y}$) has the following structure:
  \begin{enumerate}
  \item $0$ is a simple eigenvalue of $\mathscr{L}_1$ associated to the null set $$\mathscr{N}(\mathscr{L}_1)=\mathrm{Span}(\mathcal{M},v_1\mathcal{M},v_2\mathcal{M},v_3\mathcal{M},|v|^2\mathcal{M});$$
  \item the continuous spectrum of $\mathscr{L}_1$ is given by $(-\infty,-\nu_0]$ where $$\nu_0=\inf_v \IR |v-\vb|\mathcal{M}(\vb)\d\vb;$$
  \item the non zero eigenvalues of $\mathscr{L}_1$ are all negative and can accumulate only at $-\nu_0$. Consequently, $\mathscr{L}_1$ admits a positive spectral gap $\nu >0$.
   \end{enumerate} In particular, if
$$\widehat{\mathcal{X}}=\{f \in \mathcal{X}\,;\,\IR f \d v=\IR vf\d v=\IR |v|^2f(v)\d v=0\}, \qquad \widehat{\mathcal{Y}}=\mathcal{Y} \cap \widehat{\mathcal{X}} $$
then   $\mathscr{N}(\mathscr{L}_1) \cap \widehat{\mathcal{Y}}=\{0\}$ and
  $\mathscr{L}_1$ is invertible from $\widehat{\mathcal{Y}}$ to
  $\widehat{\mathcal{X}}$ with explicit estimates for $\left\|\mathscr{L}_1^{-1}\right\|_{\widehat{\mathcal{X}} \to \widehat{\mathcal{Y}}}.$ Consequently, inequality \eqref{XY2} holds true with $c_0=\left\|\mathscr{L}_1^{-1}\right\|_{\widehat{\mathcal{X}} \to \widehat{\mathcal{Y}}}.$
\end{propo}
\begin{nb} The proof of the above proposition can be seen as a consequence of some general comparison principle that asserts that the linearized collision operator enjoys the same spectral properties in $\mathcal{X}$ and in the largest Hilbert space $\H=L^2(\mathcal{M}^{-1})$. A simple proof of Proposition\ref{spect} can also be recovered from \cite{canizo}.
\end{nb}
The difference $H_\la=\fe-\gl$ does not necessary belong to $\widehat{\mathcal{Y}}$ since we do not know \textit{a priori} that $G_\la$ and $F_\la$ share the same kinetic energy. Consequently, we need a slight modification of the strategy developed in Section \ref{sec:strategy} to state our main result, regarding uniqueness of the steady state. \begin{theo}\label{uniqueNon} Let $e(\cdot)$ belong to the class $\mathbb{E}_m$ for some integer $m \geq 4.$ There exists $\la^\dag \in (0,1]$ such that
$$\mathfrak{S}_\la=\left\{G_\la \in L^1_2\,; \gl  \text{ solution to \eqref{rescaled} with } \IR \gl(v) \d v =1 \text{ and } \IR v \gl(v)\d v=0 \right\} $$
reduces to a singleton for any $\la \in [0,\la^\dag).$
\end{theo}
\begin{proof} We explained in the previous paragraph that the estimates \eqref{continueQ1}, \eqref{XY}, \eqref{LaplX}, \eqref{conveX} and \eqref{XY2} of the general strategy are fulfilled with $\mathcal{X}=L^1_1(m_a)$ and $\mathcal{Y}=L^1_1(m_a)$ for  any  $a \geq 0$.  Let us fix $\varepsilon >0$ and reproduce the computations of Section \ref{sec:strategy}.  It follows that there exists $\la_0 \in (0,1)$ such that
\begin{equation}\label{estimateA2}\|\mathscr{L}_1(H_\la)\|_{\mathcal{X}} \leq \varepsilon \|H_\la\|_{\mathcal{Y}} \qquad \forall \la \in (0,\la_0).\end{equation}
Let us now introduce the following lifting of the operator $\mathscr{L}_1$ into an invertible operator
\begin{equation}\label{lift}
\mathcal{A}\::\:h \mapsto \mathcal{A}h=\left(\mathcal{A}_1 h;\mathcal{A}_2 h\right) \in \mathbb{R} \times \widehat{\mathcal{X}}\end{equation}
where the second component $\mathcal{A}_2h=\mathscr{L}_1h$ while the first component $\mathcal{A}_1$ is defined by
$$\mathcal{A}_1 h=2\mathcal{I}_0(\mathcal{M},h)=2\IRR \mathcal{M}(v)h(\vb)\zeta_0\left(|v-\vb|^2\right)\d v\d\vb.$$
We refer to the Appendix A   for notations. Since $G_\la$ and $F_\la$ share the same mass and momentum, one deduces from Lemma A. \ref{Ainvert} that
$$H_\la=\fe-\gl=\mathcal{A}^{-1}\mathcal{A}H_\la=\mathcal{A}^{-1}\left(\mathcal{A}_1(H_\la);\mathcal{A}_2(H_\la)\right)$$
with an explicit estimate of the norm $\|\mathcal{A}^{-1}\|$. Since $\mathcal{A}^{-1}$ maps $\mathbb{R}\times\mathcal{X}$ to $\mathcal{Y}$, we get
\begin{equation}\label{differe}
\|H_\la\|_{\mathcal{Y}} \leq \|\mathcal{A}^{-1}\|\,\max\bigg(\big|\mathcal{A}_1(H_\la)\big|\,;\,\left\|\mathcal{A}_2(H_\la)\right\|_{\mathcal{X}}\bigg). \end{equation}
Still using the notations of Appendix A, one has readily
\begin{multline*}
\mathcal{A}_1(H_\la)=\mathcal{I}_0(\mathcal{M}-F_\la,H_\la)+\mathcal{I}_0(\mathcal{M}-\gl,H_\la)\\+
\bigg(\mathcal{I}_0(\fe+\gl,\fe-\gl)-\mathcal{I}_\la(\fe+\gl,\fe-\gl)\bigg).
\end{multline*}
Now, it is clear that
\begin{multline*}
\left|\mathcal{I}_0(\mathcal{M}-F_\la,H_\la)+\mathcal{I}_0(\mathcal{M}-\gl,H_\la)\right| \leq C_\gamma \left(\|\mathcal{M}-\fe\|_{L^1_{3+\gamma}}+\|\mathcal{M}-\gl\|_{L^1_{3+\gamma}}\right)\|H_\la\|_{L^1_{3+\gamma}}\\
\leq C_\gamma \left(\|\mathcal{M}-\fe\|_{\mathcal{X}}+\|\mathcal{M}-\gl\|_{\mathcal{X}}\right)\|H_\la\|_{\mathcal{Y}} \qquad \forall \la \in (0,1].\end{multline*}
Consequently, applying then Lemma A. \ref{lem:dissTem} with $f_1=f_2=\fe+\gl$ and $g_1=g_2=\fe-\gl$, we get the existence of constant $C >0$ and some $\la_1 \in (0,1)$ such that
\begin{equation}\label{A1H}
\left|\mathcal{A}_1(H_\la)\right| \leq C \left(\varepsilon +\|\mathcal{M}-\fe\|_{\mathcal{X}}+\|\mathcal{M}-\gl\|_{\mathcal{X}} \right)\|H_\la\|_{\mathcal{Y}} \qquad \forall \la \in (0,\la_1).\end{equation}
In particular, on the basis of \eqref{conveX},   there exists $\la_2\in (0,\la_1)$ such that
$$\left|\mathcal{A}_1(H_\la)\right|  \leq 2C\varepsilon \|H_\la\|_{\mathcal{Y}} \qquad \forall \la \in (0,\la_2).$$
Using now \eqref{estimateA2} together  with \eqref{differe}, we deduce that
$$\|H_\la\|_{\mathcal{Y}} \leq  \varepsilon \max(1,2C)\|\mathcal{A}^{-1}\|\,\|H_\la\|_{\mathcal{Y}} \qquad \forall \la \in (0,\la^\dagger)$$
where $\la^\dagger=\min(\la_0,\la_2).$ Taking $\varepsilon >0$ small enough yields therefore the desired uniqueness: $H_\la=0$ for any $\la \in (0,\la^\dagger).$
\end{proof}
\subsection{Quantitative version of the uniqueness result.}  We derive is this section a quantitative version of the Theorem \ref{convergenceMax} which shall result in a quantitative estimate of the above parameter $\la^\dagger$.
\begin{propo}\label{quadraticM} Let $e(\cdot)$ belongs to the class $\mathbb{E}_m$ with $m  \geq 4$. Assume moreover that  there exist two positive constants $\mathfrak{a},\mathfrak{b} >0$ and two exponents  $\overline{\gamma}  >\gamma >0$ such that
\begin{equation}\label{order2}
\left|e(r)- 1+\mathfrak{a}\,r^\gamma\right| \leq \mathfrak{b}\,r^{\overline{\gamma}} \text{ for any } r \geq  0.\end{equation}
For any $a \geq 0$, there exist some explicit $\la_0 \in (0,1)$, $c_0,c_1 >0$ and exponent $\alpha=\min(\gamma,\overline{\gamma}-\gamma) >0$ such that the estimate
\begin{equation*}
\|\gl- \mathcal{M}\|_{L^1(m_a)} \leq c_0\la^\alpha + c_1 \left\|\gl-\mathcal{M}\right\|^2_{L^1_1(m_a)} \qquad \forall \la \in (0,\la_0)
\end{equation*}
holds for any $\gl \in \mathfrak{S}_\la.$
\end{propo}
\begin{proof} We apply a slight modification of the proof of Theorem \ref{uniqueNon} where, instead of estimating the difference of two solutions to \eqref{rescaled}, we estimate the difference $\gl-\mathcal{M}$.  Recall that $\mathcal{A}$ is the lifting operator given by \eqref{lift}.  Thus,
$$\gl-\mathcal{M}=\mathcal{A}^{-1}\mathcal{A}(\gl-\mathcal{M})=\mathcal{A}^{-1}\left(\mathcal{A}_1(\gl-\mathcal{M});\mathcal{A}_2(\gl-\mathcal{M})\right)$$
where the norm of $\|\mathcal{A}^{-1}\|$ is explicit. In particular, since $\mathcal{A}^{-1}$ maps $\mathbb{R} \times \widehat{\mathcal{X}}$ to $\mathcal{Y}$, we get
\begin{equation}\label{differeM} \|\gl-\mathcal{M}\|_{\mathcal{Y}}\leq \|\mathcal{A}^{-1}\|\max\bigg(\big|\mathcal{A}_1(\gl-\mathcal{M})\big|;\left\|\mathcal{A}_2(\gl-\mathcal{M})\right\|_{\mathcal{X}}\bigg).\end{equation}
Let us estimate separately the two terms $\mathcal{A}_1(\gl-\mathcal{M})$ and $\mathcal{A}_2(\gl-\mathcal{M})$. In the one hand,
\begin{multline*}
\mathcal{A}_1(\gl-\mathcal{M})=2\mathcal{I}_0(\mathcal{M},\gl-\mathcal{M})=\mathcal{I}_0(\mathcal{M}-\gl,\gl-\mathcal{M})+\\
\bigg(\mathcal{I}_0(\gl,\gl)-\mathcal{I}_\la(\gl,\gl)\bigg)
\end{multline*}
where we used the fact that $\mathcal{I}_0(\mathcal{M},\mathcal{M})=\mathcal{I}_\la(\gl,\gl)=6$. Now, it is clear that
$$\left|\mathcal{I}_0(\mathcal{M}-\gl,\gl-\mathcal{M})\right|\leq C_\gamma \|\gl-\mathcal{M}\|^2_{L^1_{3+\gamma}}.$$
Moreover, according to Lemma A. \ref{lem:quantit} and under assumption \eqref{order2},
$$\bigg|\mathcal{I}_0(\gl,\gl)-\mathcal{I}_\la(\gl,\gl)\bigg| \leq C_0\la^\alpha \|\gl\|^2_{L^1_{3+\gamma+\overline{\gamma}}}$$
for some explicit constant $C_0 >0$ and exponent $\alpha=\min(\gamma,\overline{\gamma}-\gamma) >0$. Therefore, since $\sup_{\la \in (0,1]}\|\gl\|^2_{L^1_{3+\gamma}}$, with the notations of the previous section
\begin{equation}\label{A1glM}
\left|\mathcal{A}_1(\gl-\mathcal{M})\right| \leq C_1\,\la^\alpha + C_\gamma \|\gl-\mathcal{M}\|^2_{L^1(m_a)} \qquad \forall \la \in (0,1]
\end{equation}
for some positive constant $C_1 >0$. On the other hand,
\begin{equation*}\begin{split}
\mathcal{A}_2(\gl-\mathcal{M})&= \Q_1(\gl-\mathcal{M}, \mathcal{M} -\gl) + \Q_1(\gl,\gl)\\
&=\Q_1(\gl-\mathcal{M}, \mathcal{M} -\gl)+ \left(\Q_1(\gl,\gl)-\Q_{\el}(\gl,\gl)\right) + \Q_{\el}(\gl,\gl)\\
&=\Q_1(\gl-\mathcal{M}, \mathcal{M} -\gl)+ \left(\Q_1(\gl,\gl)-\Q_{\el}(\gl,\gl)\right) -\la^\gamma \Delta \gl.
\end{split}\end{equation*}
Therefore, using Corollary \ref{preciserate}, there exist an explicit $\la_0 \in (0,1)$ and constants $C_2,C_3 >0$ such that
\begin{multline}\label{A2glM}
\|\mathcal{A}_2(\gl-\mathcal{M})\|_{L^1(m_a)} \leq C_1\|\gl-\mathcal{M}\|_{L^1_1(m_a)}^2 + C_2\la^{\gamma}\|\gl\|_{L^1_{\gamma+\frac{10}{3}}(m_a)}\|\gl\|_{\mathbf{W}^{1,1}_{\gamma+\frac{10}{3}}(m_a)}\\+\la^\gamma\|\gl\|_{\mathbf{W}^{2,1}(m_a)} \qquad \forall \la \in (0,\la_0).
\end{multline}
Using interpolation, similar to the proof of \eqref{Deltama}, and Proposition \ref{prop:momdiff} we obtain
$$\sup_{\la \in (0,1]}\left(\|\gl\|_{\mathbf{W}^{1,1}_{\gamma+\frac{10}{3}}(m_a)}+\|\gl\|_{\mathbf{W}^{2,1}(m_a)}\right) < \infty.$$
Hence, inequality \eqref{A2glM} reads
$$\|\mathcal{A}_2(\gl-\mathcal{M})\|_{L^1(m_a)} \leq C_1\|\gl-\mathcal{M}\|_{L^1_1(m_a)}^2 + C_3\la^{\gamma} \qquad \forall \la \in (0,\la_0)$$
for explicit constants $C_1,C_3 >0$. Combining this estimate with \eqref{A1glM} and \eqref{differeM} yields the desired conclusion.
\end{proof}
\begin{theo}\label{theo:convExp}  Assume that $e(\cdot)$ satisfies \eqref{order2} and belongs to the class $\mathbb{E}_m$ with $m \geq 4$.  Fix the exponential weight $m_a(v)=\exp(a|v|)$ with $a\geq0$. There exist an explicit $\la^\star \in (0,1)$ and constant $c >0$ such that
$$\|\gl-\mathcal{M}\|_{L^1(m_a)} \leq c \la^\alpha \qquad \forall \la \in (0,\la^\star)$$
where $\gl \in \mathfrak{S}_\la$ and $\alpha=\min(\gamma,\overline{\gamma}-\gamma).$
\end{theo}
\begin{proof} The proof follows from the Proposition \ref{quadraticM} and the non quantitative convergence Theorem \ref{convergenceMax}. Indeed, recall the estimate
\begin{equation}\label{estglM}
\|\gl- \mathcal{M}\|_{L^1_1(m_a)} \leq c_0\la^\alpha + c_1 \left\|\gl-\mathcal{M}\right\|^2_{L^1_1(m_a)} \qquad \forall \la \in (0,\la_0)\end{equation}
for some explicit constants $c_0,c_1 >0$. Then, since $\lim_{\la \to 0}\left\|\gl-\mathcal{M}\right\|_{L^1_1(m_a)}=0$, there is some \textit{a priori} non explicit $\la^\star \in (0,\la_0)$ such that
\begin{equation}\label{c1alpha}c_1\left\|\gl-\mathcal{M}\right\|_{L^1_1(m_a)}\leq \frac{1}{2} \qquad \forall \la \in (0,\la^\star).\end{equation}
Therefore, estimate \eqref{estglM} becomes
\begin{equation}\label{c0alpha}\|\gl-\mathcal{M}\|_{L^1_1(m_a)} \leq 2c_0\la^\alpha \qquad \forall \la \in (0,\la^\star).\end{equation}
This gives \textit{a posteriori} an explicit estimate for $\la^\star$ since the optimal $\la^\star$ will be the one for which \eqref{c1alpha} and \eqref{c0alpha} are identity which yields the estimate $\la^\star \geq (4c_0c_1)^{-\frac{1}{\alpha}}$. Since all the parameters $c_0,c_1,\alpha$ happen to be explicitly computable we get the result.
\end{proof}
\begin{nb} We wish to emphasize here several points about our approach. First, recall that in the case of constant restitution coefficient, the approach of \cite{MiMo3} yields directly quantitative results. This was possible thanks to  a clever application of the Cercignani's conjecture for the elastic Boltzmann operator derived in \cite{villcerc}.  This allowed to compare the entropy dissipation functional and the distance to a given Maxwellian distribution, more specifically, the distance to the elastic limit.  The disadvantage of this approach is that requires pointwise exponential lower bounds  and high regularity for the associated steady solution.  Such lower bounds are related to the spreading property of the collision operator and their technical extension to the case of non-constant restitution coefficient is not trivial.  In contrast, the strategy here does not uses entropy techniques at all, thus, it does not require neither pointwise lower bounds nor regularity assumptions.  This makes it well-suited for problems in which no regularity of the steady solution is available, see \cite{cold}.
\end{nb}
%
%

It is easy to deduce explicit estimates for the parameter $\la^\dagger$ in Theorem \ref{uniqueNon} under the above assumption \eqref{order2} on $e(\cdot)$:

\begin{theo}\label{theo:uniqueQuant} If $e(\cdot)$ belong to the class $\mathbb{E}_m$ for some integer $m \geq 4$ and satisfies \eqref{order2}, there is an \emph{explicit} parameter $\la^\dag \in (0,1]$ such that $\mathfrak{S}_\la$
reduces to a singleton for any $\la \in [0,\la^\dag)$.
\end{theo}
\begin{proof} Recall that the only non quantitative part in the strategy described in Section \ref{sec:strategy} was the convergence rate of $\gl$ towards $\mathcal{M}$. It is made explicit now thanks to Theorem \ref{theo:convExp} and, resuming the above strategy one gets that there exists some explicit $C_0 >0$ such that
$$\|\mathscr{L}_1(H_\la)\|_{\mathcal{X}} \leq C_0 \la^\alpha \|H_\la\|_{\mathcal{Y}} \qquad \forall \la \in (0,\la^\star)$$
where $H_\la=\fe-\gl$ with $\fe,\gl \in \mathfrak{S}_\la$ and $\la^\star$ is the parameter in Theorem \ref{theo:convExp}. Recall that $\la^\star$ can be estimated from below in an explicit way.  Using Theorem \ref{theo:convExp} one can replace \eqref{A1H} in the proof of Theorem \ref{uniqueNon} by the following quantitative estimate
$$\left|\mathcal{A}_1(H_\la)\right| \leq C_1 \la^\alpha \|H_\la\|_{\mathcal{Y}} \qquad \forall \la \in (0,\la^\star)$$
where $C_1 >0$ is an explicit constant. Then, resuming the proof of Theorem \ref{uniqueNon} yields
$$\|H_\la\|_{\mathcal{Y}} \leq  C_2\la^\alpha \|H_\la\|_{\mathcal{Y}} \qquad \forall \la \in (0,\la^\star)$$
where $C_2=\max(C_0,C_1)\|\mathcal{A}^{-1}\| >0$ is explicit.  We see therefore that $H_\la=0$ provided $\la < \la^\dag= \min\left(\la^\star,C_2^{-1/\alpha}\right).$
\end{proof}
The above uniqueness result in the quasi-elastic limit $\la \to 0$ translates to a \textit{weak thermalization} uniqueness result.
\begin{theo}\label{unique} For any restitution coefficient $e(\cdot)$ belonging to the class $\mathbb{E}_m$ with $m \geq 4$ and satisfying \eqref{order2}, there exists some explicit $\mu^\dagger >0$ such that for any $\mu \leq \mu^\dagger$, there exists a unique solution $F$ to
\begin{equation*}
\Q_e(F,F)+\mu \Delta F=0
\end{equation*}
with $\IR F(v)\d v=1$ and $\IR v F(v)\d v=0.$
\end{theo}
\begin{proof} The proof is a simple consequence of our scaling choice.  Indeed, Theorem \ref{theo:uniqueQuant} asserts that, for $\la < \la^\dag$ the steady problem
$$\Q_{\el}(\gl,\gl)+\la^\gamma \Delta \gl=0$$
admits an unique solution with unit mass and vanishing momentum. Performing the backward scaling $F(v)=\la^{-3}\gl\left(\la^{-1}v\right)$ one gets that there exists an unique solution $F$ with unit mass and vanishing momentum to the problem
$$\Q_e(F,F)+\la^{3+\gamma} \Delta F=0$$
whenever $\la < \la^\dag$. This clearly yields the conclusion with $\mu^\dag=\left(\la^\dag\right)^{3+\gamma}.$
\end{proof}
\section*{Appendix A: Properties of the collision operator}
\setcounter{equation}{0}
\renewcommand{\theequation}{A.\arabic{equation}}

We collect in this Appendix some facts about the Boltzmann collision operator important in their own right.  Some of the properties of $\Q_e$ that we will establish here are known and some others new.  We shall consider a collision operator with more general collision kernel than the hard-spheres case considered in the paper, more precisely, a collision kernel $B(u,\sigma)$ of the form
\begin{equation}
\label{Bu} {B}(u,\sigma)=\Phi(|u|)b(\widehat{u} \cdot \sigma).
\end{equation}
The kinetic potential $\Phi(\cdot)$ is a suitable nonnegative function in $\R^3$ and the \textit{angular kernel} $b(\cdot)$ is assumed in $L^{1}(-1,1)$.  The associated collision operator $\Q_{B,e}$ is defined through the weak formulation
\begin{equation}\label{Ie3B}
\int_{\mathbb{R}^{3}}\Q_{B,e}(f,f)(v)\psi(v)\d v=\frac{1}{2}\int_{\mathbb{R}^{3} \times \mathbb{R}^3}f(v)f(\vb)\mathcal{A}_{B,e}[\psi](v,\vb)\;\d\vb\d v
\end{equation}
for any test function $\psi=\psi(v)$  where
$$\mathcal{A}_{B,e}[\psi](v,\vb)=\int_{\mathbb{S}^{2}}\bigg(\psi( {v'})+\psi( {\vb'})-\psi(v)-\psi(\vb)\bigg) {B}(u, \sigma)\d\sigma$$
with $v',\vb'$ are defined in \eqref{postsig}. For any fixed vector $\widehat{u}$, the angular kernel defines a measure on the sphere through the mapping $ \sigma \in \mathbb{S}^{2}\mapsto b(\widehat{u}\cdot\sigma)\in[0,\infty]$ and we will assume it to satisfy the renormalized \textit{Grad's cut-off} assumption
\begin{equation}\label{normalization}
\left\|b\right\|_{L^{1}(\mathbb{S}^{2})}=2\pi\left\|b\right\|_{L^{1}(-1,1)}=1.
\end{equation}
For technical reasons, we shall also assume that
\begin{equation}\label{crois}
\tilde{b}\::\:x \in [-1,1] \longmapsto \tilde{b}(x)=b(x)+b(-x) \quad \text{ is non decreasing.}
\end{equation}
A particularly relevant model is the one of hard-spheres corresponding to $\Phi(|u|)=|u|$ and $b(\widehat{u} \cdot \sigma)=1/4\pi$.  For this  particular model we shall simply denote the collision operator $\Q_{B,e}$ by $\Q_e$.

\subsection*{A.1. Convolution-like estimates for $\Q^+_{B,e}$} We begin by recalling some of the regularity and integrability properties of the gain part $\Q^+_e$ established in \cite{AloLo1} and \cite{AloCar}. We start first with Young-like estimates in $L^p_\eta$ with $\eta \geq 0$.
\begin{theoA}[\textbf{Alonso-Carneiro-Gamba \cite{AloCar}}]\label{alogam}
Assume that the collision kernel $B(u,\sigma)=\Phi(|u|)b(\widehat{u} \cdot \sigma)$ satisfies \eqref{normalization} and  $\Phi(\cdot) \in L^\infty_{-k}$ for some $k \in \R$. In addition, assume that $e(\cdot)$ fulfills Assumption \ref{HYP}. Let $1 \leq p,q,r\leq \infty$ with $1/p+1/q=1+1/r$. Then, for any $\alpha \geq 0$, there exists $C_{p,r,\alpha,k}(b)$ such that
$$\|\Q^+_{B,e}(f,g)\|_{L^r_\alpha} \leq C_{r,p,\alpha,k}(b)\left\|\Phi\right\|_{L^\infty_{-k}}\|f\|_{L^p_{\alpha+k}}\,\|g\|_{L^q_{\alpha+k}}$$
where the constant $C_{r,p,\alpha,k}(b)$ is given by
\begin{multline}\label{Crpakb}
C_{r,p,\alpha,k}(b)=c_{k,\alpha,r}\left(\int_{-1}^1 \left(\dfrac{1-s}{2}\right)^{-3/2r'}b(s)\d s\right)^{\frac{r'}{q'}}\\
\left(\int_{-1}^1 \left(\dfrac{1+s}{2}+(1-\beta_0)^2\frac{1-s}{2}\right)^{-\frac{3}{2r'}}b(s)\d s\right)^{\frac{r'}{p'}}
\end{multline}
for some numerical constant $c_{k,\alpha,r}$ independent of $b$ and $e(\cdot)$ and with $\beta_0=\beta(0)=\frac{1+e(0)}{2}.$
\end{theoA}
Theorem A.\ref{alogam} has been modified in \cite{AloLo1} to provide $L^p_\eta$ bounds with $\eta \leq 0$.
\begin{theoA}
\label{alo}
Assume that the collision kernel $B(u,\sigma)=\Phi(|u|)b(\widehat{u} \cdot \sigma)$ satisfies \eqref{normalization} and  $\Phi(\cdot) \in L^\infty_{-k}$ for some $k \in \R$. In addition, assume that $e(\cdot)$ fulfills Assumption \ref{HYP}. Then, for any $1 \leq p \leq \infty$  and $\eta \in \R$, there exists $\mathbf{C}_{\eta,p,k}(B)>0$ such that
$$
\left\|\Q^+_{B,e}(f,g)\right\|_{L^p_\eta} \leq \mathbf{C}_{\eta,p,k}(B) \,\|f\|_{L^1_{|\eta +k|+|\eta|}} \,\|g\|_{L^p_{\eta +k}}
$$
where the constant $\mathbf{C}_{\eta,p,k}(B)$ is given by:
\begin{equation*}
\mathbf{C}_{\eta,p,k}(B)=c_{k,\eta,p}\;\gamma(\eta,p,b)\left\|\Phi\right\|_{L^\infty_{-k}}
\end{equation*}
with a constant $c_{k,\eta,p} >0$ depending only on $k,\eta$ and $p$.  Furthermore, the dependence on the angular kernel is given by
\begin{equation}\label{Cetab}
\gamma(\eta,p,b)=\int_{-1}^1 \left(\frac{1-s}{2}\right)^{-\frac{3+\eta_+}{2p'}}b(s)\d s,
\end{equation}
where $1/p+1/p'=1$ and $\eta_+$ is the positive part of $\eta$. Similarly, there exists $\widetilde{\mathbf{C}}_{\eta,p,k}(B)>0$ such that
$$
\left\|\Q^+_{B,e}(f,g)\right\|_{L^p_\eta} \leq \widetilde{\mathbf{C}}_{\eta,p,k}(B) \,\|g\|_{L^1_{|\eta +k|+|\eta|}} \,\|f\|_{L^p_{\eta +k}}
$$
where the constant $\widetilde{\mathbf{C}}_{\eta,p,k}(B)$ is given by
\begin{equation*}
\widetilde{\mathbf{C}}_{\eta,p,k}(B)=\widetilde{c}_{k,\eta,p} \;\widetilde{\gamma}(\eta,p,b)\left\|\Phi\right\|_{L^\infty_{-k}}
\end{equation*}
for some constant $\widetilde{c}_{k,\eta,p} >0$ depending only on $k,\eta$ and $p$.  The dependence on the angular kernel is given by
\begin{equation*}
\widetilde{\gamma}(\eta,p,b)=\int_{-1}^1 \left(\frac{1+s}{2}+\left(1-\beta_0\right)^2\frac{1-s}{2}\right)^{-\frac{3+\eta_+}{2p'}}b(s)\d s
\end{equation*}
where $1/p+1/p'=1$ and $\beta_0=\beta(0)=\frac{1+e(0)}{2}.$
\end{theoA}
\begin{corA}\label{quadratic-estim}
Assume that the collision kernel $B(u,\sigma)=\Phi(|u|)b(\widehat{u} \cdot \sigma)$ satisfies \eqref{normalization} and  $\Phi(\cdot) \in L^\infty_{-k}$ for some $k \in \R$. In addition, assume that $e(\cdot)$ fulfills Assumption \ref{HYP}. Then, for any $1 \leq p \leq \infty$  and $\eta \in \R$, there exists a numerical constant $C_{k,\eta,p} >0$ (which does not depend on $B(\cdot,\cdot)$) such that
$$
\left\|\Q^+_{B,e}(f,f)\right\|_{L^p_\eta} \leq C_{k,\eta,p}\|b\|_{L^1(\mathbb{S}^2)}\|\Phi\|_{L^\infty_{-k}} \,\|f\|_{L^1_{|\eta +k|+|\eta|}} \,\|f\|_{L^p_{\eta +k}}.
$$
\end{corA}

\subsection*{A.2. Useful change of variables for non constant restitution coefficient} We establish here several changes of variables that are useful for the study of the continuity properties given in Section \ref{sec:continuity}.
\begin{defiA} A restitution coefficient $e(\cdot) \::\:r \mapsto e(r) \in [0,1]$ is said to belong to the class $\mathcal{C}_0$ if $e(\cdot)$ satisfies the following:
\begin{enumerate}
\item The mapping  $r \in \mathbb{R}_+ \mapsto e(r) \in (0,1]$ is absolutely continuous and non-increasing.
\item The mapping $r\in\mathbb{R}^{+}\mapsto \vartheta_e(r):=r\;e(r )$ is strictly increasing.
\item $e(0)=1$.
\end{enumerate}
Moreover, for a given $\gamma >0$, we shall say that $e(\cdot)$ belongs to the class $\mathcal{C}_\gamma$ if it belongs to $\mathcal{C}_0$ and
there exists $\mathfrak{a }>0$ such that
$$e(r) \simeq 1-\mathfrak{a}r^\gamma \qquad \text{ as } \qquad r \simeq 0.$$
\end{defiA}
\begin{nbA}
Recall that if $e(\cdot)$ belongs to the class $\mathcal{C}_\gamma$ then $$\ell_\gamma(e):=\sup_{r >0} \dfrac{1-e(r)}{r^\gamma} < \infty.$$
\end{nbA}
\begin{lemmeA}\label{lemmB1} Define $\beta_e(r)=\frac{1+e(r)}{2}$ and the mapping
$$\eta_e\::\:r \in \mathbb{R}^+ \longmapsto r\beta_e(r).$$
Then, $e(\cdot)$ belongs to the class $\mathcal{C}_0$ if and only if $\eta_e(\cdot)$ is strictly increasing and differentiable with
$$\frac{r}{2} \leq \eta_e(r) \leq r\,; \quad \frac{1}{2} \leq  \eta_e'(r) \leq \frac{\eta_e(r)}{r} \quad \text{for any} \quad  r > 0 \quad \text{and} \quad \eta_e'(0)=1.$$
Equivalently, the inverse mapping $\alpha_e(\cdot)$ of $\eta_e(\cdot)$ satisfies
$$r \leq \alpha_e(r) \leq 2r\,; \quad  \frac{\alpha_e(r)}{r} \leq  \alpha_e'(r) \leq 2 \quad \text{for any} \quad  r > 0 \quad \text{and} \quad \alpha_e'(0)=1.$$
\end{lemmeA}
\begin{lemmeA}\label{lemmB1.5}
\cite[Lemma 2.3]{MMR}
For any  $\sigma \in \S$ and $\delta \in (0,2)$ define the cone
\begin{equation}\label{omegadelta}\Omega_\delta=\Omega_\delta(\sigma)=\bigg\{u \in \R^3\setminus\{0\}\;;\;\us > \delta-1\bigg\}.
\end{equation}
Define the mapping $\Phi_\sigma$ as
$$\Phi_\sigma\::\: u \in \R^3 \longmapsto \Phi_\sigma(u)=\dfrac{u+|u|\sigma}{2}.$$
Then, $\Phi_\sigma$ is a $C^\infty$-diffeomorphism from $\Omega_\delta$  onto ${\Omega}_{\delta^\star}$ where $\delta^\star={1+\sqrt{\frac{\delta}{2}}}$ and with Jacobian $J_\sigma(u)=\frac{1}{8}\left(1+\us\right).$
Its inverse mapping  $\varphi_\sigma=\Phi_\sigma^{-1}$ is given by $\varphi_\sigma(w)=2w-\dfrac{|w|}{\widehat{w}\cdot \sigma}\sigma.$
\end{lemmeA}
With the notations of Lemma A.\ref{lemmB1} and Lemma A.\ref{lemmB1.5} we can establish the following change of variables formula which generalizes \cite[Prop. 3.2]{MMR}.
\begin{lemmeA}\label{change} For a given restitution coefficient  $e(\cdot)$ in the class $\mathcal{C}_0$, one defines the mapping
$$\Pi_e\::\:w \mapsto z=\beta_e(|w|)w=\dfrac{1+e(|w|)}{2}w=\Pi_e(w).$$
Then, for any $\delta >0$, $\Pi_e$ is a $C^\infty$-diffeomorphism from $\Omega_\delta$  onto itself with Jacobian $\mathcal{J}_e(|z|)$  given by
\begin{equation}\label{Jla}
\mathcal{J}_e(\varrho)=\frac{1}{2}\left(1+\vartheta'_e\left(\alpha_e(\varrho)\right)\right)\beta_e^2\left(\alpha_e(\varrho)\right)\qquad \forall \varrho \geq 0.\end{equation}
The inverse mapping $\pi_e=\Pi_e^{-1}$ is given by
$$\pi_e(z)=\dfrac{\alpha_e(|z|)}{|z|}z=\dfrac{z}{\beta_e(\alpha_e(|z|))}.$$
If one combines the two applications $\Pi_e \circ \Phi_\sigma$ we get the change of variables
$$u \longmapsto z=\beta_e \left(|\Phi_\sigma(u)|\right)\Phi_\sigma(u)$$ which is a $C^\infty$-diffeomorphism from $\Omega_\delta$ onto $\Omega_{\delta^\star}$. Its inverse mapping is given by
$$z \longmapsto \zeta_e(z)=\varphi_\sigma \circ \pi_e(z)$$
with Jacobian given by $J_\sigma(\zeta_e(z))\mathcal{J}_e(z).$  One has
$$\zeta_e(z)=\mu_e(z)\varphi_\sigma(z) \qquad \text{ with }  \qquad \mu_e(z)=\dfrac{\alpha_e(|z|)}{|z|}.$$
\end{lemmeA}
\begin{proof} The properties of $\Pi_e$ are proven by direct calculations and noticing that if $z=\Pi_e(w)$ then
$$\widehat{z}=\widehat{w} \qquad \text{ and } \qquad |w|=\alpha_e(|z|).$$
With this identity, one can computes the Jacobian of the transformation passing to polar coordinates. The final expression of $\zeta_e(z)$ is immediate after noticing that $\varphi_\sigma(r w)=r \varphi_\sigma(w)$ for any $r >0$ and any $w \in \R^3$.
\end{proof}
\begin{nbA} Observe that for $e(\cdot)$ belonging to $\mathcal{C}_0$, since $\vartheta_e'(r)=re'(r)+e(r) \leq 1$ for any $r \geq 0$ and $\beta_e(r) \in [\frac{1}{2},1]$, one has the universal bound
\begin{equation}\label{boundsJe} \frac{1}{8} \leq \mathcal{J}_e (\varrho) \leq 1 \qquad \forall \varrho \geq 0.\end{equation}
\end{nbA}
\begin{lemmeA}\label{convex}
Let $e(\cdot)$ be a restitution coefficient in the class $\mathcal{C}_0$ and let $s \in [0,1]$. Then, there exists a restitution coefficient $\tilde{e}_s(\cdot)$ belonging to $\mathcal{C}_0$ such that
$$1+s\left(\mu_e(z)-1\right)=\mu_{\tilde{e}_s}(z) \qquad \forall z \in \R^3$$
where $\mu_e$ has been defined in Lemma A.\ref{change}.
\end{lemmeA}
\begin{proof} Define $\mu(z)=1+s(\mu_e(z)-1)=(1-s) + s\mu_e(z)$ and recall that
$$\mu_e(z)=\dfrac{\alpha_e(|z|)}{|z|}.$$
In order to prove that there exists $\tilde{e}_s(\cdot)$ in the class $\mathcal{C}_0$ such that $\mu=\mu_{\tilde{e}_s}$, thanks to Lemma A. \ref{lemmB1}, it suffices to prove that the mapping $\alpha \::\:r \mapsto (1-s)r+s\alpha_e(r)$ satisfies
$$r \leq \alpha(r) \leq 2r\,;  \qquad  \frac{\alpha(r)}{r} \leq  \alpha '(r) \leq 2 \qquad \text{ for any } \quad  r > 0\qquad \text{ and } \qquad \alpha'(0)=1.$$
Since $\alpha_e$ satisfies all these properties, it follows that the same is true for $\alpha$.
\end{proof}
The following proposition is reminiscent of the so-called \textit{cancellation Lemma} for the classical Boltzmann operator \cite{Alexandre,VillCanc}.
\begin{propoA}\label{prop:repre} Let $e(\cdot)$ be a given restitution coefficient belonging to the class $\mathcal{C}_0$ and let
$$\mathcal{B}(u,\us)=\Theta(|u|)\,b(\us)$$
be a given collision kernel with $\Theta(r) \geq 0$ and $b(s)=b(-s)$ with $\mathrm{Supp}\:b \in [-1+\delta,1-\delta]$ for some $\delta >0$. Let $\Q^+_{\mathcal{B},e}$ and $\Q^+_{\mathcal{B},1}$ denote  the positive part of the collision operator associated to $\mathcal{B}$ with restitution coefficient $e(\cdot)$ and elastic interactions respectively. For any test function $\psi$ and any given $f,g$, one has
\begin{multline}\label{represent}
\IR \left[\Q^+_{\mathcal{B},e}(f,g)-\Q^+_{\mathcal{B},1}(f,g)\right]\psi\,\d v=\\
\frac{1}{2}\IR f(v)\d v\IS \d \sigma\int_{\Omega^\star_\delta} \psi(v+z)\left[\frac{1}{\mathcal{J}_e(|z|)}F_{v,\sigma}\left(\zeta_e(z)\right)-F_{v,\sigma}(\varphi_\sigma(z))\right] \d z
\end{multline}
where $F_{v,\sigma}(u)=\Theta(|u|)\widetilde{b}(\us)g(v+u)$  with $\widetilde{b}(\us)=\frac{b(\us)}{J_\sigma(u)}=\frac{8b(\us)}{1+\us}$, $u \in \R^3$.
\end{propoA}
\begin{proof} We set for simplicity
$$I_e=\IR \Q^+_{\mathcal{B},e}(f,g)(v)\psi(v)\d v\ \ \text{and}\ \ I_1= \IR \Q^+_{\mathcal{B},1}(f,g)\psi(v)\d v.$$
Thus,
\begin{equation*}
I_e =\dfrac{1}{2}\int_{\R^3 \times \R^3 \times \S}\mathcal{B}(u,\us) f(v)g(\vb)\psi(v'_e)\d v \d\vb \d\sigma\end{equation*}
where $u=v-\vb$ and $v'_e=v-\beta (|u| \sqrt{\tfrac{1-\widehat{u} \cdot \sigma}{2}})\frac{u-|u|\sigma}{2}$
is the post-collisional velocity associated to $e(\cdot)$.  In particular, the change of variables $\vb \to u$ yields
$$I_e=\dfrac{1}{2}\IR f(v)\d v\IS \d\sigma \IR\mathcal{B}(u,\us) g(v-u)\psi(v'_e)\d v \d u \d\sigma.$$
The change of variables $u \to -u$ in the last integral gives
$$I_e=\frac{1}{2}\IR f(v)\d v\IS \d\sigma \int_{\Omega_\delta} \Theta(|u|)b (\us)g(v+u)
\psi\big(v+\beta \left(|\Phi_\sigma(u)|\right)\Phi_\sigma(u)\big)\d u$$
where we used that, for fixed $\sigma$, the support of $b$ is included in $[-1+\delta,1-\delta]$ so that the variable $u$ belongs to the cone $\Omega_\delta$ defined by \eqref{omegadelta}. With the notations of Lemma A.\ref{change}, we  perform the change of variables  $z=\Pi_e \circ \Phi_\sigma(u)$ in the previous integral to get
\begin{equation}\label{3}
I_e=\frac{1}{2}\IR f(v)\d v\IS \d\sigma \int_{{\Omega}_\delta^\star} \psi(v+z)\frac{1}{\mathcal{J}_e(|z|)}F_{v,\sigma}\left(\zeta_e(z)\right) \d z
\end{equation}
where $F_{v,\sigma}(u)=\Theta(|u|)\widetilde{b}(\us)g(v+u).$ In the same way, for the particular case of elastic interactions (i.e. for $e\equiv1$) since  $\zeta_1(z)=\varphi_\sigma(z)$  and $\mathcal{J}_1(|z|)=1$  one simply has
$$I_1=\frac{1}{2}\IR f(v)\d v\IS \d\sigma \int_{{\Omega}_\delta^\star} \psi(v+z)F_{v,\sigma}\left(\varphi_\sigma(z)\right) \d z$$
which clearly gives \eqref{represent}.
\end{proof}
\begin{propoA}\label{propB:cont} Under the assumptions of Proposition A.\ref{prop:repre}, if $e(\cdot)$ belongs to the class $\mathcal{C}_\gamma$ for some $\gamma >0$ then there exists $C_e >0$ such that
\begin{multline*}
\left|\IR \left[\Q^+_{\mathcal{B},e}(f,g)-\Q^+_{\mathcal{B},1}(f,g)\right]\psi\,\d v\right| \leq C_e   \IR \Q^+_{\mathcal{B}_\gamma,1}(f , g )\,|\psi(v)|\d v \\
 + 2^{\gamma+6}\ell_\gamma(e)\int_0^1\d s\int_{\R^3} \Q^+_{\mathcal{\overline{B}_\gamma},\tilde{e}_s}\left(f, h\right)|\psi(v)|\d v \end{multline*}
where $h(v)=g(v)+|\nabla g(v)|$.  The kernels $\mathcal{B}_\gamma$ and $\mathcal{\overline{B}_\gamma}$ are given by
$$\mathcal{B}_\gamma(u,\us)=\mathcal{B}(u,\us)|u|^\gamma, \qquad \mathcal{\overline{B}_\gamma}(u,\us)=\max(\mathcal{B}(u,\us),|\nabla_u \mathcal{B}(u,\us)|)|u|^{\gamma+2},$$
moreover, $\tilde{e}_s(\cdot)$ is a given restitution belonging to the class $\mathcal{C}_0$ for any $s \in [0,1]$.
\end{propoA}
\begin{proof} Using the notation of the Proposition A.\ref{prop:repre}, we set $D_e=I_e-I_1$.   Thanks to \eqref{represent}, we may split $D_e$ as $D_e=D_{e,1}+D_{e,2}$ with
$$D_{e,1}=\frac{1}{2}\IR f(v)\d v\IS \d\sigma \int_{{\Omega}_\delta^\star} \psi(v+z)\bigg[F_{v,\sigma}\left(\zeta_e(z)\right)-F_{v,\sigma}\left(\varphi_\sigma(z)\right)\bigg]\frac{\d z}{\mathcal{J}_e(|z|)}$$
and
$$D_{e,2}=\frac{1}{2}\IR f(v)\d v\IS \d\sigma \int_{{\Omega}_\delta^\star} \psi(v+z)\left[\frac{1}{\mathcal{J}_e(|z|)}-1\right]F_{v,\sigma}\left(\varphi_\sigma(z)\right)\d z.$$
We begin estimating $|D_{e,1}|$ which is the more involved part. For fixed $v,\sigma$, we use the following representation formula
$$F_{v,\sigma}(V)-F_{v,\sigma}(U)=\int_0^1 \nabla F_{v,\sigma}(U+s(V-U))\cdot U \d s$$
with $V=\zeta_e(z)$ and $U=\varphi_\sigma(z)$ to get
\begin{multline*}\bigg[F_{v,\sigma}\left(\zeta_e(z)\right)-F_{v,\sigma}\left(\varphi_\sigma(z)\right)\bigg]=
\left(\mu_e(z)-1\right) \int_0^1  \varphi_\sigma(z) \cdot \nabla F_{v,\sigma}\big(\varphi_\sigma(z)+s(\mu_e(z)-1)\varphi_\sigma(z)\big)\d s.
\end{multline*}
Therefore
\begin{multline*}
 D_{e,1} =\dfrac{1}{2} \IR  f(v) \d v \IS \d\sigma \int_{{\Omega}_\delta^\star} \psi(v+z) \left(\mu_e(z)-1\right)\frac{\d z}{\mathcal{J}_e(|z|)}\\
\int_0^1  \varphi_\sigma(z) \cdot \nabla F_{v,\sigma}\left(\varphi_\sigma(z)+s(\mu_e(z)-1)\varphi_\sigma(z)\right)  \d s
\end{multline*}
Now, according to Lemma A. \ref{convex}, for any $s \in [0,1]$, there exists a restitution coefficient $\tilde{e}_s(\cdot)$ in $\mathcal{C}_0$ such that
$\mu_{\tilde{e}_s}(z)=1+s(\mu_e(z)-1).$ Therefore,
$$\varphi_\sigma(z)+s(\mu_e(z)-1)\varphi_\sigma(z)=\zeta_{\tilde{e}_s}(z)$$
and, performing the backward change of variable $z \mapsto u=\zeta_{\tilde{e}_s}^{-1}(z)$ with Jacobian $\d z=J_\sigma(u)\mathcal{J}_{\tilde{e}_s}(\Pi_{\tilde{e}_s} \circ\Phi_\sigma(u))\d u$ we get
\begin{multline*}
D_{e,1}  = 4\int_0^1 \d s \IR |f(v)|\d v \IS \d\sigma \int_{{\Omega}_\delta}\psi\left(v+\Pi_{\tilde{e}_s} \circ \Phi_\sigma(u)\right)
\bigg[\mu_e(\Pi_{\tilde{e}_s} \circ\Phi_\sigma(u))-1\bigg]\times \\
\times (1+\us) \dfrac{\mathcal{J}_{\tilde{e}_s}(\Pi_{\tilde{e}_s} \circ\Phi_\sigma(u))}{\mathcal{J}_e(\Pi_{\tilde{e}_s} \circ\Phi_\sigma(u))}
 \nabla F_{v,\sigma}\left(u\right)\cdot \varphi_\sigma \big(\Pi_{\tilde{e}_s} \circ\Phi_\sigma(u)\big)\d u.
\end{multline*}
Since $\tilde{e}_s(\cdot)$ is a restitution coefficient in the class $\mathcal{C}_0$, thanks to the universal bounds \eqref{boundsJe} we see that
$$\underset{\underset{{s \in (0,1), \sigma \in \S}}{ u \in \R^3}}{\sup}(1+\us)\dfrac{\mathcal{J}_{\tilde{e}_s}(\Pi_{\tilde{e}_s} \circ\Phi_\sigma(u))}{\mathcal{J}_e(\Pi_{\tilde{e}_s} \circ\Phi_\sigma(u))} \leq 16 < \infty.$$
Moreover, it is easy to see that
$$\big|\mu_e(z)-1\big|=\dfrac{|z|}{\beta_e(\alpha_e(|z|))}\left|1-\beta_e(\alpha_e(|z|))\right|  \leq \ell_\gamma(e) \,|z| \alpha_e(|z|)^\gamma \leq 2^\gamma \ell_\gamma(e) |z|^{\gamma+1} \qquad \forall z \in \R^3.$$
Since $|\Pi_{\tilde{e}_s} \circ \Phi_\sigma(u)| \leq |\Phi_\sigma(u)|\leq |u|$ and $|\varphi_\sigma \big(\Pi_{\tilde{e}_s} \circ\Phi_\sigma(u)\big)| \leq |u|$, we get
$$|D_{e,1}| \leq 2^{\gamma+6} \ell_\gamma(e)\IR |f(v)|\d v \IS \d\sigma \int_{{\Omega}_\delta}|\psi(\widetilde{v}_s)||u|^{\gamma+2} \d u
\int_0^1 \left|\nabla F_{v,\sigma}(u)\right|\d s$$
where $\widetilde{v}_s=v+ \Pi_{\tilde{e}_s} \circ \Phi_\sigma(u)$.  One can check that
$$\,|u|^{\gamma+2}\,|\nabla F_{v,\sigma}(u)| \leq  \mathcal{\overline{B}_\gamma}(u,\us)\left(g(v+u)+|\nabla g(v+u)|\right)=\mathcal{\overline{B}_\gamma}(u,\us) h(v+u).$$
Therefore, performing again the change of variable $u \to -u$, we obtain
$$|D_{e,1}| \leq 2^{\gamma+6} \ell_\gamma(e)\int_0^1\d s\int_{\R^3 \times \R^3 \times \S}  |f(v)|\mathcal{\overline{B}_\gamma}(u,\us) \,h(v-u)
|\psi(v'_s)| \d v\d u\d \sigma$$
where $v'_s=v+ \Pi_{\tilde{e}_s} \circ \Phi_\sigma(-u)=v-\beta_{\tilde{e}_s}\left(|u| \sqrt{\tfrac{1-\widehat{u} \cdot \sigma}{2}}\right)\frac{u-|u|\sigma}{2}$ is the post-collisional velocity associated to the restitution coefficient $\tilde{e}_s$. This proves that
\begin{equation}\label{De1}|D_{e,1}| \leq 2^{\gamma+4} \ell_\gamma(e)\int_0^1\d s\int_{\R^3} \Q^+_{\mathcal{\overline{B}_\gamma},\tilde{e}_s}\left(f, h\right)(v)|\psi(v)|\d v \end{equation}
where $\Q^+_{\mathcal{\overline{B}_\gamma},\tilde{e}_s}$ is the collision operator associated to the kernel $\mathcal{\overline{B}_\gamma}$ and the restitution coefficient $\tilde{e}_s$.  For the estimate of $|D_{e,2}|$ it is enough to prove that there exists $C_e >0$ such that
\begin{equation}\label{jaco}\left[\frac{1}{\mathcal{J}_e(|z|)}-1\right] \leq C_e  |z|^\gamma \qquad \forall z \in \R^3.\end{equation}
Indeed, if \eqref{jaco} holds then
$$|D_{e,2}| \leq C_e  \IR |f(v)|\d v\IS \d\sigma \int_{{\Omega}_\delta^\star} |\psi(v+z)|\,|F_{v,\sigma}\left(\varphi_\sigma(z)\right)|\,|z|^\gamma\d z.$$
Performing the \textit{backward} change of variables $u=\varphi_\sigma(z)$ as before
$$|D_{e,2}| \leq C_e \IR |f(v)|\d v\IS \d\sigma \int_{{\Omega}_\delta} |\psi(v+\Phi_\sigma(u))|\,|F_{v,\sigma}\left(u\right)|\,|u|^\gamma J_\sigma(u)\d u$$
where we used that $|\Phi_\sigma(u)| \leq |u|$.  Changing again the variable $u$ into $-u$ we get
$$|D_{e,2}| \leq C_e \int_{\R^3 \times \R^3 \times \S} \mathcal{B}_\gamma(u,\us) |f(v)|\,|g(\vb)|\,|\psi(v'_1)|\d v \d\vb \d\sigma$$
where $v'_1$ is the post-collisional velocity associated to \textit{elastic interactions}, that is, $v'_1=v-\frac{u-|u|\sigma}{2}$.  This gives
$$|D_{e,2}|\leq C_e \IR \Q^+_{\mathcal{B}_\gamma,1}(f,g)(v)\,|\psi(v)|\d v$$
which, combined with \eqref{De1} yields the result.  The idea to prove \eqref{jaco} is to evaluate $\mathcal{J}_e(\varrho)$ for $\varrho \simeq 0$. Since  $e(r) \simeq 1-\mathfrak{a}r^\gamma$ for $r \simeq 0$ one checks that
$$\dfrac{1}{2}\left(1+\vartheta_e'(r)\right)\beta_e^2(r) \simeq 1+\frac{\mathfrak{a}(\gamma-1)}{2}\,r^\gamma \qquad \text{ for } \qquad r \simeq 0.$$
Since $\alpha_e(\varrho) \simeq \varrho$ for $\varrho \simeq 0$, we get $\mathcal{J}_e(\varrho) \simeq 1+\frac{\mathfrak{a}(\gamma-1)}{2}\varrho^\gamma$ as $\varrho \simeq 0.$ Therefore, $$\sup_{\varrho \geq 0} \frac{1-\mathcal{J}_e(\varrho)}{\varrho^\gamma}  < \infty$$ and \eqref{jaco} follows for some constant $C_e$ depending only on $e(\cdot)$.
\end{proof}
\begin{nbA}\label{nbB:Cl} Notice that, defining as in Section \ref{sec:stea}, the rescaled restitution coefficient $e_\la(r)=e(\la r)$ $(\la >0$), one sees from the above reasoning that
$$\mathcal{J}_{\el}(\varrho) \simeq 1+\frac{\mathfrak{a}(\gamma-1)}{2}\la^\gamma \varrho^\gamma \quad \text{ as } \quad \la \to 0.$$
In particular, for $\la$ small enough the constant $C_{\el}$ appearing in the above Proposition satisfies
$$C_{\el}  \leq \mathfrak{a}(1-\gamma)\la^\gamma.$$
This property will be important in Section \ref{sec:continuity}.
\end{nbA}
\begin{exaA} Assume $e(\cdot)$ is the restitution coefficient corresponding to visco-elastic hard-spheres
$$e(r)=1+ \sum_{k=1}^\infty (-1)^k a_k r^{\frac{k}{5}}, \qquad r\geq 0.$$
Then, setting $H_e(r)=\dfrac{1}{2}\left(1+\vartheta_e'(r)\right)\beta_e^2(r)$, it is not difficult to prove that there is some explicit constant $C >0$ such that
$$|H_e(r)-1| \leq C(1-e(r)), \qquad \forall r \geq 0.$$
In particular, $|H_e(r)-1|\leq C\ell_\gamma(e)r^\gamma$ for any $r \geq 0$ from which we deduce that \eqref{jaco} follows with a constant $C_e$ proportional to $\ell_\gamma(e)$. Since $\ell_\gamma(\el)=\la^\gamma\ell_\gamma(e)$, there exists some constant $c >0$ such that $C_{\el} \leq c\la^\gamma$ for any $\la \in (0,1]$ (not just for $\la$ small enough as in the previous remark).
\end{exaA}
\subsection*{A.3. About the energy identity} Recall that, for any solution $\gl$ to \eqref{rescaled}, one has the identity
\begin{equation*}
6\varrho=\dfrac{1}{\lambda^{3+\gamma}}\IRR \gl(v)\gl(\vb)\mathbf{\Psi}_{e}(\la^2 |v-\vb|^2)\d v\d\vb\end{equation*}
where $\varrho=\IR \gl(v)\d v$ and $\mathbf{\Psi}_{e}(\cdot)$ is defined by \eqref{Psie}. Notice that  for any fixed $r >0$,
$$\dfrac{1}{\lambda^{3+\gamma}}\mathbf{\Psi}_{e}(\la^2 r^2) \simeq \frac{\mathfrak{a}}{4+\gamma} r^{3+\gamma} \qquad \text{as} \qquad \la \simeq 0.$$
Define for simplicity
$$\zeta_\la(r^2)=\dfrac{1}{\lambda^{3+\gamma}}\mathbf{\Psi}_{e}(\la^2 r^2) \qquad \text{ and } \qquad \zeta_0(r^2)=\frac{\mathfrak{a}}{4+\gamma} r^{3+\gamma},$$
and the two functionals
$$\mathcal{I}_\la(f,g)=\IRR f(v)g(\vb)\zeta_\la\left(|v-\vb|^2\right)\d v\d\vb,$$
and $$\mathcal{I}_0(f,g)=\IRR f(v)g(\vb)\zeta_0\left(|v-\vb|^2\right)\d v\d\vb.$$
We will write $\mathcal{I}_\la(f)=\mathcal{I}_\la(f,f)$ and $\mathcal{I}_0(f)=\mathcal{I}_0(f,f)$. Then, one has the following
\begin{lemmeA}\label{lem:dissTem} There exist a positive constant $A_\gamma >0$ such that, for any $\delta >0$ and $\varepsilon >0$ there exists $\la_0 \in (0,1)$ such that
\begin{multline*}
\sup_{\la \in (0,\la_0)}\left|\mathcal{I}_\la(f_1,g_1)-\mathcal{I}_0(f_2,g_2)\right|\leq A_\gamma\left(\|f_1-f_2\|_{L^1_{3+\gamma}}\|g_1\|_{L^1_{3+\gamma}} +\|g_1-g_2\|_{L^1_{3+\gamma}}\|f_2\|_{L^1_{3+\gamma}}\right) \\
+ \varepsilon\left(\|f_2\|_{L^1}\|g_2\|_{L^1}  + \|f_2\|_{L^1_{3+\gamma+\delta}} \|g_2\|_{L^1_{3+\gamma+\delta}}\right).
\end{multline*}
In particular, if $g \in L^1_{3+\gamma+\delta}$ and $f \in L^1_{3+\gamma}$, then
$$\limsup_{\la \to 0} \left|\mathcal{I}_\la(f)-\mathcal{I}_0(g)\right| \leq A_\gamma\|f-g\|_{L^1_{3+\gamma}} \left(\|f\|_{L^1_{3+\gamma}}+\|g\|_{L^1_{3+\gamma}}\right).$$
\end{lemmeA}
\begin{proof} Note that $\left|\mathcal{I}_\la(f_1,g_1)-\mathcal{I}_0(f_2,g_2)\right| \leq \mathcal{D}^1_\la+\mathcal{D}^2_\la+\mathcal{D}^3_\la$ where
\begin{equation*}\begin{split}
\mathcal{D}^1_\la&=\IRR \left|f_1(v)-f_2(v)\right|\,|g_1(\vb)|\zeta_\la\left(|v-\vb|^2\right)\d v \d\vb \\
\mathcal{D}^2_\la&=\IRR |f_2(v)|\,\left|g_1(\vb)-g_2(\vb)\right|\,\zeta_\la\left(|v-\vb|^2\right)\d v\d\vb\\
\mathcal{D}^3_\la&= \IRR |f_2(v)|\,|g_2(\vb)|\,\left|\zeta_\la(|v-\vb|^2)-\zeta_0\left(|v-\vb|^2\right)\right|\d v.\d\vb
\end{split}
\end{equation*}
Let us investigate separately these three terms. Since there is some positive constant $K_\gamma >0$ such that $\mathbf{\Psi}_e(r^2) \leq K_\gamma r^{3+\gamma}$, it is clear that $\zeta_\la\left(|v-\vb|^2\right)\leq K_\gamma |v-\vb|^{3+\gamma} \leq 2^{\frac{3+\gamma}{2}}K_\gamma \langle v \rangle^{3+\gamma}\langle \vb\rangle^{3+\gamma}$ for any $(v,\vb)$. Therefore
\begin{multline*}
\mathcal{D}_\la^1 \leq 2^{\frac{3+\gamma}{2}}K_\gamma \IRR \left|f_1(v)-f_2(v)\right|\,|g_1(\vb)|\langle v \rangle^{3+\gamma}\langle \vb\rangle^{3+\gamma}\d v\d\vb\\
=2^{\frac{3+\gamma}{2}}K_\gamma\|f_1-f_2\|_{L^1_{3+\gamma}}\|g_1\|_{L^1_{3+\gamma}}.
\end{multline*}
In the same way
$$\mathcal{D}_\la^2 \leq 2^{\frac{3+\gamma}{2}}K_\gamma\|g_1-g_2\|_{L^1_{3+\gamma}}\|f_2\|_{L^1_{3+\gamma}}.$$
Regarding the term $\mathcal{D}^3_\la$, set $\omega_\la(R)=\sup_{0 \leq r \leq R}\left|\zeta_\la(r^2)-\zeta_0(r^2)\right|$ for any $R > 0$.  It is clear that for any fixed $R >0$ one has $\lim_{\la \to 0}\omega_\la(R)=0$.  Let $R >0$ be fixed and split $\mathcal{D}_\la^3$ as
\begin{multline*}
\mathcal{D}_\la^3=\int_{|v-\vb| \leq R}|f_2(v)|\,|g_2(\vb)|\,\left|\zeta_\la(|v-\vb|^2)-\zeta_0\left(|v-\vb|^2\right)\right|\d v\d\vb\\
+\int_{|v-\vb| > R}|f_2(v)|\,|g_2(\vb)|\,\left|\zeta_\la(|v-\vb|^2)-\zeta_0\left(|v-\vb|^2\right)\right|\d v\d\vb\\
\leq \omega_\la(R) \|f_2\|_{L^1}\|g_2\|_{L^1}  +(K_\gamma+C_\gamma)\int_{|v-\vb| > R}|f_2(v)|\,|g_2(\vb)|\,\left|v-\vb\right|^{3+\gamma}\d v\d\vb
\end{multline*}  where we used the fact that $\left|\zeta_\la(|v-\vb|^2)-\zeta_0\left(|v-\vb|^2\right)\right| \leq (K_\gamma+C_\gamma)|v-\vb|^{3+\gamma}$ for any $(v,\vb)$.  Consequently, for any $\delta >0$,
$$\mathcal{D}_\la^3 \leq \omega_\la(R)\|f_2\|_{L^1}\|g_2\|_{L^1} + \frac{K_\gamma+C_\gamma}{R^{\delta}} \IRR |f_2(v)|\,|g_2(\vb)|\,|v-\vb|^{3+\gamma+\delta}\d v\d\vb,$$
that is,
$$\mathcal{D}_\la^3 \leq \omega_\la(R)\|f_2\|_{L^1}\|g_2\|_{L^1} + 2^{\frac{3+\gamma}{2}}\frac{K_\gamma+C_\gamma}{R^{\delta}} \|g_2\|_{L^1_{3+\gamma+\delta}}\|f_2\|_{L^1_{3+\gamma+\delta}}.$$
Taking first $R >0$ large enough and then $\la$ small enough we get the conclusion.
\end{proof}
\begin{lemmeA}\label{lem:quantit} Assume that  there exist two positive constants $\mathfrak{a},\mathfrak{b} >0$ and two exponents  $\overline{\gamma}  >\gamma >0$ such that
\begin{equation*}
\left|e(r)- 1+\mathfrak{a}\,r^\gamma\right| \leq \mathfrak{b}\,r^{\overline{\gamma}} \text{ for any } r \geq  0.\end{equation*}
Then, there exist two explicit positive constant $A_\gamma, B_\gamma >0$ such that
\begin{multline}\label{dissTem}
\left|\mathcal{I}_\la(f_1,g_1)-\mathcal{I}_0(f_2,g_2)\right|\leq A_\gamma\left(\|f_1-f_2\|_{L^1_{3+\gamma}}\|g_1\|_{L^1_{3+\gamma}} +\|g_1-g_2\|_{L^1_{3+\gamma}}\|f_2\|_{L^1_{3+\gamma}}\right)\\
 + B_\gamma\,\la^\alpha \|f_2\|_{L^1_{3+\gamma+\overline{\gamma}}}\|g_2\|_{L^1_{3+\gamma+\overline{\gamma}}} \qquad \forall \la \in (0,1)
\end{multline}
where $\alpha=\min(\gamma,\overline{\gamma}-\gamma)$.
\end{lemmeA}
\begin{proof} For any $\la \in (0,1]$ and $r >0$
$$\zeta_\la(r^2)-\zeta_0(r^2)=\dfrac{r^{3+\gamma}}{2}\int_0^1 \left(\frac{1-e^2(\la\,r\,z)}{(\la\,r\,z)^\gamma} -2\mathfrak{a}\right)z^{3+\gamma}\d z.$$
Then, under our assumption on $e(\cdot)$, there are three constants $A, B, C >0$ such that
$$\left|\zeta_\la(r^2)-\zeta_0(r^2)\right| \leq A\la^{\overline{\gamma}-\gamma} r^{3+\gamma}+B\la^\gamma\,r^{3+2\gamma}+ C \la^{\overline{\gamma}} r^{3+\gamma+\overline{\gamma}}\qquad \forall \la > 0, r >0.$$
In other words, there is $C_\gamma >0$ such that
$$\bigg|\zeta_\la(|v-\vb|^2)-\zeta_0(|v-\vb|^2)\bigg| \leq C_\gamma \la^\alpha \langle  v \rangle^{3+\gamma+\overline{\gamma}}\langle \vb \rangle^{3+\gamma+\overline{\gamma}} \quad \forall v,\vb \in \R^3\times\R^3$$
 where $\alpha=\min(\gamma,\overline{\gamma}-\gamma)$. Consequently,
\begin{equation}\label{Dla3}\mathcal{D}_\la^3 \leq C_\gamma \la^\alpha \|f_2\|_{L^1_{3+\gamma+\overline{\gamma}}}\|g_2\|_{L^1_{3+\gamma+\overline{\gamma}}}.
\end{equation}
With this estimate the proof follows as the proof of Lemma A.\ref{lem:dissTem}.
\end{proof}
\begin{nb} For visco-elastic hard-spheres, the assumption \eqref{order2} is met with $\gamma=\frac{1}{5}$ and $\overline{\gamma}=\frac{2}{5}$. In particular,
$\alpha=\frac{1}{5}.$
 \end{nb}
For a given $a \geq 0$ define the exponential weight $m_a(v)=\exp(a|v|)$ and introduce
$$\mathcal{X}=L^1(m_a) \quad \text{ and } \quad \mathcal{Y}=L^1_1(m_a).$$
Define
$$\widetilde{\mathcal{X}}=\mathcal{X} \cap  \left\{f\,:\,\IR f(v)\d v = \IR f(v)v\d v =0\right\}\ \ \text{and}\ \
\widehat{\mathcal{X}}=\widetilde{\mathcal{X}} \cap \left\{f\,:\,\IR f(v)|v|^2\d v=0\right\},$$
and the operator
$$\mathcal{A}\::\:h \in \widetilde{\mathcal{X}} \mapsto \mathcal{A}h=\left(\mathcal{A}_1 h;\mathcal{A}_2 h\right) \in \mathbb{R} \times \widehat{\mathcal{X}}$$
where
$$\mathcal{A}_1 h=2\mathcal{I}_0(\mathcal{M},h) \qquad \text{and} \qquad \mathcal{A}_2h=\Q_1(h,\mathcal{M})+\Q_1(\mathcal{M},h)=\mathscr{L}_1h.$$
The operator $\mathcal{A}$ is a suitable lifting operator of $\mathscr{L}_1$.
\begin{lemmeA}\label{Ainvert}
The linear functional
$$\mathcal{A}\::\widetilde{\mathcal{Y}} \longrightarrow \mathbb{R} \times \widehat{\mathcal{X}}$$
is invertible and the norm $\|\mathcal{A}^{-1}\|=\|\mathcal{A}^{-1}\|_{\mathbb{R} \times \widehat{\mathcal{X}} \to \widetilde{\mathcal{Y}}}$ can be estimated explicitly.
\end{lemmeA}
\begin{proof} The fact that the mapping $\mathcal{A}_2=\mathscr{L}_1\::\: \widetilde{\mathcal{X}} \longrightarrow \widehat{\mathcal{X}}$ is invertible with explicit inverse is a direct consequence of Proposition \ref{spect} (see \cite{Mo} for details).  Set
\begin{multline*}
\wp_\gamma:=\IRR \left(|\vb|^2-3\Theta\right)\mathcal{M}(v)\mathcal{M}(\vb)\zeta_0(|v-\vb|^2)\d v\d\vb\\
=C_\gamma\IRR \left(|\vb|^2-3\Theta\right)\mathcal{M}(v)\mathcal{M}(\vb)|v-\vb|^{3+\gamma}\d v\d\vb.\end{multline*} Direct inspection shows that $\wp_\gamma \neq 0$ for any $\gamma >0$. Arguing as in \cite[Lemma 4.3]{MiMo3},  we deduce that, for any $y \in \mathbb{R},$ $g \in \widehat{\mathcal{X}}$ the unique solution to the equation
$\mathcal{A}h=(y,g)$ is given by $h=h_1 \, \varphi_1  + h^\bot$ with
$$h^\bot=\mathscr{L}_1^{-1}g,\qquad h_1=\dfrac{1}{2\wp_\gamma}\left(y-\mathcal{A}_1 h^\bot\right).$$
This proves the Lemma.
\end{proof}
\section*{Appendix B: Existence of a steady solution for  diffusively driven granular gases}\label{CauchyPB}
\setcounter{equation}{0}
\renewcommand{\theequation}{B.\arabic{equation}}
The main objective of this section is to prove Theorem \ref{theo:exists}, that is, to prove the existence of an steady solution $F$ to \eqref{steady}. The proof, see \cite{GaPaVi}, is based on a dynamic version of Tykhonov fixed point theorem and it is achieved by controlling the $L^2$-norm, the moments and the regularity of the solution to the time-dependent problem associated to \eqref{steady}.  Consider the diffusively driven Boltzmann equation
\begin{align}\label{driven}
\partial_t f(t,v)&=\Q_{e}(f,f)(t,v) + \mu \Delta f(v,t) \quad  t >0, \; v \in \mathbb{R}^3\nonumber\\
f(0,v)&=f_0(v) \quad v \in \mathbb{R}^3,
\end{align}
with $\mu >0$ and where the initial datum $f_0$ is a nonnegative velocity distribution satisfying
\begin{equation}\label{initial}\IR f_0(v)\d v=1, \quad \IR f_0(v) v \d v =0 \quad \text{ and } \quad \IR f_0(v)|v|^3 \d v < \infty.\end{equation}
Notice that if $E_f(t)$ denotes the kinetic energy of $f(t,v)$ at time $t \geq 0$, that is, $E_f(t)=\IR f(t,v)|v|^2\d v$
then it satisfies
$$\dfrac{\d }{\d t}E_f(t)=-\mathcal{I}_e(f(t))+6\mu$$
where $\mathcal{I}_e$ is the energy dissipation functional defined by \eqref{dissIeF} (justifying, a posteriori, the terminology we used in the core of the paper).  Problem \eqref{driven} is well posed due to the following theorem.
\begin{theoB}\label{exist} Assume the restitution coefficient $e(\cdot)$ satisfies Assumption \ref{HYP} and the initial datum $f_0 \in L^1(\R^3) \cap L\log L(\R^3)$ satisfies  \eqref{initial}. Then, there exists a unique nonnegative
weak solution
$$f \in L^\infty([0,\infty),L^1_2(\R^3)), \qquad f\log f \in L^\infty([0,\infty),L^1(\R^3))$$
to equation \eqref{driven}, with the initial condition $f (\cdot, 0) = f_0$. Furthermore, if in addition $f_0 \in L^1_2 \cap L^2(\R^3)$ then $f \in \mathcal{C}_b^\infty([t_0,\infty), \mathcal{S}(\R^3))$ for every $t_0 > 0$.
\end{theoB}
The proof of Theorem \ref{theo:exists} can be deduced from Theorem B.\ref{exist} following the proof of \cite[Theorem 5.2]{GaPaVi}, thus, we shall only recall the main steps in the proof of Theorem B.\ref{exist}.  The proof will follow the path presented in \cite[Theorem 5.1]{GaPaVi} with the differences clearly explained.
\subsection*{B.1 Povzner-type inequalities}\label{sec:povzner}  We derived in \cite{AloLo1} Povzner's estimates in the spirit of  \cite{BoGaPa} and \cite{MiMo2}. We shall extend this result, using some ideas of \cite{GaPaVi}. Recall that for any \textit{nonnegative} function $f$ and text function $\psi(v)=\Psi(|v|^2)$ with $\Psi$ nondecreasing and convex
$$
\int_{\mathbb{R}^{3}}\Q_{B,e}(f,f)(v)\psi(v)\d v=\frac{1}{2}\int_{\mathbb{R}^{3} \times \mathbb{R}^3}f(v)f(\vb)\Phi(|u|) \mathcal{A}_{B,e}[\Psi](v,\vb)\;\d\vb\d v,
$$
where
\begin{equation*}\begin{split}\mathcal{A}_{B,e}[\Psi](v,\vb)&=\int_{\mathbb{S}^{2}}\bigg( \Psi( |{v'}|^2)+\Psi( |{\vb'}|^2) - \Psi( |v |^2)-\Psi( |\vb|^2) \bigg)b(\widehat{u}\cdot\sigma)\d\sigma\\
&=\mathcal{A}^+_{B,e}[\Psi](v,\vb)-\left(\Psi( |v |^2)+\Psi( |\vb|^2)\right).
\end{split}\end{equation*}
The collision cross-section $B(u,\sigma)$ is given by \eqref{Bu} with the normalization assumption \eqref{normalization} (notation is slightly changed with respect to \cite{AloLo1}). Define  the velocity of the center of mass $U=\dfrac{v+\vb}{2}$ so that
$$v'=U+\dfrac{|u|}{2}\omega \ \ \text{and}\ \ \vb'=U-\dfrac{|u|}{2}\omega$$
with $\omega=(1-\beta)\widehat{u}+\beta \sigma$.  We proved in \cite[Eq. (2.15)]{AloLo1} that the post-collisional integral can be estimated from above as follows
\begin{equation*}\label{Ae+}
\mathcal{A}^+_{B,e}[\Psi](v,\vb)\leq \int_{\{\widehat{U}\cdot\sigma\geq0\}} \left[\Psi\left(E \frac{3+\widehat{U}\cdot\sigma}{4}\right)+\Psi\left(E \frac{1-\widehat{U}\cdot\sigma}{4} \right)\right]\tilde{b}(\widehat{u}\cdot\sigma)\d\sigma.\end{equation*}
Under assumption \eqref{crois} one can prove, see \cite[Lemma 1]{BoGaPa}, that this integral (involving $\widehat{U}$ and $\widehat{u}$) takes its maximum value whenever $\widehat{U}=\widehat{u}$, that is,
\begin{equation}\label{Ae+}
\mathcal{A}^+_{B,e}[\Psi](v,\vb)\leq 2\pi\int_0^1 \left[\Psi\left(E \frac{3+s}{4}\right)+\Psi\left(E \frac{1-s}{4} \right)\right]\tilde{b}(s)\d s
\end{equation}
where $\tilde{b}(s)=b(s)+b(-s)$ and $E=|v|^2+|\vb|^2.$  At this point, we shall adopt the viewpoint of \cite{GaPaVi} and assume that $\Psi$ satisfies the following conditions:
\begin{subequations}\label{cond:psi}
\begin{eqnarray}
\label{cond:psi-1} & \Psi(x)\ge0, \quad x>0; \quad \Psi(0) = 0;
\\ \label{cond:psi-2} & \Psi  \;\text{is convex,}\;\; \Psi'' \in L^\infty_{\textrm{loc}}((0,\infty));
\\ \label{cond:psi-3} & \Psi'(ax) \le \eta_1(a)\,
\Psi'({x})\quad  \text{ and } \quad \Psi''(ax) \le \eta_2(a)\, \Psi''(x),
\quad x>0\quad a>1,
\end{eqnarray}
\end{subequations}
where \(\eta_1(\cdot)\) and \(\eta_2(\cdot)\) are locally bounded functions. Then, one has the following generalization of \cite[Lemma 3.3]{GaPaVi} to non constant restitution coefficient.
\begin{propoB}\label{Povzner} Assume that \(\Psi(x)\) satisfies
\eqref{cond:psi}. Then, for any $(v,\vb) \in \R^3 \times \R^3$,
$$\mathcal{K}_{B,e}[\Psi](v,\vb)  \leq A \left(|v|^2 \Psi'(|\vb|^2)+ |\vb|^2 \Psi'(|v|^2)\right) -k\left(|v|^2+|\vb|^2\right)^2 \Psi''(|v|^2+|\vb|^2)$$
where \(A=\eta_1(2)\) while $k >0$ is an explicit constant depending only on $\eta_2$ and on $b(\cdot)$. For instance in the hard-sphere case $b(\cdot)=\frac{1}{4\pi}$, then $k\;\eta_2(2)=\tfrac{5}{96}$.
\end{propoB}
\begin{proof} Recall that, see \cite[Lemma 3.1]{GaPaVi}, if $\Psi$ satisfies \eqref{cond:psi} then
\begin{equation}
\label{by_above}
\Psi(x+y) - \Psi(x) - \Psi(y)
\le A\, (\,x\,\Psi'(y) + y\,\Psi'(x)\,)
\end{equation}
and
\begin{equation}
\label{by_below}
\Psi(x+y) - \Psi(x) - \Psi(y) \ge a_0\, xy \,\Psi''(x+y),
\end{equation}
where \(A=\eta_1(2)\) and \(a_0=(2\eta_2(2))^{-1}\). Let $(v,\vb)$ be fixed and write $\mathcal{A}_{B,e}[\Psi](v,\vb)=\mathcal{P}[\Psi](v,\vb)-\mathcal{N}[\Psi](v,\vb)$ where
$$\mathcal{P}[\Psi](v,\vb)=\int_{\mathbb{S}^{2}} \bigg(\Psi(|v|^2+|\vb|^2)-\Psi(|v|^2)-\Psi(|\vb|^2)\bigg)b(\widehat{u}\cdot\sigma)\d\sigma$$
and
$$\mathcal{N}[\Psi](v,\vb)=\int_{\mathbb{S}^{2}} \bigg(\Psi(|v|^2+|\vb|^2)-\Psi(|v'|^2)-\Psi(|\vb'|^2)\bigg)b(\widehat{u}\cdot\sigma)\d\sigma.$$
Using \eqref{by_above} and the normalization assumption \eqref{normalization} one gets directly that
$$\mathcal{P}[\Psi](v,\vb)\leq A \left(|v|^2 \Psi'(|\vb|^2)+ |\vb|^2 \Psi'(|v|^2)\right).$$
Let us extimate $\mathcal{N}[\Psi](v,\vb)$ from below.  First, one notices that
$$\mathcal{N}[\Psi](v,\vb)= \Psi(|v|^2+|\vb|^2)-\mathcal{A}_{B,e}^+[\Psi](v,\vb)$$ and deduces from \eqref{Ae+} that
$$\mathcal{N}[\Psi](v,\vb)\geq \Psi(E) -2\pi\int_0^1 \left[\Psi\left(E \frac{3+s}{4}\right)+\Psi\left(E \frac{1-s}{4} \right)\right]\tilde{b}(s)\d s, \quad E=|v|^2+|\vb|^2.$$
Second, since $\int_0^1 \tilde{b}(s)\d s=\frac{1}{2\pi}$ according to \eqref{normalization} one can write
$$\mathcal{N}[\Psi](v,\vb)\geq 2\pi\int_0^1 \left[\Psi(E)-\Psi\left(E \frac{3+s}{4}\right)+\Psi\left(E \frac{1-s}{4} \right)\right]\tilde{b}(s)\d s.$$
Noticing that $E=E \frac{3+s}{4}+E \frac{1-s}{4}$ for any $s \in (0,1)$, it is possible to apply directly \eqref{by_below} to obtain
$$\mathcal{N}[\Psi](v,\vb) \geq \frac{\pi a_0}{8} E^2 \Psi''(E)\int_0^1  (3+s)(1-s) \tilde{b}(s)\d s.$$
Setting $k=\frac{\pi a_0}{8}\int_0^1  (3+s)(1-s) \tilde{b}(s)\d s$, the desired conclusion follows.
\end{proof}
\begin{nbB} Note that the above estimate does not depend on the restitution coefficient $e(\cdot)$.  Indeed, the two constants $A$ and $b$ are depending only on $\Psi$ and the angular cross-section $b(\cdot)$ but not on $e(\cdot)$.
\end{nbB}
With Proposition B.\ref{Povzner} the following properties are derived exactly as shown in \cite{GaPaVi}.
\begin{lemmeB}\label{momentsp} Let $p > 1$ and $\Psi(x)=x^p$. Then, for $b(\cdot)=1/4\pi$ one has
$$|v-\vb|\mathcal{K}_{ e}[\Psi](v,\vb) \leq  -k_p\left(|v|^{2p+1} + |\vb|^{2p+1}\right) + A_p\left(|v|\,|\vb|^{2p} + |v|^{2p}\,|\vb|\right) \quad \forall (v,\vb) \in \R^6$$
where the constants $k_p$ and $A_p$ are independent on the restitution coefficient $e(\cdot)$. As a consequence, for any nonnegative distribution $f=f(v) \geq 0$,
\begin{multline*}
\IR \Q_e(f,f)(v)\,|v|^{2p}\d v \leq - k_p \left(\IR f(v)\d v\right)\left(\IR f(v)|v|^{2p+1}\d v\right)\\ +A_p\left(\IR f(v)|v|\d v\right)\left(\IR f(v)|v|^{2p}\d v\right).\end{multline*}
\end{lemmeB}

\subsection*{B.2 Proof of Theorem B. \ref{exist}}  The proof is a modification of \cite[Theorem 5.2]{GaPaVi} and we only give a sketch of it explaining where the original argument has to be modified to handle the non-constant restitution coefficient.  Using our Povzner's estimates, the propagation and appearance of moments given in \cite[Lemma 3.5]{GaPaVi} follow.  Additionally, using the control of  $\|\Q_{B,e}(f,f)\|_{L^p}$ derived in Corollary A. \ref{quadratic-estim}, we can easily adapt the proof of \cite[Lemma 4.7]{GaPaVi} to our case, yielding a local in time propagation of $\mathbf{H}^1(\R^3)$ norms.  Therefore, the \textit{a priori} estimates for the solution to \eqref{driven} derived for the constant restitution case extends.

Let us first deal with a smooth initial datum $f_0$ with compact support. For any truncation parameters $M > 1> m >0$, define then
$$\Phi_{m,M}(|u|)=m+\min\left(|u|,M\right)$$
and set $B_{m,M}(u,\sigma)=\frac{1}{4\pi}\Phi_{m,M}(|u|)$, $u \in \R^3,\,\sigma \in \S$. Define the collision operator $\Q_{m,M}=\Q_{B_{m,M},e}$ (using the notations of equation \eqref{Ie3B}). For any $T >0$, let $g=g(t,v) \in L^\infty([0,T]\,;\,L^1_2(\R^3) \cap L^2(\R^3))$ be a nonnegative function with
$$\IR g(t,v)\d v=1\qquad\text{ and } \qquad \IR g(t,v)\d v=0 \qquad \forall t \in [0,T].$$
Consider the auxiliary problem
\begin{equation}\label{drivenmM}\begin{cases}
&\partial_t f(t,v)-\mu\Delta_v f(t,v)+M f(t,v) =\Q_{m,M}(g,g)(t,v) + M  g(t,v)\qquad  t \in [0,T]\,, \; v \in \mathbb{R}^3\\
&f(0,v)=f_0(v).
\end{cases}
\end{equation} Setting $h=\Q_{m,M}(g,g)(t,v) + M g(t,v)$, one checks, see \cite[Theorem 5.2]{GaPaVi}, that $h \in L^\infty([0,T]\,;\,L^1_2(\R^3) \cap L^2(\R^3))$ and $h \geq -g(g \ast \Phi_{m,M}) + M  g \geq 0$.  The unique solution $f \in L^\infty([0,T]\,;\,L^1_2(\R^3) \cap L^2(\R^3))$ to \eqref{drivenmM} can be given explicitly and by a classical
parabolic regularity result
\begin{equation}
\label{fp:H2}
\|f\|_{\mathbf{H}^2([0,T]\times\R^3)} \le C_M
(\|h\|_{L^2([0,T]\times\R^3)} + \|f_0\|_{\mathbf{H}^1(\R^3)}).
\end{equation} Denoting by $\mathcal{T}$ the operator that maps $g$ into $f$, the core of the proof consists in showing that for a certain choice of constants $A_1$ and $A_2$, the operator $\mathcal{T}$ maps $\mathcal{B}$ into itself.  Here we refer to the set,
\begin{equation}\label{B}
\begin{split}
\mathcal{B}&=\bigg\{ f \in L^1([0,T] \times \R^3)\,:\,f \geq 0\,,\,\IR f(t,v)\d v=1\,,\,\IR f(t,v)v\d v=0,\\
&E_f(t):=\IR f(t,v) |v|^2\d v \leq A_1\,,\,\IR f^2(t,v)\d v \leq A_2 \quad \text{ for a.e. } t \in [0,T]\bigg\}.
\end{split}
\end{equation}
The first three properties are clearly satisfied. To determine $A_1$, one multiplies equation \eqref{drivenmM} by $|v|^2$ and integrate by parts. This yields
\begin{equation}\label{Eft}
\begin{split}
\dfrac{\d}{\d t} &E_f(t) + M  E_f(t) \leq 6\mu+M  E_g(t) + \IR \Q_{m,M}(g,g)(t,v)|v|^2\d v\\
&\leq 6\mu +M^2E_g(t) - \IRR g(t,v)g(t,v_*)\Phi_{m,M}(|u |)\dfrac{\mathbf{\Psi}_e(|u|^2)}{|u|}\d v\d v_*,
\end{split}
\end{equation}
where we used \eqref{PhiPsi} in the last identity.  Since $\Psi_e(r)$ may be arbitrarily small for small $r >0$, the argument changes slightly with respect to \cite{GaPaVi}.  Using that $\limsup_{r \to \infty}e(r)=e_0 < 1$, there exists some $R_0 \gg 1$ and some constant $C >0$ such that
$$\mathbf{\Psi}_e(|u|^2) \geq C|u|^3 \qquad \forall |u| > R_0.$$
Therefore,
\begin{equation*}\begin{split}
\IRR g(t,v)g(t,v_*)\Phi_{m,M}(|u|)\dfrac{\mathbf{\Psi}_e(|u|^2)}{|u|}\d v\d v_* &\geq C\int_{|u| \geq R_0} g(t,v)g(t,v_*)\Phi_{m,M}(|u |)|u|^2 \d v\d v_*\\
 &\geq Cm\int_{|u| \geq R_0} g(t,v)g(t,v_*)|u|^2 \d v\d v_*.\end{split}\end{equation*}
 Since $g$ has unit mass,
\begin{multline*}
\int_{|u| \geq R_0} g(t,v)g(t,v_*)|u|^2 \d v\d v_* =\IRR g(t,v)g(t,v_*)|u|^2 \d v\d v_*\\
-\int_{|u| < R_0} g(t,v)g(t,v_*)|u|^2 \d v\d v_* \geq 2E_g(t)-R_0^2.\end{multline*}
Going back to \eqref{Eft} finally leads to the estimate
\begin{equation}\label{eftegt}\dfrac{\d}{\d t} E_f(t) + M  E_f(t) \leq 6\mu+M  E_g(t) -  2CmE_g(t)+ CmR_0^2.\end{equation}
Setting $A_1'=\dfrac{6\mu+CmR_0^2}{2Cm}$ and assuming $E_g(t) \leq A_1'$ yields the differential inequality
$$\dfrac{\d}{\d t} E_f(t) + M  E_f(t) \leq M  A_1'$$
which, in turn, implies $E_f(t) \leq \max\left(A_1',E_f(0)\right)$ for any $t \in [0,T]$.  Thus, one may choose
$$A_1=\max\left(A_1',\IR f_0(v)|v|^2\d v\right)$$
in the definition \eqref{B} of $\mathcal{B}$. For the determination of the parameter $A_2 >0$ just follow the path of \cite[Theorem 5.2]{GaPaVi}. This leads to the existence of a solution $f=f_{m,M} \in L^\infty([0,T],L^1_2(\R^3) \cap L^2(\R^3))$  to the modified Boltzmann equation
\begin{equation*}\begin{cases}
\partial_t f(t,v) &=\Q_{m,M}(f,f)(t,v)+\mu \Delta_v f(t,v) \qquad  t \in [0,T]\,, \; v \in \mathbb{R}^3\\
f(0,v)&=f_0(v).
\end{cases}
\end{equation*}
It remains to pass to the limit as $M \to \infty$ and $m \to 0$. To this end, we will show that
the bounds found in the \textit{a priori} estimates hold for the fixed point solutions and are
uniform in $M$ and $m$. From \eqref{eftegt} with $f=f_{m,M}$
$$\dfrac{\d}{\d t}E_f(t) \leq 6\mu -  2CmE_f(t)+ CmR_0^2,$$
which yields
$$E_f(t) \leq A_1=\max\left(\dfrac{3\mu}{Cm}+\frac{R_0^2}{2},\IR f_0(v)|v|^2\d v\right)$$
which provides a bound independent of $M$. Using Proposition B.\ref{Povzner}, it is possible to adapt the proof of \cite[Theorem 5.2]{GaPaVi} to get that, for any $p \geq 1$ and $T >0$, the bounds of \(f=f_{m,M}\) in
\(L^\infty([0,T],L^1_{2p}(\R^3))\) are
independent of \(M\).  Since $f \in \mathbf{H}^2([0,T] \times \R^3)$, using the extension of \cite[Lemma 4.7]{GaPaVi} and then \cite[Lemmas 4.8 \& 4.9]{GaPaVi},
\[
f\in L^\infty([0,T],\mathbf{H}^n_{2p}(\R^3)),
\]
for every $n \geq 1$, and every \(p\ge0\), with bounds
independent on \(M\).  This allows to pass to the
limit as \(M\to\infty\) in the weak form and to show that
the limit solutions satisfy the equation with the kernel
\[
(m+|u|)\,b(u,\sigma).
\]
Following the argument of \cite{GaPaVi} it is possible to prove that the bounds in \(L^\infty([0,T],L^1_{2p}(\R^3))\)  are actually independent on \(m\) and \(T\). This allows to pass to the limit as \(m\to0\) and the limit solution
obtained is a solution to \eqref{driven}.  A standard approximation argument generalize the initial conditions from smooth ones.

\end{document}